\documentclass[12pt,letterpaper,reqno]{amsart}
\usepackage[T1]{fontenc}
\usepackage[utf8]{inputenc}
\usepackage{lmodern}
\usepackage[margin=0.8in,headheight=14pt,headsep=15pt,footskip=25pt]{geometry}
\usepackage{amsmath,amssymb,amsfonts,mathtools}
\usepackage{array,booktabs,longtable,tabularx,multirow}
\usepackage{enumitem}
\usepackage{microtype}
\usepackage{graphicx}
\usepackage{xcolor}
\usepackage{cite}
\usepackage[hidelinks,pdfusetitle,hypertexnames=false]{hyperref}
\usepackage{url}
\usepackage{bm}

\DeclareMathOperator{\Aut}{Aut}
\DeclareMathOperator{\core}{core}
\DeclareMathOperator{\ord}{ord}
\DeclareMathOperator{\meo}{meo}
\DeclareMathOperator{\supp}{supp}
\DeclareMathOperator{\lcm}{lcm}
\DeclareMathOperator{\Fix}{Fix}
\DeclareMathOperator{\diag}{diag}
\DeclareMathOperator{\tr}{tr}
\DeclareMathOperator{\Tr}{Tr}
\DeclareMathOperator{\GL}{GL}
\DeclareMathOperator{\SL}{SL}
\DeclareMathOperator{\PGL}{PGL}
\DeclareMathOperator{\PSL}{PSL}
\DeclareMathOperator{\PSU}{PSU}
\DeclareMathOperator{\UT}{UT}
\DeclareMathOperator{\SD}{SD}
\newcommand{\PSp}{\mathrm{PSp}}
\newcommand{\POmega}{\mathrm{P\Omega}}

\newtheorem{theorem}{Theorem}[section]
\newtheorem{lemma}[theorem]{Lemma}
\newtheorem{proposition}[theorem]{Proposition}
\newtheorem{corollary}[theorem]{Corollary}
\newtheorem{conjecture}[theorem]{Conjecture}
\theoremstyle{definition}
\newtheorem{definition}[theorem]{Definition}

\newcolumntype{L}[1]{>{\raggedright\arraybackslash}p{#1}}
\setlist[itemize,1]{leftmargin=2.4em,itemsep=3pt,topsep=5pt}
\allowdisplaybreaks[1]

\numberwithin{equation}{section}
\title[The EGZ theorem for finite groups]{The Erd\H{o}s--Ginzburg--Ziv theorem for finite groups}
\author{Yongke Qu}
\address{Department of Mathematics, Luoyang Normal University, Luoyang, 471934, P.R. China}
\email{yongke1239@163.com}
\date{}
\begin{document}
\begin{abstract}
Let $G$ be a finite group, and let $p$ be the smallest prime divisor of $|G|$. The Gao constant $\mathsf{E}(G)$ is the least integer $\ell$ such that every sequence over $G$ of length $\ell$ contains a subsequence of length $|G|$ whose terms can be ordered to have product equal to the identity. In 1961, Erd\H{o}s, Ginzburg and Ziv proved that every sequence of $2n-1$ integers has an $n$-term subsequence whose sum is divisible by $n$; equivalently, $\mathsf{E}(C_n)\leq 2n-1$, where $C_n$ denotes the cyclic group of order $n$. In this paper, we prove that $\mathsf{E}(G)\leq |G|+|G|/p-1$ whenever $G$ contains no cyclic subgroup of index $p$. Consequently, every finite non-cyclic group satisfies $\mathsf{E}(G)\leq |G|+|G|/p+p-2$, with equality if and only if $G$ contains a cyclic subgroup of index $p$. This establishes the conjectured upper bound for the Gao constant and provides a complete characterization of the equality case.
\end{abstract}
\maketitle

\section{Introduction}\label{sec:intro}
Let $G$ be a finite group written multiplicatively. A sequence over $G$ is a finite multiset of elements of $G$. Such a sequence is called product-one when its terms can be ordered so that their product is the identity. The Gao constant $\mathsf{E}(G)$ is the least integer $\ell$ such that every sequence over $G$ of length $\ell$ has a product-one subsequence of length $|G|$. The small Davenport constant $\mathsf{d}(G)$ is the maximal length of a sequence with no nonempty product-one subsequence.

The Erd\H{o}s--Ginzburg--Ziv theorem~\cite{ErdosGinzburgZiv1961}, proved in 1961, asserts that every sequence of $2n-1$ integers has an $n$-term subsequence whose sum is divisible by $n$. The bound is sharp, as shown by the sequence consisting of $n-1$ zeros and $n-1$ ones. Consequently, $\mathsf{E}(C_n)=2n-1$, where $C_n$ denotes the cyclic group of order $n$. Together with the study of the Davenport constant, this theorem forms one of the foundations of zero-sum theory. Although much of the early work concerned abelian groups, product-one problems provide a natural extension to arbitrary finite groups.

In 1976, Olson~\cite{Olson1976} proved that $\mathsf{E}(G)\leq 2|G|-1$ for every finite group. Subsequent work established sharper bounds for non-cyclic solvable groups. In 1984, Yuster and Peterson~\cite{YusterPeterson1984} obtained the bound $2|G|-2$, and Yuster~\cite{Yuster1988} derived further improvements in 1988. Gao~\cite{Gao1996b} proved the bound $11|G|/6-1$ in 1996, and Gao and Li~\cite{GaoLi2010} subsequently strengthened it to $7|G|/4-1$ in 2010. The latter authors also proposed the following conjecture~\cite[Conjecture 3]{GaoLi2010}.

\begin{conjecture}\label{conj:gao-li}
Every finite non-cyclic group $G$ satisfies $\mathsf{E}(G)\leq 3|G|/2$.
\end{conjecture}

In 2015, Han~\cite{Han2015} proved that $\mathsf{E}(G)\leq |G|+\frac{|G|}{p}+p-2$ for every finite non-cyclic nilpotent group, where $p$ is the smallest prime divisor of $|G|$. This result establishes Conjecture~\ref{conj:gao-li} for nilpotent groups. He also proposed the following stronger conjecture~~\cite[Conjecture 3.1]{Han2015}.

\begin{conjecture}\label{conj:han}
Every finite non-cyclic group $G$ satisfies $\mathsf{E}(G)\leq |G|+\frac{|G|}{p}+p-2$, where $p$ is the smallest prime divisor of $|G|$.
\end{conjecture}

Further cases of Conjecture~\ref{conj:gao-li} have since been established. In 2021, Gao, Li and Qu~\cite{GaoLiQu2021} established it for non-cyclic groups of odd order. Recently, in 2026, Godara, Joshi, and Mazumdar~\cite{GodaraJoshiMazumdar2026} verified it for non-cyclic groups of the form $A\rtimes_{-1}C_2$, where $A$ is a finite abelian $p$-group and the nonidentity element of $C_2$ acts on $A$ by inversion. In 2026, Zhao and Wang~\cite{ZhaoWang2026} proved it for non-cyclic groups whose orders are not divisible by four and characterized the equality case.

The problem of determining $\mathsf d(G)$ for a finite abelian group $G$ was proposed by Davenport in 1966, motivated by connections between $\mathsf d(G)+1$ and algebraic number theory. This problem has received considerable attention (see \cite{GirardSchmid2019,Liu2020} for example). Nevertheless, exact values are known only for certain classes of finite abelian groups. In 1996, Gao \cite{Gao1996a} proved that $\mathsf E(G)=\mathsf d(G)+|G|$ for every finite abelian group. For non-abelian groups, the Davenport constant also has close connections with the Noether number (see \cite{CziszterDomokos2014,CziszterDomokosGeroldinger2016,CziszterDomokosSzollosi2018,HanZhang2019,Schmid1991} for further details). The extension of this identity to non-abelian groups has also attracted considerable interest. In 2005, Zhuang and Gao \cite{ZhuangGao2005} established it for dihedral groups of order $2p$ where $p\geq 4001$ is a prime, and proposed the following conjecture~\cite[Conjecture~2]{ZhuangGao2005}.

\begin{conjecture}\label{conj:zhuang-gao}
Every finite group $G$ satisfies $\mathsf{E}(G)=|G|+\mathsf{d}(G)$.
\end{conjecture}

Subsequent research has devoted considerable attention to finite non-abelian groups containing a cyclic subgroup of index two. In 2008, Gao and Lu \cite{GaoLu2008} extended the preceding result to dihedral groups of order $2n$ for all $n\geq 23$. Extending their method, Bass \cite{Bass2007} proved the conjecture for all dihedral groups and dicyclic groups, as well as for non-abelian groups of order $pq$, where $p$ and $q$ are primes. In 2023, Avelar, Brochero Mart\'{\i}nez and Ribas \cite{AvelarBrocheroRibas2023} verified it for non-abelian semidirect products $C_2\ltimes C_n$, apart from a single exceptional case, which was subsequently resolved by Oh, Ribas, Zhao and Zhong \cite{OhRibasZhaoZhong2026}. Yang, Zhang and Feng \cite{YangZhangFeng2024} in 2024 and Ribas \cite{Ribas2025} in 2025 established the conjecture for certain non-split metacyclic groups containing a cyclic subgroup of index two. More recently, Qu, Gao and Li~\cite{QuGaoLi2026} proved it for every finite non-cyclic group containing a cyclic subgroup whose index is the smallest prime divisor of the group order.

For the small Davenport constant, Qu, Li and Teeuwsen~\cite{QuLiTeeuwsen2025} established the optimal upper bound $\mathsf d(G)\leq |G|/p-1$ for every non-cyclic group $G$ containing no cyclic subgroup of index $p$, where $p$ is the smallest prime divisor of $|G|$. In this paper, we further investigate the Gao constant $\mathsf{E}(G)$ for these groups and establish the following main result.

\begin{theorem}\label{thm:main-bound}
Let $G$ be a finite non-cyclic group, and let $p$ be the smallest prime divisor of $|G|$. If $G$ has no cyclic subgroup of index $p$, then
$$ \mathsf{E}(G)\leq |G|+\frac{|G|}{p}-1. $$
In particular, if $G$ is non-solvable, then $\mathsf{E}(G)\leq 3|G|/2-1$.
\end{theorem}
Combining Theorem~\ref{thm:main-bound} with \cite[Theorem~1.5]{QuGaoLi2026} yields the following corollary.
\begin{corollary}\label{cor:prime-index-equality}
Let $G$ be a finite non-cyclic group, and let $p$ be the smallest prime divisor of $|G|$. Then
$$ \mathsf{E}(G)\leq |G|+\frac{|G|}{p}+p-2, $$
with equality if and only if $G$ has a cyclic subgroup of index $p$. Moreover, 
$$\mathsf{E}(G)\leq 3|G|/2,$$ 
with equality if and only if $G$ has a cyclic subgroup of index $2$. Thus, Conjectures~\ref{conj:gao-li} and \ref{conj:han} both hold, with equality cases described above.
\end{corollary}

\subsection*{Proof strategy of Theorem~\ref{thm:main-bound}}

The proof mainly consists of the following two aspects.
\begin{itemize}
\item \emph{Solvable groups.} We first establish the bound for nilpotent groups. Theorem~\ref{thm:solvable-configs} then shows that a solvable counterexample of minimal order must fall into one of five configurations. The estimates ruling out these configurations are established in Sections~\ref{sec:a4-extensions}--\ref{sec:prime-square-quotients} and summarized in Section~\ref{sec:solvable-completion}.
\item \emph{Non-solvable groups.} Proposition~\ref{prop:simple-reduction} reduces the problem to non-abelian simple groups. For these groups, Theorem~\ref{thm:bounded-height} handles sequences of bounded height, while Lemma~\ref{lem:high-mult-completion} provides a method for treating the remaining sequences. The subsequent arguments are organized according to the product constructions they employ, and a final proposition verifies that all relevant groups and element orders are covered.
\end{itemize}

Section~\ref{sec:preliminaries} collects the preliminary results, and Section~\ref{sec:simple-reduction} presents the reduction for non-solvable groups. Sections~\ref{sec:solvable-reduction}--\ref{sec:prime-square-quotients} establish the result for solvable groups, while Sections~\ref{sec:growth-compatibility}--\ref{sec:high-multiplicity} develop the bounded-height argument and the common completion criteria. The treatment of sequences containing a term of high multiplicity is organized into four parts: subgroup and compatible extraction in Section~\ref{sec:subgroup-extraction}, short product-one subsequences in Section~\ref{sec:short-subsequences}, growth of conjugation orbits in Section~\ref{sec:orbit-growth}, and cyclic seeds in Section~\ref{sec:cyclic-seeds}. Numerical applications of each method are summarized in tables. Section~\ref{sec:structural-coverage} verifies that all relevant cases are covered and completes the proof for simple groups. Finally, Section~\ref{sec:main-results} combines these results with the argument for solvable groups to complete the proof of Theorem~\ref{thm:main-bound}.

\section{Notation and preliminary results}\label{sec:preliminaries}

Throughout the paper, all groups are finite. The trivial group is denoted by $1$, and $C_m$ denotes the cyclic group of order $m$. A finite abelian group of rank at most two is written as $C_{n_1}\times C_{n_2}$, where $n_1$ and $n_2$ are positive integers satisfying $n_1\mid n_2$; the case $n_1=1$ includes cyclic groups. We use multiplicative notation unless an abelian group is explicitly written additively.

The free abelian monoid on $G$ is denoted by $\mathcal{F}(G)$. The multiplication in $\mathcal{F}(G)$ is denoted by the bold dot symbol $\bm\cdot$ rather than by juxtaposition. Its elements are finite sequences, or multisets, of elements of $G$. Let $S=g_1\bm\cdot\ldots\bm\cdot g_\ell$ be a sequence over $G$. We write $|S|=\ell$ for the length and $v_g(S)=|\{i:g_i=g\}|$ for the multiplicity of $g$. The support and height of a nonempty sequence are
$$ \supp(S)=\{g:v_g(S)>0\} \qquad\text{and}\qquad h(S)=\max_g v_g(S). $$
Brackets distinguish sequence powers from group powers: $g^{[a]}$ denotes a sequence of $a$ copies of $g$, whereas $g^a$ is a group power. A sequence $T \in \mathcal F(G)$ is called a subsequence of $S$ and is denoted by $T \mid S$ when $\mathsf v_g(T) \leq \mathsf v_g(S)$ for all $g\in G$. Denote by $S\bm\cdot T^ {[-1]}$ the subsequence of $S$ obtained by removing the terms of $T$ from $S$. For $X\subseteq G$, let $S_X$ be the subsequence of terms lying in $X$.

We define
$$ \pi(S)=\{g_{\tau(1)}\cdots g_{\tau(\ell)}:\tau \mbox{ a permutation of } [1, \ell]\} \qquad\text{and}\qquad \Pi_k(S)=\bigcup_{\substack{T\mid S, \ |T|=k}}\pi(T). $$
The empty sequence has product set $\{1\}$, so $\Pi_0(S)=\{1\}$. A sequence is product-one when $1\in\pi(S)$, and it is product-one free when none of its nonempty subsequences is product-one. The small Davenport constant $\mathsf{d}(G)$ is the maximal length of a product-one free sequence. Thus every sequence of length at least $\mathsf{d}(G)+1$ has a nonempty product-one subsequence.

Write $\exp(G)$ for the exponent of $G$ and $\meo(G)$ for its maximum element order. For positive integers $e$, $k$, let $\mathsf{\eta}_e(G)$, respectively $\mathsf{s}_k(G)$, be the least integer $L$ forcing a nonempty product-one subsequence of length at most $e$, respectively exactly $k$, in every sequence of length at least $L$, whenever this threshold exists. These constants exist precisely when $e\geq \meo(G)$, respectively $\exp(G)\mid k$: necessity follows from constant sequences, and sufficiency from sufficiently many equal terms. Thus $\mathsf{E}(G)=\mathsf{s}_{|G|}(G)$. In additive notation, we write $\sigma(S)$ for a sequence sum and $\Sigma_k(S)$ for its $k$-term subsums. A homomorphism of $G$ acts on sequences over $G$ term by term. We use $\exp(G)$ for the exponent, $G'$ for the commutator subgroup, $Z(G)$ for the centre, and $\Phi(G)$ for the Frattini subgroup. For $H\triangleleft G$, the quotient map is denoted by $\varphi_H:G\rightarrow G/H$.

\subsection{Classical sequence bounds and inverse results}\label{sec:sequence-bounds}

We first record the values and bounds for $\mathsf{d}(G)$ and $\mathsf{E}(G)$ used below.

\begin{lemma}\label{lem:davenport-bounds}
The following assertions hold.
\begin{itemize}
\item[\textup{(i)}] If $r\geq1$ and $p$ is prime, then
$$ \mathsf{d}(C_p^r)=r(p-1). $$
If $G\cong C_{n_1}\times C_{n_2}$ is a finite abelian group of rank at most two, where $n_1\mid n_2$, then
$$ \mathsf{d}(G)=n_1+n_2-2. $$
\item[\textup{(ii)}] Suppose that $G$ is non-cyclic and that $p$ is the smallest prime divisor of $|G|$. If $G$ has a cyclic subgroup of index $p$, then
$$ \mathsf{d}(G)=|G|/p+p-2. $$
\item[\textup{(iii)}] Suppose that $G$ is non-cyclic and that $p$ is the smallest prime divisor of $|G|$. If $G$ has no cyclic subgroup of index $p$, then
$$ \mathsf{d}(G)\leq |G|/p-1. $$
\end{itemize}
\end{lemma}
\begin{proof}
Part (i) follows from Olson's results~\cite{Olson1969a,Olson1969b}; see also~\cite[Theorem~3.1]{GaoGeroldinger2006}. Parts (ii) and (iii) are established by Qu, Li and Teeuwsen~\cite{QuLiTeeuwsen2025}.
\end{proof}

\begin{lemma}\label{lem:gao-identities}
The following assertions hold.
\begin{itemize}
\item[\textup{(i)}] For every finite abelian group $A$, one has $\mathsf{E}(A)=|A|+\mathsf{d}(A)$. In particular,
$$ \mathsf{E}(C_m)=2m-1 \qquad\text{and}\qquad \mathsf{E}(C_p^r)=p^r+r(p-1). $$
\item[\textup{(ii)}] If $G$ is non-cyclic, $p$ is the smallest prime divisor of $|G|$, and $G$ has a cyclic subgroup of index $p$, then
$$ \mathsf{E}(G)=|G|+\mathsf{d}(G)=|G|+|G|/p+p-2. $$
\item[\textup{(iii)}] The equality $\mathsf{E}(G)=|G|+\mathsf{d}(G)$ is also known for
$$ G=C_{uv}\rtimes C_u \qquad\text{and}\qquad G=D_{2v}\times C_2. $$
In the first family, $u,v\geq1$, and every action is allowed. In the second family, $v\geq4$ is even and $D_{2v}$ has order $2v$. The equality also holds for $G=\UT_3(\mathbb{F}_q)$ when $q$ is an odd prime.
\item[\textup{(iv)}] If $G$ satisfies the hypotheses of Lemma~\ref{lem:davenport-bounds}(iii) and the equality $\mathsf{E}(G)=|G|+\mathsf{d}(G)$ is known for $G$, then $\mathsf{E}(G)\leq |G|+|G|/p-1$.
\end{itemize}
\end{lemma}
\begin{proof}
Part (i) is Gao's theorem~\cite[Theorem~1]{Gao1996a}; its cyclic case is the Erd\H{o}s--Ginzburg--Ziv theorem~\cite{ErdosGinzburgZiv1961}. The elementary abelian value follows from Lemma~\ref{lem:davenport-bounds}(i). Part (ii) is the theorem for groups with a cyclic subgroup of prime index~\cite[Theorem~1.5]{QuGaoLi2026}. The assertions in (iii) follow, respectively, from~\cite[Theorem~1.2]{HanZhang2019}, \cite[Theorem~4.2]{BrocheroMartinezLemosMoriyaRibas2023}, and~\cite{QuWang2026}; for $q=3$, see also~\cite{GodaraSarkar2025}. Part (iv) follows by applying Lemma~\ref{lem:davenport-bounds}(iii) to the assumed equality.
\end{proof}

\begin{lemma}\label{lem:gao-bounds}
The following assertions hold.
\begin{itemize}
\item[\textup{(i)}] For every $G$,
$$ |G|+\mathsf{d}(G)\leq \mathsf{E}(G)\leq 2|G|-1. $$
\item[\textup{(ii)}] If $N\triangleleft G$ and $G/N\cong C_m\rtimes C_m$ for $m\geq2$, then
$$ \mathsf{E}(G)\leq |G|+|G|/m+m-2. $$
\item[\textup{(iii)}] If $G$ is non-cyclic and nilpotent and $p$ is the smallest prime divisor of $|G|$, then $\mathsf{E}(G)\leq |G|+|G|/p+p-2$.
\item[\textup{(iv)}] If $G\cong S_3\times C_3$ or $G\cong(C_3\times C_3)\rtimes_{-1}C_2$, then $\mathsf{E}(G)\leq26$.
\end{itemize}
\end{lemma}
\begin{proof}
The upper bound in (i) is Olson's theorem~\cite{Olson1976} and the lower bound follows from~\cite[Lemma 4]{ZhuangGao2005}. Parts (ii), (iii), and (iv) follow from~\cite[Theorem~1.2]{HanZhang2019}, \cite[Theorem~1.4]{Han2015}, and~\cite[Lemma~3.11]{ZhaoWang2026}, respectively.
\end{proof}

\begin{lemma}\label{lem:repeated-extraction}
For positive integers $j$ and $m$,
$$ \mathsf{s}_{jm}(C_m)=(j+1)m-1. $$
More generally, a sequence over $G$ of length at least $\mathsf{E}(G)+(j-1)|G|$ contains $j$ disjoint product-one subsequences, each of length $|G|$. Consequently, if $|S|=N+d$, where $N=|G|$, and $1\notin\Pi_N(S)$, then every subgroup $H\leq G$ satisfies
\begin{equation} |S_{G\setminus H}|\geq d+|H|-\mathsf{E}(H)+1. \tag{2.1}\label{eq:subgroup-concentration} \end{equation}
\end{lemma}
\begin{proof}
The first result follows from~\cite[Theorem~6.10(2)]{GaoGeroldinger2006}. If $|S_H|\geq N+\mathsf{E}(H)-|H|$, then this extraction in $H$ would yield $N/|H|$ disjoint $|H|$-term product-one subsequences. Their union would be an $N$-term product-one subsequence of $S$, contrary to the hypothesis. Hence $|S_H|\leq N+\mathsf{E}(H)-|H|-1$, which proves the last assertion.
\end{proof}

We use Reiher's theorem~\cite{Reiher2007} in the following form.

\begin{lemma}\label{lem:reiher}
For every integer $m\geq2$, one has $\mathsf{s}_m(C_m^2)=4m-3$.
\end{lemma}

The following lemma follows from \cite[Corollary~2.4]{QuGaoLi2026}.

\begin{lemma}\label{lem:cyclic-inverse}
Let $m\geq2$, let $j\geq1$, and let $W\in\mathcal{F}(C_m)$ have length $(j+1)m-2$. If $1\notin\Pi_{jm}(W)$, then
$$ W=u^{[j_1m-1]}\bm\cdot v^{[j_2m-1]}, $$
where $j_1,j_2\geq1$, $j_1+j_2=j+1$, and $\langle uv^{-1}\rangle=C_m$. Moreover, if $U\mid W$ contains $u^{[m-1]}\bm\cdot v^{[2m-1]}$, then
$$ \Pi_{|U|-m+1}(U)=C_m. $$
For $j=2$, the two multiplicities are $m-1$ and $2m-1$.
\end{lemma}

The following estimate is~\cite[Lemma~7]{GaoLu2008}.

\begin{lemma}\label{lem:abelian-products}
Let $K$ be abelian of order $a\geq2$, let $b\geq2$, and let $|W|=a+b-2$. If $1\notin\Pi_a(W)$, then $|\Pi_{a-2}(W)|\geq b-1$.
\end{lemma}

The following lemma follows from \cite[Lemma~2.9]{QuGaoLi2026}.

\begin{lemma}\label{lem:central-multiplicity}
Let $d\geq \mathsf{d}(G)$ be an integer, and let $|S|=|G|+d$. If $v_z(S)\geq d$ for some $z\in Z(G)$, then $1\in\Pi_{|G|}(S)$.
\end{lemma}

\subsection{Matching and lifting}\label{sec:matching-lifting}

We use the cyclic zero-sum matching theorem~\cite[Theorem~8]{BialostockiDierker1992}; see also~\cite{FrickEtAl2025}. The colouring in this theorem is an arbitrary map on the $e$-element subsets.

\begin{lemma}\label{lem:zero-sum-matching}
Let $m,e\geq1$, and $|X|=m(e+1)-1$. Every map $f:\binom{X}{e}\rightarrow C_m$ admits pairwise disjoint $e$-element subsets $X_1,\ldots,X_m$ such that
$$ f(X_1)\cdots f(X_m)=1. $$
\end{lemma}

\begin{lemma}\label{lem:lifting-bounds}
The following inequalities hold.
\begin{itemize}
\item[\textup{(i)}] For every subnormal subgroup $H\leq G$,
$$ \mathsf{E}(G)+1\leq |G:H|\bigl(\mathsf{E}(H)+1\bigr). $$
\item[\textup{(ii)}] For every cyclic normal subgroup $Z\triangleleft G$,
$$ \mathsf{E}(G)+1\leq |Z|\bigl(\mathsf{E}(G/Z)+1\bigr). $$
\end{itemize}
\end{lemma}
\begin{proof}
For (i), suppose first that $H\triangleleft G$, and let $q=|G:H|$. We prove that
\begin{equation} \mathsf{E}(G)\leq q\bigl(\mathsf{E}(H)-1\bigr)+\mathsf{E}(G/H). \tag{2.2}\label{eq:normal-subgroup-lift} \end{equation}
Consider a sequence whose length is the right-hand side of this inequality. By Lemma~\ref{lem:repeated-extraction} applied in $G/H$, there are $\mathsf{E}(H)$ pairwise disjoint subsequences $T_i$ of length $q$, each admitting an ordering whose product $h_i$ lies in $H$. The sequence $h_1\bm\cdot\ldots\bm\cdot h_{\mathsf{E}(H)}$ has a product-one subsequence of length $|H|$. Arrange the corresponding subsequences $T_i$ in the order of these products, keeping the chosen internal ordering of each $T_i$. Their concatenation has $q|H|=|G|$ terms and product one, proving~\eqref{eq:normal-subgroup-lift}. Since $\mathsf{E}(G/H)\leq2q-1$, the right-hand side of~\eqref{eq:normal-subgroup-lift} is at most $q(\mathsf{E}(H)+1)-1$.

If $H$ is only subnormal, then apply the normal-subgroup inequality along a chain
$$ H=H_0\triangleleft H_1\triangleleft\cdots\triangleleft H_s=G $$
and multiply the indices.

For (ii), let $m=|Z|$, $q=|G:Z|$, and $e=\mathsf{E}(G/Z)$. The case $m=1$ is immediate. Consider a sequence $S=g_1\bm\cdot\ldots\bm\cdot g_{m(e+1)-1}$, and let $X$ be its set of positions. For each $e$-element subset $Y\subseteq X$, the sequence on $Y$ contains a $q$-term subsequence whose image in $G/Z$ is product-one. Choose one such subsequence and one ordering, and let $f(Y)\in Z$ be the resulting product. These choices define a map $f:\binom{X}{e}\rightarrow Z\cong C_m$.

By Lemma~\ref{lem:zero-sum-matching}, there are disjoint subsets $Y_1,\ldots,Y_m$ such that $f(Y_1)\cdots f(Y_m)=1$. The chosen $q$-term subsequences inside these disjoint position sets are disjoint. Concatenating their chosen orderings yields $mq=|G|$ terms and product one. Hence $\mathsf{E}(G)\leq m(e+1)-1$.
\end{proof}

We normalize the invariant by defining
\begin{equation} \overline{\mathsf{E}}(G)=\frac{\mathsf{E}(G)+1}{|G|}. \tag{2.3}\label{eq:norm-gao} \end{equation}
By Lemma~\ref{lem:lifting-bounds},
\begin{equation} \overline{\mathsf{E}}(G)\leq\overline{\mathsf{E}}(H) \qquad\text{and}\qquad \overline{\mathsf{E}}(G)\leq\overline{\mathsf{E}}(G/Z). \tag{2.4}\label{eq:norm-lifting} \end{equation}
The first inequality holds when $H$ is subnormal in $G$, and the second holds when $Z$ is cyclic and normal in $G$. Thus the desired bound is $\overline{\mathsf{E}}(G)\leq1+1/p$, and the non-solvable bound is $\overline{\mathsf{E}}(G)\leq3/2$.

\subsection{Group algebras and finite-group structure}\label{sec:algebra-structure}

\begin{lemma}\label{lem:augmentation-bounds}
Let $P$ be a finite $p$-group of order $n$, and let $J$ be the augmentation ideal of $\mathbb{F}_p[P]$. Define $t(P)=\min\{b\geq1:J^b=0\}$. Then
$$ \mathsf{d}(P)+1\leq t(P) \qquad\text{and}\qquad \mathsf{E}(P)\leq n+t(P)-1. $$
If $P$ is non-cyclic and has no cyclic subgroup of index $p$, then $t(P)\leq n/p$.
\end{lemma}
\begin{proof}
Write $b=t(P)$. If $g_1\bm\cdot\ldots\bm\cdot g_b$ were product-one free, then the coefficient of $1$ in $\prod_i(1-g_i)$ would be $1$, because only the empty sequence contributes. This contradicts $J^b=0$ and proves that $\mathsf{d}(P)+1\leq b$.

The known bound $b\leq n$, with equality only for cyclic $P$, is recorded in~\cite[Section~2]{AltunbulakAksuGreen2016}. Let $c$ generate $C_n$, and let $u=c-1$. The augmentation ideal $I$ of
$$ \mathbb{F}_p[P\times C_n]\cong \mathbb{F}_p[P]\otimes_{\mathbb{F}_p}\mathbb{F}_p[u]/(u^n) $$
is generated by $J$ and the central element $u$. Every term in $I^{n+b-1}$ has at least $b$ factors from $J$ or $n$ factors $u$, so $I^{n+b-1}=0$. By the preceding coefficient argument,
$$ \mathsf{d}(P\times C_n)+1\leq n+b-1. $$
Lift each term $g$ of a sequence of length $n+b-1$ to $(g,c)$. A nonempty product-one subsequence has length divisible by $n$ and less than $2n$, so its length is exactly $n$. This proves the bound for $\mathsf{E}(P)$.

Finally, Koshitani's theorem~\cite[Theorem~1.6]{Koshitani1977} states that, for non-cyclic $P$, one has $n/p<b<n$ if and only if $P$ has a cyclic subgroup of index $p$. Excluding such a subgroup and the equality case for cyclic $P$ yields $b\leq n/p$.
\end{proof}

\begin{lemma}\label{lem:group-structure}
The following structural assertions hold.
\begin{itemize}
\item[\textup{(i)}] A nontrivial solvable group has a normal subgroup of prime index. A solvable minimal nontrivial normal subgroup of any group is elementary abelian.
\item[\textup{(ii)}] The coset action on $G/H$ has kernel $\core_G(H)=\bigcap_{g\in G}H^g$. If $|G:H|$ is the smallest prime divisor of $|G|$, then $H\triangleleft G$.
\item[\textup{(iii)}] If $Z\leq Z(G)$ and $G/Z$ is cyclic, then $G$ is abelian. If $G'\leq Z(G)$, then $G$ is nilpotent of class at most two.
\item[\textup{(iv)}] For a $p$-group $P$, one has $\Phi(P)=P^p[P,P]$, and $P/\Phi(P)$ is elementary abelian. A subset generates $P$ if and only if its image generates $P/\Phi(P)$. In particular, a non-cyclic $p$-group has a quotient $C_p^2$.
\item[\textup{(v)}] A normal Hall subgroup has a complement. The product of two nilpotent normal subgroups is nilpotent.
\item[\textup{(vi)}] A representation in characteristic not dividing the group order is completely reducible.
\item[\textup{(vii)}] Every group of odd order is solvable.
\item[\textup{(viii)}] If $A$ acts on a finite abelian group $V$ and $\gcd(|A|,|V|)=1$, then
$$ V=C_V(A)\times[V,A] \qquad\text{and}\qquad C_{V/D}(A)=C_V(A)D/D. $$
The second equality holds for every $A$-invariant subgroup $D\leq V$. If $V$ is a nontrivial cyclic $q$-group, then the actions on $V$ and $V/\Phi(V)\cong C_q$ have the same kernel. Their common image is cyclic of order dividing $q-1$.
\end{itemize}
\end{lemma}
\begin{proof}
For (i), take the quotient of a solvable group by a maximal proper normal subgroup. This quotient is simple and solvable. Its last nontrivial derived subgroup is abelian and normal, so the quotient is abelian. An abelian simple group is cyclic of prime order. If $N$ is a solvable minimal nontrivial normal subgroup, then $N'$ is characteristic in $N$, hence normal in the ambient group. Solvability and minimality imply $N'=1$. Each Sylow subgroup of the abelian group $N$ is characteristic, so only one prime $q$ divides $|N|$. Finally, $\{x\in N:x^q=1\}$ is nontrivial and characteristic; minimality implies that it is all of $N$. Thus $N$ is elementary abelian.

The kernel formula in (ii) follows because an element fixes every coset if and only if it lies in every conjugate of $H$. Suppose that $|G:H|=p$ is the smallest prime divisor of $|G|$. The transitive coset-action image has order dividing both $|G|$ and $p!$. No prime smaller than $p$ can divide this order, and $p!$ has only one factor $p$. Transitivity forces the image to have order $p$. The image therefore acts regularly, its point stabilizer is trivial, and $H$ is the kernel.

For (iii), a lift $g$ of a generator of $G/Z$ satisfies $G=\langle g\rangle Z$, which is abelian. The other assertion follows from the central series $1\leq Z(G)\leq G$.

For (iv), first let $H<P$. The action of $H$ on the cosets $P/H$ has $|N_P(H):H|$ fixed points, while each nontrivial orbit has size divisible by $p$. Since $|P:H|$ is divisible by $p$ and there is at least one fixed point, $|N_P(H):H|$ is a positive multiple of $p$. Thus $H<N_P(H)$. If $H$ is maximal, then this forces normality, and the resulting simple $p$-group quotient has order $p$. Consequently, $P^p[P,P]\leq\Phi(P)$. The quotient by $P^p[P,P]$ is a vector space over $\mathbb{F}_p$, and the intersection of its hyperplanes is zero. Taking the preimages of those hyperplanes proves the reverse inclusion.

If a set has images generating $P/\Phi(P)$ but generates a proper subgroup of $P$, then a maximal subgroup containing that proper subgroup gives a contradiction. This proves the assertion about generators and the existence of the quotient $C_p^2$. Part (v) follows from the Schur--Zassenhaus theorem and Fitting's product theorem~\cite[Theorem~3.8 and Corollary~1.29]{Isaacs2008}.

For (vi), let $p:V\rightarrow W$ be a linear projection onto an invariant subspace, and average its conjugates:
$$ \overline{p}=|A|^{-1}\sum_{a\in A}apa^{-1}. $$
This is an equivariant projection onto $W$, so $\ker\overline{p}$ is an invariant complement. Induction on the dimension proves complete reducibility. Part (vii) is the Feit--Thompson theorem~\cite{FeitThompson1963}.

For the abelian assertions in (viii), write $V$ additively, and choose an integer $c$ satisfying $c|A|\equiv1\pmod{\exp(V)}$. The averaging endomorphism
$$ e(v)=c\sum_{a\in A}a(v) $$
has image $C_V(A)$ and acts as the identity on that image. It annihilates every $a(v)-v$, and
$$ v-e(v)=c\sum_{a\in A}\bigl(v-a(v)\bigr) $$
lies in $[V,A]$. Hence $\ker e=[V,A]$, proving the direct product. If $v+D$ is fixed by $A$, then $e(v)+D=v+D$. Every fixed coset therefore has a fixed representative, proving the quotient formula.

Finally, for $V\cong C_{q^b}$, the kernel of $\Aut(V)\rightarrow\Aut(V/\Phi(V))$ consists of the units congruent to $1$ modulo $q$, and it is a $q$-group. Its intersection with the image of $A$, whose order is coprime to $q$, is trivial. That image therefore embeds in $\Aut(C_q)\cong C_{q-1}$.
\end{proof}

\subsection{Setpartitions and insertion}\label{sec:setpartition-insertion}

For an abelian group $N$, a setpartition $\mathcal{A}=A_1\bm\cdot\ldots\bm\cdot A_t$ is a sequence of nonempty subsets of $N$. Define $\Pi_k(\mathcal{A})$ by choosing one element from each of $k$ distinct terms. For a nonempty subset $F\subseteq N$, its stabilizer is $H(F)=\{x\in N:xF=F\}$. This is a subgroup, and $F$ is a union of its cosets.

The following is the setpartition form of the DeVos--Goddyn--Mohar theorem~\cite[Theorem~1.3]{DeVosGoddynMohar2009}, abbreviated DGM below. The integer $\nu_Q$ counts the sets meeting $Q$, rather than the elements in their intersections with $Q$.

\begin{lemma}\label{lem:dgm}
Let $1\leq k\leq t$, and $\mathcal{A}=A_1\bm\cdot\ldots\bm\cdot A_t$ be a setpartition of an abelian group $N$. Let $F=\Pi_k(\mathcal{A})$, and $M=H(F)$. For a coset $Q\in N/M$, define $\nu_Q=|\{i:A_i\cap Q\neq\varnothing\}|$. Then
$$ |F|\geq |M|\left(1-k+\sum_{Q\in N/M}\min\{k,\nu_Q\}\right). $$
\end{lemma}

\begin{lemma}\label{lem:kneser-cd}
If $A$ and $B$ are nonempty subsets of an abelian group and $M=H(AB)$, then
$$ |AB|\geq |AM|+|BM|-|M|. $$
If $p$ is a prime and $A_1,\ldots,A_t$ are nonempty subsets of $C_p$, then
$$ |A_1\cdots A_t|\geq \min\left\{p,\;1+\sum_{i=1}^t(|A_i|-1)\right\}. $$
\end{lemma}

The first assertion is Kneser's theorem~\cite[Theorem~4.3]{Nathanson1996}. The second assertion follows by repeated application of the two-set Cauchy--Davenport theorem~\cite[Theorem~2.2]{Nathanson1996}.

\begin{lemma}\label{lem:orbit-insertion}
Let $N\triangleleft G$ be abelian. Suppose that the ordering $g_1,\ldots,g_\ell$ of a sequence $T$ has product $c=g_1\cdots g_\ell\in N$. Let $b_0=1$ and $b_j=g_1\cdots g_j$ for $1\leq j\leq\ell$. Let $U_1,\ldots,U_t$ be pairwise disjoint sequences that are disjoint from $T$, and choose $u_i\in\pi(U_i)\cap N$. For $v\in N$, define
$$ O_T(v)=\{b_jvb_j^{-1}:0\leq j\leq\ell\} \qquad\text{and}\qquad O(v)=\{gvg^{-1}:g\in G\}. $$
Then, for every $I\subseteq\{1,\ldots,t\}$,
$$ c\prod_{i\in I}O_T(u_i)\subseteq \pi\left(T\bm\cdot\prod_{i\in I}U_i\right). $$
If the prefixes represent every coset of $N$ in $G$, then
$$ O(c)\prod_{i\in I}O(u_i)\subseteq \pi\left(T\bm\cdot\prod_{i\in I}U_i\right). $$
\end{lemma}
\begin{proof}
Choose an ordering of $U_i$ with product $u_i$. Inserting that ordering immediately after $g_j$ changes the product to $b_ju_ib_j^{-1}c$; insertion at the beginning corresponds to $j=0$. Since $N$ is abelian, two conjugating elements in the same coset of $N$ induce the same automorphism of $N$.

For simultaneous insertions, arrange the chosen orderings at their selected prefix positions, allowing several subsequences at one position when needed. Expanding the entire product in this order yields the product of the corresponding factors $b_ju_ib_j^{-1}$, followed by $c$. All these factors lie in $N$, so they commute. Thus every independent choice from the sets $O_T(u_i)$ is obtained.

If the prefixes cover $G/N$, then $O_T(v)=O(v)$. Rotate the ordering of $T$ after the prefix $b_j$. Its product becomes $b_j^{-1}cb_j$. Modulo $N$, its proper prefixes are the left translate by $b_j^{-1}$ of the original proper prefixes. They still cover $G/N$, since $c\in N$ and the final prefix repeats the identity coset. As $j$ varies, the rotated products cover $O(c)$. Applying the insertion argument to each rotation proves the last assertion.
\end{proof}

\subsection{Three consequences of the setpartition bound}\label{sec:setpartition-conseq}

The next lemma collects the DGM estimates used in the solvable case.

\begin{lemma}\label{lem:setpartition-conseq}
Let $1\leq k\leq t$, and $\mathcal{A}=A_1\bm\cdot\ldots\bm\cdot A_t$ be a setpartition of an abelian group $N$. Let $|N|=m$, let $F=\Pi_k(\mathcal{A})$, let $M=H(F)$, and let $h=|M|$. Define $\sigma_M=\sum_{Q\in N/M}\nu_Q$.
\begin{itemize}
\item[\textup{(i)}] Let $\mathcal{C}=\{Q:\nu_Q>k\}$. Then
$$ |F|\geq h\left(1-k+\sigma_M-\sum_{Q\in\mathcal{C}}(\nu_Q-k)\right). $$
If $\mathcal{C}$ is empty, then this bound is $h(1-k+\sigma_M)$. If $\mathcal{C}=\{Q_0\}$, then it is $h(1+\sigma_M-\nu_{Q_0})$. If $|\mathcal{C}|\geq2$, then it is at least $h(k+1)$.
\item[\textup{(ii)}] If $\exp(N)\mid k$ and $1\notin F$, then $\nu_Q<k$ for every coset $Q$, and
$$ h(t-k+1)\leq |F|\leq m-h. $$
In particular, if $m\geq3$ and $t=k+m-3$, then
$$ M=1 \qquad\text{and}\qquad \sum_{j=1}^t(|A_j|-1)\leq1. $$
\item[\textup{(iii)}] Suppose that an automorphism $\alpha$ fixes every $A_j$ setwise and preserves $M$. Suppose also that every nonidentity orbit on $N/M$ has at least $s\geq2$ elements, every $A_j$ is contained in or disjoint from $M$, and $h(1+(s-1)k)\geq m$. If $F\neq N$, then $\nu_Q\leq k$ for every $Q\neq M$. Moreover, with $i=|\{j:A_j\subseteq M\}|$, one has
$$ |F|\geq h\bigl(1-k+\min\{k,i\}+s(t-i)\bigr). $$
\end{itemize}
\end{lemma}
\begin{proof}
For (i), substitute $\min\{k,\nu_Q\}=\nu_Q-\max\{0,\nu_Q-k\}$ into the bound in Lemma~\ref{lem:dgm}. The first two cases follow immediately. If at least two cosets have incidence exceeding $k$, then those cosets contribute $2k$ to the sum in that lemma. The resulting bound is $h(1-k+2k)=h(k+1)$.

For (ii), suppose that $\nu_Q\geq k$. Choose one element of $Q$ from each of $k$ different sets. Their product lies in $Q^k=M$, since $\exp(N)\mid k$. Hence $F\cap M\neq\varnothing$. Since $F$ is a union of $M$-cosets, it would contain $M$, including the identity. This contradiction proves that $\nu_Q<k$ for every coset $Q$. Therefore each summand $\min\{k,\nu_Q\}$ in Lemma~\ref{lem:dgm} equals $\nu_Q$. Each nonempty $A_j$ contributes at least one incidence, so $|F|\geq h(1-k+t)$. The missing coset $M$ yields $|F|\leq m-h$.

When $t=k+m-3$, these bounds imply $h(m-1)\leq m$. Since $m\geq3$ and $h$ is a positive integer, $h=1$. The cosets are now singletons, so $\sigma_M=\sum_j|A_j|$, and
$$ m-1\geq |F|\geq1-k+\sum_j|A_j| =m-2+\sum_j(|A_j|-1). $$
This proves the additional assertion.

For (iii), the incidences are constant on $\alpha$-orbits of cosets. If a nonidentity coset has incidence greater than $k$, then at least $s$ distinct cosets contribute $k$ to the sum in Lemma~\ref{lem:dgm}. Hence $|F|\geq h(1-k+sk)\geq m$, contrary to $F\neq N$. Therefore $\nu_Q\leq k$ for every $Q\neq M$, so $\min\{k,\nu_Q\}=\nu_Q$ for all these cosets. Each of the $i$ sets inside $M$ contributes one incidence to $M$, and each remaining set meets at least $s$ cosets because it is $\alpha$-invariant and disjoint from $M$. Applying Lemma~\ref{lem:dgm} with these incidence bounds proves the assertion.
\end{proof}

\section{Reduction to non-abelian simple groups}\label{sec:simple-reduction}

We first record an estimate for elementary abelian groups.

\begin{lemma}\label{lem:elem-abelian-bound}
Let $p$ be prime, let $r\geq2$, and let $A=C_p^r$, an elementary abelian group of rank $r$. Then $\overline{\mathsf{E}}(A)\leq3/2$ if and only if
$$ 2\bigl(r(p-1)+1\bigr)\leq p^r. $$
This inequality holds for $r\geq3$, and for $r=2$ whenever $p\geq5$. Its only failures are $(p,r)=(2,2)$ and $(3,2)$.
\end{lemma}
\begin{proof}
By Lemma~\ref{lem:gao-identities}(i), one has $\mathsf{E}(A)=p^r+r(p-1)$, which proves the equivalence. For $r=2$, the inequality becomes $p^2-4p+2\geq0$. This fails for $p=2,3$ and holds for $p\geq5$.

For $r=3$, the inequality becomes $p^3-6p+4\geq0$. This polynomial vanishes at $p=2$ and has positive derivative for real $p\geq2$. Finally,
$$ p\bigl(r(p-1)+1\bigr)-\bigl((r+1)(p-1)+1\bigr) =r(p-1)^2\geq0. $$
Multiplying the inequality for $r$ by $p$ therefore implies the inequality for $r+1$. Induction completes the proof.
\end{proof}

Now, we transform the case of non-solvable groups into non-abelian simple groups.

\begin{proposition}\label{prop:simple-reduction}
The following assertions are equivalent:
\begin{itemize}
\item[\textup{(i)}] every non-solvable finite group $G$ satisfies $\mathsf{E}(G)\leq3|G|/2-1$;
\item[\textup{(ii)}] every non-abelian finite simple group $S$ satisfies $\mathsf{E}(S)\leq3|S|/2-1$.
\end{itemize}
\end{proposition}
\begin{proof}
It remains to show that (ii) implies (i). Assume (ii), and choose a non-solvable counterexample $G$ of least order, so that $\overline{\mathsf{E}}(G)>3/2$. Every proper normal subgroup of $G$ is solvable. Otherwise, minimality and Lemma~\ref{lem:lifting-bounds}(i) would imply $\overline{\mathsf{E}}(G)\leq\overline{\mathsf{E}}(H)\leq3/2$.

Moreover, $G$ has no nontrivial cyclic normal subgroup $Z$. Indeed, $G/Z$ would be non-solvable, since an extension of a solvable group by a solvable normal subgroup is solvable. Minimality and Lemma~\ref{lem:lifting-bounds}(ii) would again imply $\overline{\mathsf{E}}(G)\leq\overline{\mathsf{E}}(G/Z)\leq3/2$.

If $G$ is simple, then assumption (ii) is already a contradiction. Otherwise, choose a minimal nontrivial normal subgroup $N<G$. The subgroup $N$ is solvable, hence elementary abelian by Lemma~\ref{lem:group-structure}(i). The absence of cyclic normal subgroups implies that $N\cong C_q^r$ with $r\geq2$. By Lemmas~\ref{lem:lifting-bounds}(i) and~\ref{lem:elem-abelian-bound}, necessarily $N\cong C_2^2$ or $C_3^2$.

In either case, $\Aut(N)$ is solvable. For $C_2^2$, the action on the three nonzero vectors identifies the automorphism group with $S_3$. For $C_3^2$, the group $\GL_2(3)$ acts on the four one-dimensional subspaces. The kernel consists of scalar matrices: a matrix preserving each coordinate line is diagonal, and preservation of the line spanned by $(1,1)$ forces its diagonal entries to be equal. Thus the kernel has order two, and the image is a subgroup of $S_4$. Hence $\GL_2(3)$ is also solvable.

The conjugation homomorphism now shows that $G/C_G(N)\leq\Aut(N)$ is solvable. Since $G$ is non-solvable, $C_G(N)$ must be non-solvable. It is normal in $G$, so the first paragraph forces $C_G(N)=G$. Consequently, $N\leq Z(G)$, and any subgroup of $N$ of order $q$ is a nontrivial cyclic normal subgroup of $G$. This is the required contradiction.
\end{proof}

\section{The solvable reduction}\label{sec:solvable-reduction}

We first settle the nilpotent case and then reduce a solvable counterexample of smallest order to five configurations. The estimates that exclude these configurations are proved in Sections~\ref{sec:a4-extensions}--\ref{sec:prime-square-quotients}.

\subsection{Nilpotent groups}\label{sec:nilpotent}

\begin{proposition}\label{prop:nilpotent-bound}
Let $G$ be a finite non-cyclic nilpotent group, and let $p$ be the smallest prime divisor of $|G|$. If $G$ has no cyclic subgroup of index $p$, then
$$ \mathsf{E}(G)\leq |G|+\frac{|G|}{p}-1. $$
\end{proposition}

\begin{proof}
By Lemma~\ref{lem:augmentation-bounds}, every non-cyclic $p$-group $P$ with no cyclic subgroup of index $p$ satisfies
\begin{equation} \mathsf{E}(P)\leq |P|+t(P)-1\leq |P|+\frac{|P|}{p}-1. \tag{4.1}\label{eq:pgroup-gao} \end{equation}
We apply this estimate to $G$, whose Sylow subgroups form a direct product.

\medskip\noindent\textbf{Case 1.} Some Sylow $q$-subgroup is non-cyclic for a prime $q>p$.\nopagebreak[4]

By the Frattini quotient theorem, this Sylow subgroup has a quotient $C_q^2$, so $G$ also has such a quotient. By Lemma~\ref{lem:gao-bounds}(ii),
$$ \mathsf{E}(G)\leq |G|+\frac{|G|}{q}+q-2. $$
Let $N=|G|$. Since $N\geq pq^2$, we have
$$ \frac{N}{p}-\frac{N}{q} =\frac{N(q-p)}{pq} \geq q(q-p)\geq q-1. $$
This proves the required estimate.

\medskip\noindent\textbf{Case 2.} Every Sylow $q$-subgroup with $q>p$ is cyclic.\nopagebreak[4]

Then $G=P\times D$, where $P$ is the Sylow $p$-subgroup of $G$ and $D$ is cyclic of order prime to $p$. The group $P$ is non-cyclic. A cyclic subgroup of index $p$ in $P$, together with $D$, would give a cyclic subgroup of index $p$ in $G$. Thus $P$ has no cyclic subgroup of index $p$. Therefore, the estimate from Lemma~\ref{lem:augmentation-bounds} applies to $P$. By Lemma~\ref{lem:lifting-bounds}(i),
$$ \mathsf{E}(G)+1 \leq |D|\bigl(\mathsf{E}(P)+1\bigr) \leq |D|\left(|P|+\frac{|P|}{p}\right) =|G|+\frac{|G|}{p}. $$
\end{proof}

\subsection{The five configurations}\label{sec:solvable-configs}

\begin{theorem}\label{thm:solvable-configs}
Suppose that the asserted bound fails for a finite non-cyclic solvable group, and choose a counterexample $G$ of smallest order. Let $n=|G|$, and let $p$ be the smallest prime divisor of $n$. Thus $G$ has no cyclic subgroup of index $p$, and
$$ \mathsf{E}(G)>n+\frac{n}{p}-1. $$
Then $G$ is not nilpotent, and it satisfies one of the following five alternatives.
\begin{itemize}
\item[(S1)] The prime $p$ is $2$, and $G$ has a cyclic normal subgroup $C$ such that $G/C\cong A_4$.
\item[(S2)] The prime $p$ is odd, $q=2p-1$ is prime, and $G\cong C_q^2\rtimes C_p$, with an irreducible action on $C_q^2$.
\item[(S3)] The numbers $p$ and $q$ are odd primes with $p<q<2p$, and $s>q$ is a prime such that $q\mid s-1$. Moreover,
$$ G\cong C_p\times(C_s\rtimes C_q) \qquad\text{and}\qquad s(q-p)<p(q-1). $$
The action of $C_q$ on $C_s$ is nontrivial.
\item[(S4)] The group $G$ has a cyclic normal subgroup $B$ such that $G/B\cong C_p^2$.
\item[(S5)] The group $G$ is isomorphic to $C_m\rtimes C_{p^2}$, where $\gcd(m,p)=1$ and the image of the action of $C_{p^2}$ has order $p^2$.
\end{itemize}
\end{theorem}

\begin{proof}
By Proposition~\ref{prop:nilpotent-bound}, the group $G$ is not nilpotent. All splittings over normal Hall subgroups below use the Schur--Zassenhaus theorem, as recorded in Lemma~\ref{lem:group-structure}. We first record the consequences of minimality and of the existence of a cyclic subgroup of prime index.

\medskip\noindent\textbf{Claim 1.} Every nontrivial proper subnormal subgroup $K$ of $G$ has a cyclic subgroup of index equal to the smallest prime divisor $q$ of $|K|$. Moreover, if $C\mathbin{\trianglelefteq} G$ is nontrivial and cyclic, and $G/C$ is non-cyclic with smallest prime divisor $p$, then $G/C$ has a cyclic subgroup of index $p$.\nopagebreak[4]

\begin{proof}[Proof of Claim~1]
The assertion for cyclic $K$ is immediate. If a non-cyclic subgroup $K$ failed to have the asserted property, then minimality and Lemma~\ref{lem:lifting-bounds} would imply
$$ \overline{\mathsf{E}}(G)\leq \overline{\mathsf{E}}(K)\leq 1+\frac{1}{q}\leq 1+\frac{1}{p}, $$
contrary to the choice of $G$. If the assertion about $G/C$ failed, then the same argument, applied to the cyclic normal quotient, would yield
$$ \overline{\mathsf{E}}(G)\leq \overline{\mathsf{E}}(G/C)\leq 1+\frac{1}{p}. $$
This is again a contradiction.
\end{proof}

\medskip\noindent\textbf{Claim 2.} Suppose that $H\mathbin{\trianglelefteq} G$ has a cyclic normal subgroup $A$ of index $p$, and that $A$ is not normal in $G$. Then $H$ is nilpotent. Moreover, if $C=\core_G(A)$, then\nopagebreak[4]
\begin{equation} H/C\cong C_p^2. \tag{4.2}\label{eq:elem-abelian-quotient} \end{equation}

\begin{proof}[Proof of Claim~2]
Since $A$ is not normal in $G$, there is an element $g\in G$ such that $A^g\neq A$. The normality of $H$ in $G$ implies that $A^g\leq H$. Moreover, conjugating $A\triangleleft H$ by $g$ gives $A^g\triangleleft H^g=H$, and $|H:A^g|=|H:A|=p$. Thus $A$ and $A^g$ are distinct maximal subgroups of $H$, so $H=AA^g$. Both factors are cyclic and normal in $H$. Fitting's theorem therefore implies that $H$ is nilpotent.

We now describe the diagonal map explicitly. For every $g\in G$, let
$$ \pi_g:H\longrightarrow H/A^g,\qquad\text{and}\qquad \pi_g(h)=hA^g. $$
Each $\pi_g$ is a homomorphism because $A^g\triangleleft H$, and each quotient $H/A^g$ has order $p$, so it is isomorphic to $C_p$. Taking all these quotient maps as coordinates gives the diagonal homomorphism
$$ \Delta:H\longrightarrow\prod_{g\in G}H/A^g,$$
where $\Delta(h)=(hA^g)_{g\in G}$.
Its kernel consists of the elements lying in every conjugate of $A$; hence
$$ \ker\Delta=\bigcap_{g\in G}A^g=\core_G(A)=C. $$
In particular, $C\triangleleft H$, and $\Delta$ induces the injective homomorphism
$$ \overline{\Delta}:H/C\longrightarrow\prod_{g\in G}H/A^g,$$
where $\overline{\Delta}(hC)=(hA^g)_{g\in G}$.

The direct product on the right is elementary abelian of exponent $p$. Its subgroup $\overline{\Delta}(H/C)$, and consequently $H/C$, is also elementary abelian. Since $C\leq A$ and $A$ is cyclic, $A/C$ is cyclic. It is nontrivial: equality $C=A$ would make $A$ normal in $G$. A nontrivial cyclic subgroup of an elementary abelian $p$-group has order $p$, so $|A:C|=p$. Finally,
$$ |H:C|=|H:A|\,|A:C|=p\cdot p=p^2. $$
Thus $H/C\cong C_p^2$, proving the claim.
\end{proof}

We distinguish three cases.

\medskip\noindent\textbf{Case 1.} The group $G$ has no normal subgroup of index $p$.\nopagebreak[4]

Solvability supplies a normal subgroup $H_0$ of prime index $q$, and the assumption in this case implies that $q>p$. Since $p\mid |H_0|$, Claim~1 supplies a cyclic subgroup $A$ of index $p$ in $H_0$. A subgroup whose index is the smallest prime divisor of the group order is normal; hence $A\mathbin{\trianglelefteq} H_0$. If $A\mathbin{\trianglelefteq} G$, then $G/A$ has order $pq$. Its Sylow $q$-subgroup is normal, because the number of such subgroups divides $p$ and is congruent to $1$ modulo $q$. The inverse image of this Sylow subgroup would be normal of index $p$ in $G$, a contradiction.

Consequently, $A$ is not normal in $G$. By Claim~2, the subgroup $C=\core_G(A)$ satisfies $H_0/C\cong C_p^2$. Conjugation by $G$ on this quotient factors through $G/H_0\cong C_q$. This action is nontrivial: otherwise $A/C$ would be invariant under $G/C$, and $A$ would be normal in $G$. Therefore
$$ q\mid |\GL_2(p)|=p(p-1)^2(p+1). $$
Since $q>p$, this forces $q\mid p+1$, and hence $q=p+1$. The only consecutive primes with this property are $(p,q)=(2,3)$. By the Schur--Zassenhaus theorem, $G/C\cong C_2^2\rtimes C_3$. The nontrivial action is the unique action of order three on the three nonidentity elements of $C_2^2$. Thus $G/C\cong A_4$. Since $C\leq A$ is cyclic, alternative (S1) holds.

For the remaining cases, fix $H\mathbin{\trianglelefteq} G$ with $|G:H|=p$. The hypothesis on $G$ implies that $H$ is non-cyclic.

\medskip\noindent\textbf{Case 2.} The prime $p$ does not divide $|H|$.\nopagebreak[4]

Let $q$ be the smallest prime divisor of $|H|$. Then $q>p$. By Claim~1, the group $H$ has a cyclic subgroup of index $q$. By Lemmas~\ref{lem:gao-identities} and~\ref{lem:lifting-bounds},
$$\mathsf{E}(G)+1 \leq p\bigl(\mathsf{E}(H)+1\bigr)=p\left(|H|+\frac{|H|}{q}+q-1\right) =n+\frac{n}{q}+p(q-1).$$
Comparison with $\mathsf{E}(G)+1>n+n/p$ yields
\begin{equation} |H|(q-p)<pq(q-1). \tag{4.3}\label{eq:hall-size-constraint} \end{equation}
Since $H$ is non-cyclic, $|H|\geq q^2$. Hence $q(q-p)<p(q-1)$, which implies that $q<2p$. Moreover,
$$ |H|<\frac{pq(q-1)}{q-p}<q^3. $$
The last inequality follows from $p<q$ and $q-p\geq1$. Every prime divisor of $|H|$ is at least $q$, and $|H|$ has at least two prime factors counted with multiplicity. Since $|H|<q^3$, it has exactly two such prime factors. Therefore $|H|=qs$ for a prime $s\geq q$.

\medskip\noindent\textbf{Case 2.1.} The prime $p$ is $2$.\nopagebreak[4]

Then $q=3$, and \eqref{eq:hall-size-constraint} implies that $|H|<12$, so $s=3$. Since $H$ is non-cyclic, $H\cong C_3^2$. The Schur--Zassenhaus theorem gives $G\cong C_3^2\rtimes C_2$. The action cannot be trivial, because $G$ is not nilpotent. An involution over $\mathbb{F}_3$ is diagonalizable with eigenvalues in $\{1,-1\}$. The two nontrivial possibilities give $S_3\times C_3$ and $C_3^2\rtimes_{-1}C_2$. By Lemma~\ref{lem:gao-bounds}(iv), both groups satisfy $\mathsf{E}(G)\leq26$, whereas $n+n/p-1=26$. This contradiction shows that $p$ is odd.

\medskip\noindent\textbf{Case 2.2.} The prime $p$ is odd and $s=q$.\nopagebreak[4]

Then $H\cong C_q^2$ and $G\cong C_q^2\rtimes C_p$. The action is nontrivial, so $p\mid q(q-1)^2(q+1)$. Since $p<q<2p$ and both primes are odd, we have $p\nmid q-1$: the only possible multiple of $p$ in the relevant interval would give $q=p+1$, which is even. It follows that $p\mid q+1$, and therefore $q+1=2p$.

Maschke's theorem implies that the action is completely reducible. Every one-dimensional action would be trivial, since $p\nmid q-1$. A reducible action would therefore be trivial, so the action is irreducible. Thus alternative (S2) holds.

\medskip\noindent\textbf{Case 2.3.} The prime $p$ is odd and $s>q$.\nopagebreak[4]

The non-cyclic group $H$ of order $qs$ is isomorphic to $C_s\rtimes C_q$ with a nontrivial action, and $q\mid s-1$. Indeed, its Sylow $s$-subgroup is normal, and a trivial complementary action would make $H$ cyclic. Its unique Sylow $s$-subgroup $H_s$ is characteristic in $H$, and hence normal in $G$. The quotient $G/H_s$ has order $pq$ and is cyclic: its Sylow $q$-subgroup is normal, and $p\nmid q-1$ rules out a nontrivial action. Thus $G\cong C_s\rtimes C_{pq}$.

The action has order $q$ or $pq$, because its restriction to the subgroup of order $q$ is nontrivial. However, \eqref{eq:hall-size-constraint} implies that
$$ s<\frac{p(q-1)}{q-p}<pq, $$
whereas an action of order $pq$ on $C_s$ would require $pq\mid s-1$. The action therefore has order $q$, and the subgroup of order $p$ in the complement acts trivially. This proves both the direct product description and the inequality in alternative (S3).

\medskip\noindent\textbf{Case 3.} The prime $p$ divides $|H|$.\nopagebreak[4]

By Claim~1, there is a cyclic subgroup $A\mathbin{\trianglelefteq} H$ of index $p$.

\medskip\noindent\textbf{Claim 3.} The group $G$ has a cyclic normal subgroup $B$ of index $p^2$.\nopagebreak[4]

\begin{proof}[Proof of Claim~3]
If $A\mathbin{\trianglelefteq} G$, then take $B=A$. Otherwise, by Claim~2, the group $H$ is nilpotent. Let $D$ be its Hall $p'$-subgroup. The subgroup $D$ is characteristic in $H$, lies in $A$, and is cyclic. Consequently $D\mathbin{\trianglelefteq} G$. Moreover, $H\leq C_G(D)$. If $C_G(D)=G$, then the cyclic normal Hall subgroup $D$ is central, and a complement gives $G=D\times P$ for a $p$-group $P$. This would make $G$ nilpotent. Hence $C_G(D)=H$.

The quotient $H/D$ is a non-cyclic $p$-group. Otherwise the nilpotent group $H$ would be a direct product of two cyclic groups of coprime orders, and hence would be cyclic. Thus $G/D$ is also a non-cyclic $p$-group. Its Frattini quotient has dimension at least two, so $G/D$ has a maximal subgroup different from $H/D$. The inverse image $K$ is normal of index $p$ in $G$ and differs from $H$.

By Claim~1, the group $K$ has a cyclic normal subgroup $L$ of index $p$. Since $|\varphi_L(D)|\mid \gcd(|D|, |K/L|)=1$, the image of $D$ in $K/L$ is trivial, so $D\leq L$. Since $L$ is cyclic, it centralizes $D$, and hence $L\leq C_G(D)=H$. Both $L$ and $K\cap H$ have order $n/p^2$, so $L=K\cap H$. Thus $B=L$ is cyclic and normal in $G$, with index $p^2$.
\end{proof}

Fix $B$ as in Claim~3, and let $b=|B|$.

\medskip\noindent\textbf{Case 3.1.} The quotient $G/B$ is isomorphic to $C_p^2$.\nopagebreak[4]

This is alternative (S4).

\medskip\noindent\textbf{Case 3.2.} The quotient $G/B$ is isomorphic to $C_{p^2}$ and $p\nmid b$.\nopagebreak[4]

Then $G\cong C_b\rtimes C_{p^2}$. If the action had order at most $p$, then the subgroup of order $p$ in the complement would centralize $C_b$. Its product with $C_b$ would then be cyclic of order $bp$, and hence of index $p$ in $G$. This contradicts the hypothesis on $G$. Thus the action has order $p^2$, and alternative (S5) holds.

\medskip\noindent\textbf{Case 3.3.} The quotient $G/B$ is isomorphic to $C_{p^2}$ and $p\mid b$.\nopagebreak[4]

Let $Z$ be the unique subgroup of order $p$ in $B$. The subgroup $Z$ is normal in $G$, and the image of its conjugation action has order dividing both $|G|$ and $p-1$. Since $p$ is the smallest prime divisor of $|G|$, this image is trivial. Hence $Z\leq Z(G)$. The quotient $G/Z$ is non-cyclic, since a cyclic quotient by a central subgroup would make $G$ abelian. Since $p^3\mid |G|$, the smallest prime divisor of $|G/Z|$ is still $p$. Claim~1 therefore supplies a cyclic subgroup of index $p$ in $G/Z$.

Let $H_1$ be the inverse image of a cyclic subgroup of index $p$ in $G/Z$. Then $H_1\mathbin{\trianglelefteq} G$, $|G:H_1|=p$, and $H_1/Z$ is cyclic. Since $Z$ is central, $H_1$ is abelian. The hypothesis on $G$ implies that $H_1$ is non-cyclic. Also $B\leq H_1$: otherwise $G=BH_1$ would be a product of two nilpotent normal subgroups and hence nilpotent.

Choose $u\in H_1$ whose image generates the cyclic group $H_1/Z$, which has order $b$. The order of $u$ is $b$ or $pb$. The latter possibility would make $H_1$ cyclic, so $|u|=b$ and $\langle u\rangle\cap Z=1$. Since $B/Z$ is the unique subgroup of index $p$ in $H_1/Z$, we obtain
$$ H_1=\langle u\rangle\times Z \qquad\text{and}\qquad B=\langle u^p\rangle\times Z. $$
The cyclicity of $B$ forces $\gcd(b/p,p)=1$. Thus $D=\langle u^p\rangle$ is the characteristic Hall $p'$-subgroup of $B$, so $D\mathbin{\trianglelefteq} G$.

Now $Q=G/D$ has order $p^3$. The subgroup $B/D$ has order $p$ and is central in $Q$, while $Q/(B/D)\cong C_{p^2}$ is cyclic. Therefore $Q$ is abelian. It contains $H_1/D\cong C_p^2$, which excludes the cyclic group of order $p^3$, and it has a cyclic quotient of order $p^2$, which excludes $C_p^3$. Consequently $Q\cong C_p\times C_{p^2}$.

Let $K$ be the inverse image in $G$ of $\Phi(Q)=Q^p$, which has order $p$. The Frattini subgroup lies in every maximal subgroup, so $K\leq H_1$. Thus $K$ is abelian, its Hall $p'$-subgroup $D$ is cyclic, and its Sylow $p$-subgroup has order $p$. It follows that $K$ is cyclic. Finally,
$$ G/K\cong Q/\Phi(Q)\cong C_p^2. $$
This is alternative (S4), and the reduction is complete.
\end{proof}

\subsection{Completion of the solvable proof}\label{sec:solvable-completion}

The estimates proved below exclude the five configurations as follows.
\begin{center}
\begin{tabular}{@{}lll@{}}
\hline
Case & Structure & Estimate\\
\hline
(S1) & Cyclic normal kernel, quotient $A_4$ & Proposition~\ref{prop:a4-extension}\\
(S2) & $C_q^2\rtimes C_p$, irreducible action & Proposition~\ref{prop:irred-frobenius}\\
(S3) & $C_p\times(C_s\rtimes C_q)$ & Proposition~\ref{prop:frobenius-direct-product}\\
(S4) & Cyclic normal kernel, quotient $C_p^2$ & Proposition~\ref{prop:elem-p2-quotient}\\
(S5) & $C_m\rtimes C_{p^2}$, action of order $p^2$ & Proposition~\ref{prop:cyclic-p2-quotient}\\
\hline
\end{tabular}
\end{center}

\begin{corollary}\label{cor:solvable-bound}
Let $G$ be a finite non-cyclic solvable group, and let $p$ be the smallest prime divisor of $|G|$. If $G$ has no cyclic subgroup of index $p$, then
$$ \mathsf{E}(G)\leq |G|+\frac{|G|}{p}-1. $$
\end{corollary}

\begin{proof}
If the statement were false, then choose a counterexample of smallest order. By Theorem~\ref{thm:solvable-configs}, it satisfies one of alternatives (S1)--(S5). The corresponding proposition in the table establishes the asserted bound in each case, a contradiction.
\end{proof}

\section{Cyclic extensions of \texorpdfstring{$A_4$}{A4}}\label{sec:a4-extensions}

\begin{lemma}\label{lem:a4-six-term}
Every sequence of length $11$ over $A_4$ has a product-one subsequence of length six. Equivalently, $\mathsf{s}_6(A_4)\leq11$.
\end{lemma}

\begin{proof}
Identify the normal Klein four subgroup with the additive group of $\mathbb{F}_4$, and write
$$ A_4=(\mathbb{F}_4,+)\rtimes\langle x\rangle \qquad\text{and}\qquad xax^{-1}=\omega a. $$
Here $\omega\in\mathbb{F}_4^{\times}$ has order three. Thus $1+\omega+\omega^2=0$. The field element $0$ is the group identity, whereas the field element $1$ is a nonidentity group element. Every element outside $\mathbb{F}_4$ has order three, since
$$ (ax^j)^3=a+\omega^ja+\omega^{2j}a=0. $$
This holds for $j=1,2$.

Let $S$ have length $11$, and suppose that $1\notin\Pi_6(S)$. Take any three-term subsequence $T$ with products in $\mathbb{F}_4$. By the cyclic Erd\H{o}s--Ginzburg--Ziv theorem, the other eight terms, projected onto $A_4/\mathbb{F}_4\cong C_3$, contain two disjoint three-term zero-sum subsequences $U$ and $W$. No sum of two of the three product sets $\pi(T)$, $\pi(U)$, and $\pi(W)$ contains zero, since such a sum would give six terms with product one.

Apply Lemma~\ref{lem:setpartition-conseq}(ii) to these sets in the additive group of $\mathbb{F}_4$, with $m=4$, $k=2$, and $t=3=k+m-3$. This group has exponent two, so
$$ |\pi(T)|+|\pi(U)|+|\pi(W)|\leq4. $$
In particular,
\begin{equation} |\pi(T)|\leq2. \tag{5.1}\label{eq:a4-products-cap} \end{equation}
This inequality holds whenever $|T|=3$ and $\pi(T)\subseteq\mathbb{F}_4$.

\medskip\noindent\textbf{Claim.} Under \eqref{eq:a4-products-cap}, all terms in a coset $\mathbb{F}_4x^j$ are equal whenever that coset contains at least three terms, for $j=1,2$. If both nonidentity cosets contain terms of $S$, then every term of $S$ in $\mathbb{F}_4$ is $0$.\nopagebreak[4]

\begin{proof}[Proof of the claim]
Two elements $ax^j$ and $bx^j$ commute if and only if
$$ a+\omega^jb=b+\omega^ja. $$
Since $1+\omega^j\neq0$, this equality holds if and only if $a=b$. A nonconstant triple in this coset therefore has two distinct products, obtained by exchanging two distinct adjacent terms. At least one of these products is nonzero in $\mathbb{F}_4$. The three cyclic rotations of an ordering with that product give its three distinct scalar multiples under multiplication by $\omega^j$. This contradicts \eqref{eq:a4-products-cap}.

For the second assertion, suppose that both nonidentity cosets occur. Take a nonzero term $v\in\mathbb{F}_4$ and terms $g\in\mathbb{F}_4x$ and $h\in\mathbb{F}_4x^2$. Let $s=gh\in\mathbb{F}_4$. The orderings $(v,g,h)$ and $(g,v,h)$ have distinct products
$$ q=v+s \qquad\text{and}\qquad q'=\omega v+s. $$
Moving the final term $h$ to the beginning of these two orderings gives the additional products $\omega^2q$ and $\omega^2q'$. If one of $q$ and $q'$ is zero, then these four products include zero and two distinct nonzero elements. If both are nonzero, then the set $\{q,q'\}$ cannot be invariant under the three-cycle induced by multiplication by $\omega^2$. In that case, the four products include all three nonzero field elements. Either possibility contradicts \eqref{eq:a4-products-cap}.
\end{proof}

Let $a$, $b$, and $c$ count the terms in $\mathbb{F}_4$, $\mathbb{F}_4x$, and $\mathbb{F}_4x^2$, respectively. Thus $a+b+c=11$. If $b\geq6$ or $c\geq6$, then six equal terms from that coset already have product one. We may therefore assume that $b,c\leq5$.

\medskip\noindent\textbf{Case 1.} Both $b$ and $c$ are positive.\nopagebreak[4]

By the claim, all field terms are identities. Every three terms in either of the other cosets are equal and have product one. The three cosets therefore supply
$$ \left\lfloor\frac{a}{3}\right\rfloor +\left\lfloor\frac{b}{3}\right\rfloor +\left\lfloor\frac{c}{3}\right\rfloor \geq\frac{a+b+c-6}{3} =\frac{5}{3}>1 $$
disjoint product-one triples. Two of these triples give the required six terms.

\medskip\noindent\textbf{Case 2.} At most one nonidentity coset occurs.\nopagebreak[4]

Interchanging $x$ and $x^{-1}$ when necessary, assume that $c=0$. If $b\leq2$, then $a\geq9$. A maximal family of equal pairs among the field terms leaves at most four unpaired terms, one for each field element. It therefore contains at least $\lceil(a-4)/2\rceil\geq3$ pairs. Each pair sums to zero, so three pairs suffice.

If $3\leq b\leq5$, then all terms outside the field equal an element $y$ of order three, and $a\geq6$. Three field identities together with three copies of $y$ would suffice. If there are fewer than three field identities, then choose three nonzero field terms $w_1$, $w_2$, and $w_3$. The prefixes of $(y,y,y)$ represent all cosets of $\mathbb{F}_4$, and the conjugacy orbit of each $w_i$ is $\mathbb{F}_4^{\times}$. By Lemma~\ref{lem:orbit-insertion}, inserting these three terms into $(y,y,y)$ realizes the sum $1+\omega+\omega^2=0$. This gives six terms with product one.
\end{proof}

\begin{proposition}\label{prop:a4-extension}
If $G$ has a cyclic normal subgroup $C$ such that $G/C\cong A_4$, then
$$ \mathsf{E}(G)\leq\frac{3|G|}{2}-1. $$
In particular, configuration (S1) cannot be a counterexample.
\end{proposition}

\begin{proof}
By Lemma~\ref{lem:a4-six-term}, every sequence of length $17$ over $A_4$ contains a product-one subsequence of length six. The remaining $11$ terms contain another such subsequence. Concatenating product-one orderings of these two subsequences gives $12$ terms with product one. Thus $\mathsf{E}(A_4)\leq17$. By Lemma~\ref{lem:lifting-bounds}(ii),
$$ \mathsf{E}(G)+1 \leq |C|\bigl(\mathsf{E}(A_4)+1\bigr) \leq18|C|=\frac{3|G|}{2}. $$
The smallest prime divisor of $|G|$ is $2$, so this is exactly the bound required to exclude alternative (S1).
\end{proof}

\section{The two normal Hall subgroup cases}\label{sec:hall-cases}

The following propositions exclude alternatives (S2) and (S3) of Theorem~\ref{thm:solvable-configs}. Products of indexed sets use one element from each selected set, and distinct occurrences retain distinct indices.

\subsection{An irreducible Frobenius group}\label{sec:irred-frobenius}

We use \cite[Lemma~3.2]{QuLiTeeuwsen2022} in the following form. If $q$ is a prime and $U$ is a zero-sum-free sequence of length $2q-3$ over $C_q^2$, then there is an affine line $x+H$ such that
\begin{equation} \bigl(C_q^2\setminus\{0\}\bigr)\setminus\Sigma(U)\subseteq x+H. \tag{6.1}\label{eq:plane-subsum-complement} \end{equation}
Here $\Sigma(U)$ denotes the set of nonempty subsums of $U$.

\begin{proposition}\label{prop:irred-frobenius}
Let $p$ be an odd prime such that $q=2p-1$ is prime. Suppose that $G=V\rtimes C_p$, where $V\cong C_q^2$ and the action is irreducible. Then
$$ \mathsf{E}(G)\leq pq^2+q^2-1. $$
\end{proposition}

\begin{proof}
Write $V$ additively, and let $\alpha$ be the automorphism induced by a generator of the complement. Define
$$ O(v)=\{\alpha^j(v):0\leq j<p\}. $$
Every nonzero orbit has exactly $p$ elements. A shorter orbit would be fixed by the group of prime order, yielding a nonzero fixed vector and an invariant line. In particular, $\alpha-I$ is invertible. Since
$$ (\alpha-I)\sum_{j=0}^{p-1}\alpha^j=\alpha^p-I=0, $$
it follows that
$$ \sum_{j=0}^{p-1}\alpha^j(v)=0. $$
No nonzero orbit is contained in an affine line. Indeed, suppose that $O(v)\subseteq x+H$, where $H$ has dimension one. Summing the orbit gives $px\in H$. Since $p\neq q$, multiplication by $p$ is invertible on $V/H$, so $x\in H$. The span of the nonzero orbit would then be an $\alpha$-invariant line, a contradiction.

Let $S$ have length $(p+1)q^2-1$, and suppose that $1\notin\Pi_{pq^2}(S)$. By Lemmas~\ref{lem:davenport-bounds}(i) and~\ref{lem:gao-identities}(i), we have $\mathsf{E}(V)=q^2+2q-2$. By the concentration inequality in Lemma~\ref{lem:repeated-extraction},
$$ |S_{G\setminus V}|\geq2q^2-\mathsf{E}(V)=(q-1)^2+1. $$
Since $(q-1)^2=4(p-1)^2$, some nonidentity coset contains at least $p$ terms. Choose $p$ of these terms to form a subsequence $B$. Its image in $G/V$ is product-one.

Let $\ell=q^2+2q-4$. The identity $q=2p-1$ gives
$$ |S|-p=p(\ell+1). $$
By Lemma~\ref{lem:repeated-extraction}, applied in $G/V\cong C_p$, there are pairwise disjoint subsequences $T_1,\ldots,T_\ell$, all disjoint from $B$, each of length $p$ and each with a selected product in $V$. Fix
$$ b\in\pi(B) \qquad\text{and}\qquad a_i\in\pi(T_i). $$
The latter choice is made for each $1\leq i\leq\ell$. The prefixes of any ordering of $B$ represent all cosets of $V$. By Lemma~\ref{lem:orbit-insertion},
\begin{equation} O(b)+\sum_{i\in I}O(a_i) \subseteq\pi\left(B\bm\cdot\prod_{i\in I}T_i\right). \tag{6.2}\label{eq:plane-orbit-insertion} \end{equation}
This holds for every $I\subseteq[1,\ell]$, and the product on the right is taken in the sequence monoid. Let $z=|\{i:a_i=0\}|$. We have $z\leq q^2-1$, since $q^2$ of these subsequences with selected product zero would already have total length $pq^2$.

\medskip\noindent\textbf{Case 1.} We have $z\leq q^2-3$.\nopagebreak[4]

Let $k=q^2-1$, and define
$$ P=\Sigma_k\bigl(O(a_1),\ldots,O(a_\ell)\bigr). $$
If $P\neq V$, then its stabilizer is a proper $\alpha$-invariant subgroup of $V$, and hence a proper invariant vector subspace. Irreducibility forces this stabilizer to be $\{0\}$. Each indexed orbit is either $\{0\}$ or has $p$ nonzero elements, and $1+(p-1)k\geq q^2$. By Lemma~\ref{lem:setpartition-conseq}(iii),
$$ \begin{aligned} |P| &\geq1-k+z+p(\ell-z)\\ &\geq1-k+(q^2-3)+p(\ell-q^2+3)\\ &=q^2+p-2>q^2, \end{aligned} $$
a contradiction. Consequently $P=V$. In particular, $-b$ is realized by one orbit element from each of $q^2-1$ indexed subsequences. By Lemma~\ref{lem:orbit-insertion}, applied as above, these subsequences together with $B$ yield the prohibited product-one subsequence of length $pq^2$.

\medskip\noindent\textbf{Case 2.} We have $z\in\{q^2-2,q^2-1\}$.\nopagebreak[4]

At least $2q-3$ of the elements $a_i$ are nonzero. Select that many, and let $U$ be the resulting sequence. If a nonempty subsequence of $U$ has sum zero and length $d$, then $2\leq d\leq2q-3$. The $d$ corresponding subsequences, together with $q^2-d$ subsequences whose selected products are zero, have product zero and total length $pq^2$. The required subsequences with selected product zero exist because $q^2-d\leq q^2-2\leq z$. Thus $U$ is zero-sum-free, and \cite[Lemma~3.2]{QuLiTeeuwsen2022} applies.

If $b=0$ and $z=q^2-1$, then use $B$ together with those $q^2-1$ subsequences with selected product zero. Otherwise, choose a nonzero orbit as follows. If $b\neq0$, then use $O(b)$. If $b=0$, then $z=q^2-2$, and there are exactly $2q-2$ nonzero elements $a_i$; use $O(c)$ for the one not included in $U$. The negative of the chosen orbit is not contained in the affine line supplied by \cite[Lemma~3.2]{QuLiTeeuwsen2022}. Thus the chosen orbit contains $-v$ for some $v\in\Sigma(U)$. Choose a nonempty subsequence $W\mid U$ with sum $v$.

By Lemma~\ref{lem:orbit-insertion}, zero can now be realized using $B$, the subsequences indexed by $W$, and, when $b=0$, the additional subsequence with selected product $c$. Between $2$ and $2q-1\leq q^2$ subsequences have been used. Complete them to $q^2$ subsequences by adjoining subsequences with selected product zero; at most $q^2-2\leq z$ are required. Their total length is $pq^2$, a contradiction.
\end{proof}

\subsection{A direct product with a non-abelian group of order \texorpdfstring{$bc$}{bc}}\label{sec:frobenius-direct-product}

We first bound the number of disjoint subsequences with selected products in the center, and then determine the exceptional case in the estimate for products of orbits.

\begin{lemma}\label{lem:central-extraction-cap}
Let $a<b<2a$ be odd primes, let $c$ be a prime such that $b\mid c-1$, and let
$$ G=C_a\times(C_c\rtimes C_b). $$
The action of $C_b$ on $C_c$ is assumed to be nontrivial. Let $Z=C_a$. If $|S|=abc+bc-1$ and $1\notin\Pi_{abc}(S)$, then the following assertions hold.
\begin{itemize}
\item[(i)] For every $z\in Z$, we have $v_z(S)\leq bc-2$.
\item[(ii)] There do not exist $ca+a-2$ pairwise disjoint subsequences of $S$, each of length $b$ and each having a selected product in $Z$.
\end{itemize}
\end{lemma}

\begin{proof}
The primes $a$, $b$, and $c$ are distinct, and $a$ is the smallest. Every subgroup of order $bc$ has trivial image under the projection onto $C_a$, since its image has order dividing both $bc$ and $a$. Thus the only subgroup of index $a$ is the non-abelian direct factor $C_c\rtimes C_b$. In particular, $G$ has no cyclic subgroup of index $a$. By Lemma~\ref{lem:davenport-bounds}(iii), we have $\mathsf{d}(G)\leq bc-1$. Lemma~\ref{lem:central-multiplicity}, applied with $d=bc-1$, proves (i). The estimate $\mathsf{d}(G)\leq bc-1$ also follows from \cite[Theorem~2]{KarthickBabuBeraGhoshSury2027}, since $G\cong C_{ac}\rtimes C_b$ with an action of order $b$, and that theorem gives $\mathsf{d}(G)=ac+b-2$.

For (ii), suppose that such subsequences exist, and let $W$ be their sequence of selected products in $Z$. Then
$$ |W|=(c+1)a-2 \qquad\text{and}\qquad 1\notin\Pi_{ca}(W). $$
The second assertion follows because $ca$ of the selected subsequences would concatenate to $abc$ terms with product one. By the inverse theorem in Lemma~\ref{lem:cyclic-inverse}, the sequence $W$ contains $x^{[a-1]}\cdot y^{[2a-1]}$, where $\langle xy^{-1}\rangle=Z$. Indeed, its two multiplicities have the form $j_1a-1$ and $j_2a-1$, with $j_1+j_2=c+1$, so one is at least $2a-1$ and the other is at least $a-1$.

Choose a subsequence $U\mid W$ containing these terms and satisfying
$$ |U|=ca+a-c-1=(a-1)(c+1). $$
Such a choice is possible: $c\geq2b+1$ implies that $|U|\geq3a-2$, and clearly $|U|\leq|W|$. By the second part of Lemma~\ref{lem:cyclic-inverse},
$$ \Pi_{ca-c}(U)=Z. $$
Consequently, the union of the original subsequences indexed by $U$ can realize every $z\in Z$ using exactly $b(ca-c)=bc(a-1)$ terms.

The number of terms outside all subsequences indexed by $U$ is
$$ abc+bc-1-b|U|=2bc-ab+b-1. $$
By Lemma~\ref{lem:gao-identities}(ii), this is at least $bc+c+b-2=\mathsf{E}(G/Z)$. Indeed, the difference is $(b-1)c-ab+1>0$, because $c\geq2b+1$ and $a<b$. Therefore the remaining terms contain a subsequence of length $bc$ with a selected product $z\in Z$. The subsequences indexed by $U$ supply $bc(a-1)$ other terms with product $z^{-1}$. Concatenation yields $abc$ terms with product one, a contradiction.
\end{proof}

\begin{lemma}\label{lem:orbit-extremal}
Let $b$ and $c$ be odd primes with $c=bt+1$ and $t\geq2$, and let $\alpha$ be an automorphism of $N\cong C_c$ of order $b$. Define
$$ O(u)=\{\alpha^j(u):0\leq j<b\}. $$
For an integer $a\geq1$, let
$$ A_i=\prod_{j=1}^{a}O(u_{ij}) \qquad\text{and}\qquad P=\Pi_{c-1}(A_1\bm\cdot\ldots\bm\cdot A_\ell). $$
Here $0\leq i\leq\ell$ and $\ell\geq c+t-2$. Suppose that $1\notin A_0P$ and that at most $c-1$ of the sets $A_1,\ldots,A_\ell$ are $\{1\}$. Then $\ell=c+t-2$, precisely $c-1$ of the sets $A_0,\ldots,A_\ell$ equal $\{1\}$, and each of the other $t$ sets avoids $1$ and has exactly one nonidentity orbit factor. In particular, at least
$$ (c-1)a+t(a-1) $$
of the elements $u_{ij}$ equal $1$.
\end{lemma}

\begin{proof}
A product of nonempty sets has cardinality at least that of each factor. Consequently $A_i$ is a singleton if and only if all its factors are $\{1\}$. Every nonsingleton is $\alpha$-invariant and contains at least $b$ nonidentity elements. If it has two nonidentity orbit factors, then the Cauchy--Davenport theorem gives $|A_i|\geq2b-1$. If $A_i$ avoids $1$, then orbit invariance improves this lower bound to $2b$. If $A_i$ contains $1$, then it improves the bound to $2b+1$. A nonsingleton containing $1$ has at least two nonidentity orbit factors, since one orbit alone avoids $1$. Such a nonsingleton therefore has at least $2b$ nonidentity elements.

Since $1\notin A_0P$, we have $P\neq N$, so the stabilizer of $P$ is trivial. No nonidentity element can occur in $c$ or more of the indexed sets $A_1,\ldots,A_\ell$. Otherwise every element of its orbit, which has $b\geq3$ elements, would occur in at least $c$ indexed sets. In the DGM bound with $k=c-1$, these elements alone would contribute $b(c-1)$, giving $|P|\geq2-c+b(c-1)\geq c$, a contradiction. Thus every nonidentity element has incidence at most $c-1$ and contributes its full incidence to the bound.

Let $z$ count the singleton sets among $A_1,\ldots,A_\ell$, and let $q$ count the nonsingletons containing $1$. Since $z\leq c-1$ and the identity has incidence $z+q$, its contribution is $\min\{c-1,z+q\}\geq z$. The preceding orbit counts therefore imply that
\begin{equation} |P|\geq2-c+z+b(\ell-z+q). \tag{6.3}\label{eq:orbit-incidence-bound} \end{equation}
Write $\ell=c+t-2+d$ and $z=c-1-s$, where $d,s\geq0$. Since $c=bt+1$, the right-hand side of \eqref{eq:orbit-incidence-bound} is
$$ c-b+bd+(b-1)s+bq. $$
If $d\geq1$, $q\geq1$, or $s\geq2$, then this expression is at least $c$. Hence
$$ \ell=c+t-2,\qquad q=0, \qquad\text{and}\qquad z\in\{c-2,c-1\}. $$

The sets $A_0$ and $P^{-1}$ are disjoint, so $|A_0|+|P|\leq c$. If $z=c-2$, then inequality~\eqref{eq:orbit-incidence-bound} implies that $|P|\geq c-1$, and hence $A_0=\{1\}$. If $z=c-1$, then the singleton sets show that $1\in P$, so $A_0$ avoids $1$. Thus precisely $c-1$ of all the sets are singletons, and the other $t$ sets avoid $1$.

Two nonidentity orbit factors in any $A_i$ with $i\geq1$ would increase its contribution to \eqref{eq:orbit-incidence-bound} by at least $b$, giving $|P|\geq c$. If $A_0$ had two such factors, then $z=c-1$, $|A_0|\geq2b$, and $|P|\geq c-b$, contrary to $|A_0|+|P|\leq c$. Hence every nonsingleton has exactly one nonidentity orbit factor. Each singleton supplies $a$ identity factors, and each of the other $t$ sets supplies $a-1$, as required.
\end{proof}

\begin{proposition}\label{prop:frobenius-direct-product}
Let $a<b<2a$ be odd primes, and let $c$ be a prime such that $b\mid c-1$. If the action in $C_c\rtimes C_b$ is nontrivial, then
$$ \mathsf{E}\bigl(C_a\times(C_c\rtimes C_b)\bigr)\leq abc+bc-1. $$
\end{proposition}

\begin{proof}
Let $G=C_a\times(C_c\rtimes C_b)$, $N=C_c$, $Z=C_a$, and $H=N\times Z\cong C_{ac}$. Suppose that $S$ has length $L=abc+bc-1$ and that $1\notin\Pi_{abc}(S)$. Lemma~\ref{lem:central-extraction-cap} applies. Since $b$ and $c$ are odd and $b\mid c-1$, we have $c\geq2b+1$. By Lemma~\ref{lem:repeated-extraction}, applied to the cyclic subgroup $H$,
$$ Q:=|S_{G\setminus H}| \geq bc+ac-\mathsf{E}(H)=(b-a)c+1\geq b. $$
We distinguish two cases.

\medskip\noindent\textbf{Case 1.} We have $Q\leq ab+(b-1)^2-1$.\nopagebreak[4]

Successively remove nonempty subsequences of the terms outside $H$ whose images in $G/H\cong C_b$ are minimal product-one sequences. When this is no longer possible, the remainder has length $e\leq b-1$, since $\mathsf{d}(C_b)=b-1$. All extracted subsequences have products in $H$. Let
$$ w=Q-e,\qquad \ell=L-Q=|S_H|, \qquad\text{and}\qquad v=abc-w. $$
The concatenation of the extracted subsequences has length $w$, and
$$ \ell-v=bc-1-e\geq bc-b. $$
Also,
$$ v\geq abc-ab-b^2+2b>bc. $$
For the last inequality, subtract $bc$ and use $c\geq2b+1$ to obtain the positive lower bound $b((2a-3)b+1)$. In particular, $\ell\geq v\geq bc$.

Fix orderings of all extracted subsequences, and let $T$ be their concatenation, with product $d\in H$ and final term $g$. Write $S_H=h_1\bm\cdot\ldots \bm\cdot h_\ell$, where $h_j=z_ju_j$, $z_j\in Z$, and $u_j\in N$. Since $g\notin H$, its conjugation action on $N$ is nontrivial. Write $g^{-1}ug=u^\lambda$, and define
$$ B_j=z_j\{u_j,u_j^\lambda\} \qquad\text{and}\qquad F=\Pi_v(B_1\bm\cdot\ldots \bm\cdot B_\ell)\subseteq H. $$
The prefixes $1$ and $dg^{-1}$ of $T$ induce the identity and conjugation by $g^{-1}$ on $H$, respectively. By the first assertion of Lemma~\ref{lem:orbit-insertion},
\begin{equation} dF\subseteq\Pi_{abc}(S). \tag{6.4}\label{eq:hall-pair-insertion} \end{equation}
It remains to prove the following claim.

\medskip\noindent\textbf{Claim 1.} We have $F=H$.\nopagebreak[4]

\begin{proof}[Proof of Claim~1]
Let $M$ be the stabilizer of $F$, and let $x=|M|$. For an $M$-coset $R$, let $c_R$ count the sets $B_j$ that meet $R$. Let $J$ count the sets $B_j$ contained in $ZM$. If $M=H$, then the claim follows. Suppose, therefore, that $M<H$, so $x\in\{1,a,c\}$.

The two points of $B_j$ either lie in a single coset in $ZM/M$ or lie in two distinct cosets outside $ZM/M$. Indeed, if $M=N$, then $ZM=H$, and both points have the same $Z$-component. If $M=Z$ or $M=1$, then the two points have the same image modulo $M$ if and only if $u_j=1$, because $\lambda-1$ acts invertibly on $N$. Therefore
$$ \sum_R c_R=2\ell-J. $$
Let $\mu$ be the number of cosets with $c_R>v$. If $\mu=0$, then, by Lemma~\ref{lem:setpartition-conseq}(i),
$$ |F|\geq x(1-v+2\ell-J) \geq x(1+\ell-v)\geq ac, $$
since $bc-b+1\geq ac$. If $\mu\geq2$, then the same lemma implies that $|F|\geq x(v+1)\geq ac$. Either conclusion contradicts $F\neq H$.

Suppose, then, that $\mu=1$, and let $R$ be the unique coset with $c_R>v$. By Lemma~\ref{lem:setpartition-conseq}(i),
$$ |F|\geq x(1+2\ell-J-c_R). $$
If $R\notin ZM/M$, then the sets counted by $c_R$ and $J$ are disjoint, so $c_R+J\leq\ell$ and $|F|\geq x(\ell+1)\geq ac$. If $R\in ZM/M$ and $x=1$, then $R=\{z\}$ for some $z\in Z$, and $B_j$ meets $R$ if and only if $h_j=z$. Hence
$$ v_z(S)\geq c_R\geq v+1\geq bc+1, $$
contrary to Lemma~\ref{lem:central-extraction-cap}(i).

It remains to consider $R\in ZM/M$ and $x\in\{a,c\}$. The corresponding original terms $h_j$ lie in the same coset $R$. We must have
$$ c_R\leq abc+x-2. $$
Otherwise, take $abc+x-1$ such terms and translate them into the cyclic group $M$. Since $x\mid abc$, there are $abc$ translated terms with product one by Lemma~\ref{lem:repeated-extraction}. Translating back preserves this product, because every element of $H$ has order dividing $ac$, and hence dividing $abc$. This would contradict the choice of $S$.

Using $J\leq\ell$, the bound on $c_R$, and $\ell\geq abc+bc-ab-b^2+2b-1$, we obtain
$$ |F|\geq f(x):=x(bc-ab-b^2+2b+2-x). $$
Let $B_0=(b-1)c-ab-b^2+2b+2-a$. Direct calculation yields
$$ f(a)-ac=aB_0 \qquad\text{and}\qquad f(c)-ac=cB_0. $$
Moreover, $c\geq2b+1$ implies that $B_0\geq(b-a)(b+1)+1>0$. Thus $|F|>ac$, another contradiction. This proves that $F=H$.
\end{proof}

By Claim~1 and Lemma~\ref{lem:orbit-insertion}, applied as above, we have $1\in\Pi_{abc}(S)$, contrary to the assumption. This completes Case~1.

\medskip\noindent\textbf{Case 2.} We have $Q\geq ab+(b-1)^2$.\nopagebreak[4]

Choose $3ab-1$ terms of $S$ containing as many terms outside $H$ as possible. By Lemma~\ref{lem:repeated-extraction}, applied in $G/N\cong C_{ab}$, these terms contain a subsequence $T$ of length $2ab$ with product-one image modulo $N$. At most $ab-1$ of the chosen terms were discarded, so $T$ contains at least
$$ \min\{Q-ab+1,2ab\}\geq(b-1)^2+1 $$
terms outside $H$. Here $2ab>(b-1)^2$ follows from $b<2a$. Among the $b-1$ nonidentity $H$-cosets, one therefore contains at least $b$ terms of $T$. Choose a subsequence $B\mid T$ of length $b$ in that coset.

\medskip\noindent\textbf{Claim 2.} The sequence $T$ is a product of two subsequences of length $ab$ with selected products in $N$, one of which contains $B$.\nopagebreak[4]

\begin{proof}[Proof of Claim~2]
Retain $B$ as one subsequence, and factor the rest into $b$-term subsequences with product-one images modulo $H$. Repeated application of the cyclic Erd\H{o}s--Ginzburg--Ziv theorem provides all but the last of these subsequences. The last also has product-one image, because the whole sequence $T$ does. There are $2a$ subsequences in total.

Their selected products lie in $H$, and their $Z$-components have total product one, since $T$ has product-one image in the abelian quotient $G/N$. The cyclic Erd\H{o}s--Ginzburg--Ziv theorem in $Z\cong C_a$ selects $a$ components with product one. The complementary $a$ components then also have product one. Grouping the corresponding subsequences gives two $ab$-term subsequences with selected products in $N$, one containing $B$.
\end{proof}

Starting with the two subsequences in Claim~2, continue extracting $ab$-term subsequences with products in $N$ until at most $2ab-2$ terms remain. Denote the subsequence containing $B$ by $T_0$, and denote the others by $T_1,\ldots,T_\ell$. Let $t=(c-1)/b\geq2$. Counting terms gives
$$ \ell\geq c+\left\lfloor\frac{c}{a}\right\rfloor-2 \geq c+t-2. $$
The first inequality follows from
$$ (\ell+1)ab\geq L-(2ab-2)=abc+bc-2ab+1. $$
Since $c$ and $a$ are distinct primes, rounding gives the asserted expression.

Refine each $T_i$ into $b$-term subsequences as in Claim~2, retaining $B$ in $T_0$. Write the $a$ selected products associated with $T_i$ as $z_{ij}u_{ij}\in H$, where $z_{ij}\in Z$ and $u_{ij}\in N$. For every $i$, we have
$$ \prod_{j=1}^{a}z_{ij}=1, $$
because $T_i$ has product-one image modulo $N$. Let $\alpha$ generate the action of $G/H$ on $N$, and define
$$ O(u)=\{\alpha^j(u):0\leq j<b\}. $$
Also define
$$ A_i=\prod_{j=1}^{a}O(u_{ij}) \qquad\text{and}\qquad P=\Pi_{c-1}(A_1\bm\cdot\ldots\bm\cdot A_\ell). $$
The prefixes of $B$ represent all $H$-cosets. By Lemma~\ref{lem:orbit-insertion}, every independent orbit choice from the $b$-term subsequences, including $B$, can therefore be realized. Conjugation fixes $Z$, whose components multiply to one in each $T_i$. Taking $T_0$ and $c-1$ other subsequences $T_i$ uses $cab$ terms and gives
\begin{equation} A_0P\subseteq\Pi_{abc}(S). \tag{6.5}\label{eq:hall-orbit-insertion} \end{equation}

If $A_i=\{1\}$, then all its orbit factors are $\{1\}$, and $T_i$ itself has a selected product one. Thus at most $c-1$ of $A_1,\ldots,A_\ell$ are singletons. Otherwise, $c$ such subsequences would already yield $abc$ terms with product one. The inclusion above implies that $1\notin A_0P$, so Lemma~\ref{lem:orbit-extremal} applies. By that lemma, at least
$$ (c-1)a+t(a-1)\geq ca+a-2 $$
of the factors $u_{ij}$ are identities. The difference between the two expressions is $(t-2)(a-1)\geq0$. The corresponding original $b$-term subsequences are pairwise disjoint and have selected products $z_{ij}\in Z$. This contradicts Lemma~\ref{lem:central-extraction-cap}(ii) and completes the proof.
\end{proof}

\section{Cyclic kernels with quotients of prime-square order}\label{sec:prime-square-quotients}

We exclude configurations (S4) and (S5) of Theorem~\ref{thm:solvable-configs}. Throughout this section, we use the subgroup concentration inequality from Lemma~\ref{lem:repeated-extraction}.

\subsection{An elementary abelian quotient}\label{sec:elem-abelian-quotient}

\begin{lemma}\label{lem:pair-centralization}
Let $B\mathbin{\trianglelefteq} G$ be cyclic of order $m\geq3$, with $G/B$ abelian. Suppose that $1\notin\Pi_{4m}(S)$ and that $S$ contains $3m-2$ disjoint two-term subsequences whose products belong to $B$. Then every term in those subsequences centralizes $B$.
\end{lemma}

\begin{proof}
Choose one product from each pair to form a sequence $W$ over $B$. Since $1\notin\Pi_{2m}(W)$, we have, by Lemma~\ref{lem:cyclic-inverse},
$$ W=u^{[m-1]}\bm\cdot v^{[2m-1]} \qquad\text{and}\qquad \langle uv^{-1}\rangle=B, $$
after interchanging $u$ and $v$ when necessary. Reversing one pair cannot change its product. A change from $u$ to $v$ would give $2m$ copies of $v$, while a change from $v$ to $u$ would give $m$ copies of each. Both changes would yield a product-one subsequence of length $2m$. Any other change would leave three distinct products, because $m-1\geq2$, contrary to Lemma~\ref{lem:cyclic-inverse}. Thus the terms of each pair commute.

Take a pair with product $u$ and a pair with product $v$, and let $T$ be their union. The other selected products form $W'=u^{[m-2]}\bm\cdot v^{[2m-2]}$, and
$$ \Pi_{2m-2}(W') =v^{-2}\{(uv^{-1})^j:0\leq j\leq m-2\} =B\setminus\{(uv)^{-1}\}. $$
Every ordering of $T$ has product in $B$. A product other than $uv$ could therefore be cancelled using $2m-2$ other pairs. Consequently $\pi(T)=\{uv\}$.

Moving an initial term to the end of an ordering shows that every term of $T$ centralizes $uv$. A term in the pair with product $u$ also centralizes $u$, and hence it centralizes $v$ and the generator $uv^{-1}$ of $B$. The pair with product $v$ is treated in the same way. Matching every pair with a pair having the other product proves the assertion.
\end{proof}

\begin{lemma}\label{lem:affine-concentration}
Let $p$ be an odd prime, let $B\mathbin{\trianglelefteq} G$ be cyclic of order $m\geq5$, and suppose that $G/B\cong C_p^2$. Let $S$ be a sequence satisfying
$$ |S|=p(p+1)m-1 \qquad\text{and}\qquad 1\notin\Pi_{p^2m}(S). $$
Every $p$-term subsequence with products in $B$ has at most two products, and every $2p$-term subsequence with products in $B$ has at most four. Suppose, in addition, that
$$ \rho:G\longrightarrow\mathbb{F}_q\rtimes C_p $$
is an epimorphism, where $q$ is prime, the scalar action is faithful, and $\rho(B)$ consists of translations. If a $B$-coset with nontrivial scalar component contains at least $p$ terms of $S$, then there is a subgroup $H$ of index $q$ such that every $B$-coset containing at least $p$ terms has at most one term outside $H$. In particular,
\begin{equation} |S_{G\setminus H}|\leq p^2(p-1). \tag{7.1}\label{eq:affine-outside-cap} \end{equation}
\end{lemma}

\begin{proof}
Let $k=pm$ and $R=k+m-3$. Starting with any prescribed $p$-term subsequence having product-one image in $G/B$, use Lemma~\ref{lem:reiher} to extract $R$ disjoint subsequences of this kind. Before the last choice, there are $4p-1\geq4p-3$ terms. If $A_i$ is the product set of the $i$th subsequence, then $1\notin\Pi_k(A_1\bm\cdot\ldots \bm\cdot A_R)$. By Lemma~\ref{lem:setpartition-conseq}(ii),
$$ \sum_i(|A_i|-1)\leq1. $$
This proves the first bound.

After removing a prescribed $2p$-term subsequence $T$, the same argument extracts $R-2$ subsequences of length $p$. Their selected products form a sequence $W$ of length $k+m-5$ with $1\notin\Pi_k(W)$. Embed $B=C_m$ in $C_k$. By Lemma~\ref{lem:abelian-products}, applied with $a=k$ and $b=m-3$, we have $|\Pi_{k-2}(W)|\geq m-4$. If $|\pi(T)|\geq5$, then its inverse set meets $\Pi_{k-2}(W)$ inside $B$. This produces $(k-2)p+2p=p^2m$ terms with product one, a contradiction. The second bound follows.

Write an affine map as $z\mapsto az+t$. A map with $a\neq1$ has order $p$ and a unique fixed point. Its centralizer is the subgroup of order $p$ fixing that point. In particular, two commuting maps with the same nonidentity scalar component are equal.

Consider any $p$ terms in a fixed $B$-coset with scalar component $a\neq1$. Cyclic rotation of an ordering whose image is a nonidentity translation would give $p$ distinct products, because the prefixes realize all powers of $a$. The bound of two products therefore forces every ordering to have identity image. Adjacent transpositions show that any two of the images commute. Hence all terms in any such coset containing at least $p$ terms have a common image of order $p$.

Choose $p$ terms $U$ from the coset in the hypothesis, with common image $g$, and let $T=\langle g\rangle$. Consider another coset with at least $p$ terms. If its scalar component is nontrivial, then its common image is some element $h$ of order $p$. If $h\notin T$, then the commutator $d=ghg^{-1}h^{-1}$ is a nonidentity translation. Using $p$ copies of $g$ and $p$ copies of $h$, the orderings
$$ ghg^{p-1}h^{p-1},\qquad hgh^{p-1}g^{p-1}, \qquad\text{and}\qquad g^ph^p $$
have products $d$, $d^{-1}$, and $1$, respectively. The prefixes of the first two orderings realize all scalar components, so their cyclic rotations give the complete scalar orbits of $d$ and $d^{-1}$. These orbits are disjoint, since a scalar group of odd order does not contain $-1$. There would therefore be at least $2p+1>4$ products, a contradiction. Thus the common image lies in $T$.

If the other coset has trivial scalar component, then two nonidentity translation images would also be impossible. To see this, choose $p$ terms including the two in question and insert them among the terms of $U$. By Lemma~\ref{lem:orbit-insertion}, the product of their scalar orbits can be realized. Two of these orbits have size $p$, so the Cauchy--Davenport theorem gives at least $\min\{q,2p-1\}\geq5$ products. Here $q\equiv1\pmod{p}$ and $q\geq2p+1$. This again contradicts the bound of four products.

Consequently $H=\rho^{-1}(T)$ has the stated property and index $q$. Each of the $p^2$ cosets contributes at most $p-1$ terms outside $H$, which proves \eqref{eq:affine-outside-cap}.
\end{proof}

\begin{proposition}\label{prop:elem-p2-quotient}
Let $p$ be the smallest prime divisor of $|G|$. Suppose that $B\mathbin{\trianglelefteq} G$ is cyclic and $G/B\cong C_p^2$. If $G$ has no cyclic subgroup of index $p$, then
$$ \mathsf{E}(G)\leq |G|+\frac{|G|}{p}-1. $$
\end{proposition}

\begin{proof}
By Proposition~\ref{prop:nilpotent-bound}, we may assume that $G$ is not nilpotent. Let $m=|B|$ and $n=p^2m$, and suppose that $|S|=n+pm-1$ but $1\notin\Pi_n(S)$.

\medskip\noindent\textbf{Case 1.} The prime $p$ is $2$.\nopagebreak[4]

We have $m\geq3$. Otherwise $B\leq Z(G)$ and $G'\leq B$ would imply that $G$ is nilpotent. A maximal family of pairs, each contained in a single $B$-coset, leaves one or three terms. One remaining term would give $3m-1=\mathsf{s}_{2m}(B)$ selected products, already a contradiction. Thus there are $3m-2$ pairs. By Lemma~\ref{lem:pair-centralization}, all but three terms lie in $K=C_G(B)$.

We have $K<G$, again because $G'\leq B$. If $[G:K]=4$, then $K=B$. By Lemma~\ref{lem:repeated-extraction}, at least $m+1\geq4$ terms lie outside $K$. If $[G:K]=2$, then $B\leq Z(K)$ and $K/B$ is cyclic, so $K$ is abelian. By hypothesis, $K$ is non-cyclic, and hence $K\cong C_2\times C_m$, where $m\geq4$ is even. By Lemmas~\ref{lem:davenport-bounds}(i) and~\ref{lem:gao-identities}(i), we have $\mathsf{E}(K)=3m$. By Lemma~\ref{lem:repeated-extraction}, at least $m\geq4$ terms then lie outside $K$. Both possibilities contradict the fact that at most three terms lie outside $K$.

\medskip\noindent\textbf{Case 2.} The prime $p$ is odd.\nopagebreak[4]

Let $A$ be the Hall $p'$-subgroup of $B$. It is normal, so the Schur--Zassenhaus theorem gives $G=A\rtimes P$ for a Sylow $p$-subgroup $P$. Thus
$$ K=C_G(A)=A\times C_P(A) $$
is nilpotent, contains $B$, and has index $p$ or $p^2$. On every primary component of $A$, the action factors through $C_p^2$ and has cyclic image, hence order at most $p$. Call a prime $q\mid |A|$ active when this image is nontrivial. By Lemma~\ref{lem:group-structure}(viii), the induced action on $C_q$ also has order $p$, so $q\equiv1\pmod{p}$ and $q\geq2p+1$.

Moreover, $m>q$. If $m=q$, then the kernel of the action of $P\cong C_p^2$ would have order $p$, making $K\cong C_q\times C_p$ a cyclic subgroup of index $p$, a contradiction. Since $p$ is the smallest prime divisor of $|G|$, it follows that
\begin{equation} m\geq pq\geq p(2p+1). \tag{7.2}\label{eq:active-prime-size} \end{equation}
This holds for every active prime $q$.

\medskip\noindent\textbf{Claim.} Each $B$-coset outside $K$ contains at most $p-1$ terms of $S$.\nopagebreak[4]

\begin{proof}[Proof of the claim]
Suppose that $x\notin K$ and $|S_{xB}|\geq p$. Choose a primary component on which $x$ acts nontrivially, and project it onto $C_q$. Factoring out the other components of $A$ and taking the image of the action of $P$ gives an epimorphism $\rho:G\rightarrow C_q\rtimes C_p$ to which Lemma~\ref{lem:affine-concentration} applies. By that lemma, the resulting subgroup $H$ of index $q$ has at most $p^2(p-1)$ terms of $S$ outside it. However, the universal bound $\mathsf{E}(H)\leq2|H|-1$, Lemma~\ref{lem:repeated-extraction}, and \eqref{eq:active-prime-size} imply that
$$ |S_{G\setminus H}| \geq pm-\frac{p^2m}{q}+1 \geq p^2(q-p)+1>p^2(p-1), $$
a contradiction.
\end{proof}

By the claim, at most $(p^2-1)(p-1)$ terms lie outside $K$. Whether $K=B$ or $[K:B]=p$, the cyclic bound or the nilpotent bound in Lemma~\ref{lem:gao-bounds}(iii), respectively, implies that
$$ \mathsf{E}(K)\leq |K|+m+p-2. $$
In the latter case, $K$ is non-cyclic by hypothesis. Consequently Lemma~\ref{lem:repeated-extraction} and \eqref{eq:active-prime-size} require
$$ |S_{G\setminus K}| \geq(p-1)m-p+2>(p^2-1)(p-1), $$
which is the final contradiction.
\end{proof}

\subsection{Cyclic quotients and coprime kernels}\label{sec:cyclic-coprime-quotient}

\begin{lemma}\label{lem:p2-orbit-saturation}
Let $p$ be prime, let $N=C_m$ with $\gcd(m,p)=1$, and let $\alpha$ have order $p^2$ and satisfy $C_N(\alpha)=1$. For a sequence $W=u_1\bm\cdot\ldots\bm\cdot u_\ell$ over $N$, define
$$ A_j=\{\alpha^e(u_j):0\leq e<p^2\}. $$
If $\ell\geq m+(m-1)/p-2$, then either
$$ 1\in\Pi_m(W) \qquad\text{or}\qquad \Pi_{m-1}(A_1\bm\cdot\ldots\bm\cdot A_\ell)=N. $$
\end{lemma}

\begin{proof}
On every primary component, the action has order $p$ or $p^2$. By Lemma~\ref{lem:group-structure}(viii), we have $q\equiv1\pmod{p}$ for every prime $q\mid m$, and $q_0\equiv1\pmod{p^2}$ for the prime associated with at least one component. Thus
$$ m\equiv1\pmod{p} \qquad\text{and}\qquad m\geq q_0\geq p^2+1\geq2p+1. $$

Assume that $1\notin\Pi_m(W)$. At least one $A_j$ has size $p^2$. Otherwise all $u_j$ lie in $D=C_N(\alpha^p)$, whose order is at most $m/q_0$, and
$$ \ell\geq m+|D|-1. $$
Indeed,
$$ \frac{m-1}{p}-1\geq\frac{m}{q_0}, $$
because $(q_0-p)m\geq q_0(q_0-p)\geq q_0(p+1)$. By Lemma~\ref{lem:repeated-extraction}, applied in the cyclic group $D$, we would have $1\in\Pi_m(W)$, a contradiction.

Let $F=\Pi_{m-1}(A_1\bm\cdot\ldots\bm\cdot A_\ell)$, and suppose that $F\neq N$. This set is nonempty, since $\ell\geq m-1$. Write $M=H(F)$ and $h=|M|$. Cyclicity makes $M$ characteristic, so each $A_j$ is either contained in $M$ or disjoint from $M$. Coprime action gives $C_{N/M}(\alpha)=1$, and every nonidentity orbit on $N/M$ has size at least $p$.

If $i$ of the sets lie in $M$, then $i\leq m+h-2$ by Lemma~\ref{lem:repeated-extraction}. Since $h(1+(p-1)(m-1))\geq m$, we have, by Lemma~\ref{lem:setpartition-conseq}(iii),
\begin{equation} |F|\geq h\bigl(2-m+\min\{m-1,i\}+p(\ell-i)\bigr) \geq h(m-ph). \tag{7.3}\label{eq:p2-stabilizer-bound} \end{equation}
For the last inequality, the expression in parentheses decreases as $i$ increases. Substitute $i=m+h-2$ and the lower bound on $\ell$.

If $h>1$, then every prime divisor of $h$ and of $m/h$ is at least $p+1$. Hence $h\geq p+1$ and $m\geq(p+1)h$, so
$$ h(m-ph)-m =(h-1)m-ph^2 \geq h(h-p-1)\geq0. $$
This contradicts $F\neq N$.

If $h=1$, then $i\leq m-1$. By Lemma~\ref{lem:setpartition-conseq}(iii), every nonidentity element has incidence at most $m-1$. The identity has incidence $i\leq m-1$ as well. Thus every summand in the DeVos--Goddyn--Mohar bound equals the corresponding incidence. Using the orbit of size $p^2$ found above, we obtain
$$ \begin{aligned} |F| &\geq2-m+\sum_j|A_j|\\ &\geq2-m+p\ell-(p-1)i+(p^2-p)\\ &\geq m-p+(p^2-p)\geq m. \end{aligned} $$
This gives the same contradiction and proves the lemma.
\end{proof}

\begin{lemma}\label{lem:p2-distinct-prefixes}
Let $p$ be prime, and let $W$ be a sequence over $\mathbb{Z}/p^2\mathbb{Z}$ with no $p^2$-term subsequence admitting a zero-sum ordering whose proper partial sums, including zero, are pairwise distinct. Then
$$ \sum_{a\equiv c\, (p)}\left\lfloor\frac{v_a(W)}{p}\right\rfloor\leq p \qquad\text{and}\qquad \sum_{a\equiv c\, (p)}v_a(W)\leq2p^2-p. $$
These inequalities hold for every nonzero $c\in\mathbb{Z}/p\mathbb{Z}$.
\end{lemma}

\begin{proof}
If the first inequality fails, then select $p+1$ terms from the auxiliary sequence having $\lfloor v_a(W)/p\rfloor$ copies of each $a\equiv c\pmod{p}$. Some $p$ of the selected terms have nonzero sum modulo $p^2$. Otherwise, comparison of all single-term deletions would force the $p+1$ terms to equal an element $a$ with $pa=0$, contrary to $c\neq0$.

Order the selected $p$ terms, whose sum is $pu$ with $p\nmid u$, and repeat this ordering $p$ times. This construction uses available terms of $W$. If $s_j$ is the sum of the first $j$ terms of one copy of the ordering, then the proper partial sums are
$$ kpu+s_j. $$
Here $0\leq k,j<p$. Equality of two such sums first gives $jc=j'c$ modulo $p$, and hence $j=j'$. It then gives $k=k'$. Thus these proper partial sums are distinct, while the total sum is zero, a contradiction.

Write $v_a=p\lfloor v_a/p\rfloor+\lambda_a$, where $0\leq \lambda_a<p$. Summing over the $p$ residues above $c$ gives the second inequality.
\end{proof}

\begin{proposition}\label{prop:cyclic-p2-quotient}
Let $p$ be the smallest prime divisor of $|G|$, and let
$$ G=C_m\rtimes_\lambda C_{p^2} =\langle a,b\mid a^m=b^{p^2}=1,\;bab^{-1}=a^\lambda\rangle. $$
Assume that $\gcd(m,p)=1$ and $\ord_m(\lambda)=p^2$. Then
$$ \mathsf{E}(G)\leq(p^2+p)m-1. $$
\end{proposition}

\begin{proof}
Let $N=\langle a\rangle$ and $A=\langle b\rangle$. Define $D=C_N(A)$, $N_0=[N,A]$, and $K=N_0\rtimes A$. By Lemma~\ref{lem:group-structure}(viii),
$$ N=D\times N_0 \qquad\text{and}\qquad G=D\times K. $$
The action on $N_0$ is faithful, and $C_{N_0}(A)=1$. Once the assertion is proved for $K$, we obtain, by Lemma~\ref{lem:lifting-bounds}(i),
$$ \mathsf{E}(G)+1\leq |D|\bigl(\mathsf{E}(K)+1\bigr)\leq(p^2+p)m. $$
We may therefore assume that $C_N(A)=1$.

As in the proof of Lemma~\ref{lem:p2-orbit-saturation}, we have $m\equiv1\pmod{p}$, and there is a prime $q_0\leq m$ such that $q_0\equiv1\pmod{p^2}$. Thus $m\geq p^2+1$, and $m\geq2p^2+1$ when $p$ is odd. Suppose that $|S|=(p^2+p)m-1$ and $1\notin\Pi_{p^2m}(S)$. Identify $G/N$ with $\mathbb{Z}/p^2\mathbb{Z}$ by $bN\mapsto1$.

\medskip\noindent\textbf{Claim.} No $p^2$-term subsequence of $S$ has a zero-sum quotient ordering with distinct proper partial sums.\nopagebreak[4]

\begin{proof}[Proof of the claim]
If $T_0$ had such an ordering, then let its product be $c\in N$. In $ST_0^{-1}$, extract a maximal family of $p^2$-term subsequences with product-one quotient image. At most $2p^2-2$ terms remain, so the number of extracted subsequences satisfies
$$ \ell\geq \left\lceil\frac{(p^2+p)m-3p^2+1}{p^2}\right\rceil =m+\frac{m-1}{p}-2. $$
Select products $u_j\in N$ from these subsequences. The sequence of selected products has no product-one subsequence of length $m$. By Lemma~\ref{lem:p2-orbit-saturation}, the products of $m-1$ of their conjugation orbits therefore fill $N$. The prefixes of $T_0$ represent every quotient coset. By Lemma~\ref{lem:orbit-insertion}, every element of $cN=N$ can be realized on $T_0$ and $m-1$ selected subsequences, using $p^2m$ terms. This contradiction proves the claim.
\end{proof}

Let $H=\langle a,b^p\rangle$. This subgroup is non-cyclic of order $pm$ and has a cyclic subgroup of index $p$. By Lemma~\ref{lem:gao-identities}(ii),
$$ \mathsf{E}(H)=(p+1)m+p-2. $$
Apply Lemma~\ref{lem:p2-distinct-prefixes} to the quotient image of $S$ and Lemma~\ref{lem:repeated-extraction} to $H$. These lemmas imply that
$$ (p-1)m-p+2\leq |S_{G\setminus H}| \leq(p-1)(2p^2-p). $$
Since $m$ is an integer, this forces $m\leq2p^2-p$. The preceding lower bounds on $m$ leave only $(p,m)=(2,5)$.

In this case, $|S|=29$. By Lemma~\ref{lem:repeated-extraction}, at least five terms lie outside $H$. The claim forbids four terms in either $bN$ or $b^3N$. Thus there are five or six terms outside $H$, and both cosets occur. Two terms in $b^2N$, together with one from each of the two outside cosets, would have quotient ordering $(1,2,3,2)$, with proper partial sums $0,1,3,2$. Therefore
$$ |S_{b^2N}|\leq1 \qquad\text{and}\qquad |S_N|\geq22. $$

Faithfulness implies that $C_G(N)=N$. A cyclic subgroup of order ten would contain and centralize $N$, so no such subgroup exists. By Lemmas~\ref{lem:davenport-bounds}(iii) and~\ref{lem:central-multiplicity}, we have $\mathsf{d}(G)\leq9$ and $v_1(S)\leq8$. Choose eighteen terms $U=u_1\bm\cdot\ldots\bm\cdot u_{18}\mid S_N$, including four nonidentity terms, and choose $x\mid S_{bN}$ and $y\mid S_{b^3N}$. By the first assertion of Lemma~\ref{lem:orbit-insertion}, applied to the ordered pair $(x,y)$ and its prefixes $1$ and $x$, we have
$$ \pi(x\bm\cdot y\bm\cdot U) \supseteq xy\prod_{j=1}^{18}\{u_j,xu_jx^{-1}\}=N. $$
Indeed, conjugation by $x$ has order four on $N=C_5$, so the four chosen nonidentity terms give two-element factors. By the Cauchy--Davenport theorem, their product is $C_5$, and the remaining factors do not diminish it. Thus twenty terms have product one, which is the final contradiction.
\end{proof}

\section{Product growth and compatible lengths}\label{sec:growth-compatibility}

Throughout this section and the next, let $n=|G|=2m$. Unless otherwise stated, $G$ is non-abelian simple. We shall prove that every sequence of length $3m-1$ has an $n$-term product-one subsequence. The lemmas below allow a selected subsequence to be extended to the required length.

\begin{lemma}\label{lem:simple-group-basics}
One has $n\geq 60$. Every proper subgroup has index at least five, and an index-five subgroup can occur only when $G\cong A_5$. Moreover,
$$ \ord(g)\mid m \qquad\text{and}\qquad \mathsf{d}(G)\leq m-1. $$
The divisibility assertion holds for every $g\in G$.
\end{lemma}

\begin{proof}
The odd-order theorem implies that $2\mid n$. A proper subgroup of index at most four would embed $G$ in a solvable symmetric group via its faithful coset action, which is impossible. Every permutation action of $G$ has trivial sign, since the sign is a homomorphism to $C_2$. In the regular action, an element of order $t$ has sign $(-1)^{n-n/t}=(-1)^{n/t}$, so $n/t$ is even and $t\mid m$. In particular, an involution implies that $4\mid n$.

Suppose that $n<60$. A nontrivial finite $2$-group has nontrivial center and is not non-abelian simple. The other multiples of four below sixty are $12,20,24,28,36,40,44,48,52,56$. Except for $40$ and $56$, each has a Sylow subgroup of index three or four, which has already been excluded. At order $40$, a Sylow $2$-subgroup of index five embeds $G$ in $A_5$, contradicting $40\nmid 60$. At order $56$, simplicity forces eight Sylow $7$-subgroups, which contain forty-eight distinct nonidentity elements. The remaining eight elements must constitute every Sylow $2$-subgroup, making that subgroup unique and normal. Thus $n\geq 60$.

An index-five action embeds $G$ in $A_5$, and the order bound now forces $G\cong A_5$. Finally, $G$ has no subgroup of index two, so by Lemma~\ref{lem:davenport-bounds}(iii), $\mathsf{d}(G)\leq m-1$.
\end{proof}
\begin{lemma}\label{lem:product-set-tools}
Let $F$ be a finite group of order $N$.
\begin{itemize}
\item[(i)] If $|A|+|B|>N$, then $AB=F$.
\item[(ii)] If $1\in A\cap B$ and the equation $ab=1$, with $a\in A$ and $b\in B$, has only the solution $a=b=1$, then $|AB|\geq |A|+|B|-1$.
\item[(iii)] Define $\lambda_A(x)=|Ax\setminus A|$. If $X\subseteq F\setminus\{1\}$ generates $F$ and $\varnothing\neq A\subsetneq F$, then some $x\in X$ satisfies
$$ \lambda_A(x)\geq \min\left\{\frac{\min(|A|,N-|A|)+1}{2}, \frac{|X\cup X^{-1}|+2}{4}\right\}. $$
\end{itemize}
\end{lemma}

\begin{proof}
For (i), the sets $A$ and $gB^{-1}$ intersect for each $g\in F$, so $g\in AB$. Part (ii) is \cite[Lemma~2.2]{QuGao2024}. Part (iii), when $|A|\leq N/2$, is \cite[Lemma~2.3]{QuGao2024}. For the general case, observe that
$$ \lambda_{F\setminus A}(x)=|A\setminus Ax|=|Ax\setminus A|=\lambda_A(x), $$
where the middle equality follows from $|Ax|=|A|$. Apply the cited lemma to the smaller of $A$ and its complement.
\end{proof}

The following lemma converts a generating difference set into a lower bound for product-set growth.

\begin{lemma}\label{lem:difference-growth}
If $\varnothing\neq A\subsetneq F$, $B\neq\varnothing$, and $\langle BB^{-1}\rangle=F$, then
\begin{equation} |AB|\geq |A|+ \min\left\{ \left\lceil\frac{\min(|A|,|F|-|A|)+1}{2}\right\rceil, \left\lceil\frac{|BB^{-1}|+1}{4}\right\rceil \right\}. \tag{8.1}\label{eq:difference-growth} \end{equation}
\end{lemma}

\begin{proof}
Apply Lemma~\ref{lem:product-set-tools}(iii) to the symmetric generating set $X=BB^{-1}\setminus\{1\}$. Write the resulting element $x$ as $uv^{-1}$, where $u,v\in B$. Right translation by $v^{-1}$ gives $|Au\cup Av|=|Ax\cup A|$. Since $Au\cup Av\subseteq AB$, the boundary estimate and integrality imply the required inequality.
\end{proof}

We next describe the coset action of an ordered product when no nonempty subsequence has a product in the subgroup.

\begin{lemma}\label{lem:coset-cycle}
Let $K\leq F$ have index $b\geq 2$. Suppose that a sequence $V$ of length $b-1$ has no nonempty subsequence with an ordered product in $K$. For every ordering $x_1,\ldots,x_{b-1}$ of $V$, its product $P=x_1\cdots x_{b-1}$ acts as a single $b$-cycle on the left cosets of $K$. Consequently, $F=\langle P\rangle K$.
\end{lemma}

\begin{proof}
Set $p_0=1$ and $p_i=x_1\cdots x_i$. If $p_iK=p_jK$ for $i<j$, then $x_{i+1}\cdots x_j=p_i^{-1}p_j\in K$, contrary to the hypothesis. Thus $p_0K,\ldots,p_{b-1}K$ exhaust the $b$ left cosets.

The same reasoning applies to any cyclic rotation of the word, since each subword still uses distinct indexed terms of $V$. Fix $1\leq j<b$, rotate after $x_j$, and multiply every prefix coset of the rotated word on the left by $p_j$. The resulting cosets are
$$ p_jK,p_{j+1}K,\ldots,p_{b-1}K,Pp_1K,Pp_2K,\ldots,Pp_jK. $$
They again exhaust the coset space without repetition. Delete the common cosets $p_jK,\ldots,p_{b-1}K$ from this list and the original prefix list. It follows that
$$ P\{p_1K,\ldots,p_jK\}=\{p_0K,\ldots,p_{j-1}K\}. $$
Taking the difference of the identities for $j$ and $j-1$ gives $Pp_jK=p_{j-1}K$; for $j=1$, the preceding identity already gives this equality. Also, $PK=p_{b-1}K$. Hence these cosets form one cycle under $P$. Every coset is consequently $P^iK$ for some $i$, proving that $F=\langle P\rangle K$.
\end{proof}

\begin{lemma}\label{lem:subgroup-extraction}
The following assertions hold.
\begin{itemize}
\item[(i)] If $|F:K|=b$, then any $b$ terms over $F$ contain a nonempty subsequence of length at most $b$ with an ordered product in $K$.
\item[(ii)] If, in addition, $F$ is non-solvable and $K$ is abelian, then $b-1$ terms suffice.
\item[(iii)] If $U\mid S$, $\Pi_j(U)=F$, $0\leq j\leq N=|F|$, and $|U|+N\leq |S|+j$, then $1\in\Pi_N(S)$.
\item[(iv)] If $H\leq F$, $h=|H|$, and $|S_H|\geq N+\mathsf{E}(H)-h$, then $1\in\Pi_N(S)$. In particular, $N+\mathsf{d}(H)$ terms in $H$ suffice when $\mathsf{E}(H)=h+\mathsf{d}(H)$ is known, and $N+h-1$ terms always suffice.
\end{itemize}
\end{lemma}

\begin{proof}
For (i), order the $b$ terms arbitrarily. Among the $b+1$ prefix products, two lie in the same left coset, and their intervening nonempty subword has product in $K$. If (ii) failed, then by Lemma~\ref{lem:coset-cycle}, $F=\langle P\rangle K$. This contradicts It\^o's theorem, according to which a product of two abelian subgroups is solvable \cite[Proposition~1.2]{Tsang2024}.

For (iii), the inequality gives $|S|-|U|\geq N-j$. Select $N-j$ terms outside $U$ with an ordered product $z$. Choose a $j$-term product $z^{-1}$ from $U$ and concatenate the two orderings.

For (iv), apply Lemma~\ref{lem:repeated-extraction} in $H$ to extract $N/h$ disjoint subsequences of length $h$ and concatenate them. The final assertions follow from $\mathsf{E}(H)=h+\mathsf{d}(H)$ when this equality is known, and from $\mathsf{E}(H)\leq 2h-1$ in general.
\end{proof}

\begin{definition}\label{def:compatible-subsequence}
For $H\leq F$ and $a\in N_F(H)$, call a nonempty sequence $V$ \emph{$a$-compatible modulo $H$} when $\pi(V)\cap a^{|V|}H\neq\varnothing$. Define $b_F(a,H)=|F:H|-1$ in either improved case of the next lemma, and define $b_F(a,H)=|F:H|$ otherwise.
\end{definition}

\begin{lemma}\label{lem:compatible-extraction}
Disjoint compatible subsequences concatenate compatibly. Every $I=|F:H|$ terms contain a compatible subsequence of length at most $I$. If either of the following conditions holds, then $I-1$ terms suffice:
\begin{itemize}
\item[(i)] $F$ is non-solvable, $H$ is abelian, and $a\in C_F(H)$; or
\item[(ii)] $F$ is non-abelian simple and $|H|$ is odd.
\end{itemize}
These assertions follow from ordinary subgroup extraction in $F\times C_t$, where $t=\ord(a)$, using
$$ \widehat{H}=\{(a^ih,z^i):h\in H,\ i\in\mathbb{Z}/t\mathbb{Z}\}. $$
Here $C_t=\langle z\rangle$. This subgroup has index $I$, and its coset action is $(g,z^i):xH\mapsto gxa^{-i}H$.
\end{lemma}

\begin{proof}
The identity
$$ (a^uh)(a^vk)=a^{u+v}(a^{-v}ha^v)k\in a^{u+v}H $$
proves concatenation and, in the direct product, proves that $\widehat{H}$ is a subgroup. Its $t$ second coordinates each allow $|H|$ first coordinates, so its index is $I$. A product of lifted terms $(x,z)$ lies in $\widehat{H}$ if and only if the original subsequence is compatible. The displayed transitive action has stabilizer $\widehat{H}$, since $ga^{-i}H=H$ if and only if $g\in a^iH$. By Lemma~\ref{lem:subgroup-extraction}(i), the general bound follows.

In (i), $\widehat{H}$ is generated by the commuting subgroups $H\times\{1\}$ and $\langle(a,z)\rangle$, so $\widehat{H}$ is abelian. The group $F\times C_t$ has the non-solvable quotient $F$, and Lemma~\ref{lem:subgroup-extraction}(ii) applies.

For (ii), let $N=|F|$. Every left translation in the action is even, because its sign is a homomorphism $F\rightarrow C_2$. Right multiplication by $a^{-1}$ has cycles of common length $d=|\langle a\rangle H:H|$: a coset returns after $j$ steps if and only if $a^j\in H$. By Lemma~\ref{lem:simple-group-basics}, $d\mid N/2$. Since $|H|$ is odd, both $I=N/|H|$ and $I/d$ are even. Thus right multiplication also has sign $(-1)^{I-I/d}=1$. Every element of the lifted action is therefore even, whereas a failure with $I-1$ terms would produce an odd $I$-cycle by Lemma~\ref{lem:coset-cycle}. This contradiction proves (ii).
\end{proof}

The next lemma combines an interval of cosets among the prescribed-length product sets with compatible subsequences from the remaining terms.

\begin{lemma}\label{lem:compatible-interval}
Let $S=TR$, $H\leq F$, $a\in N_F(H)$, and $0\leq\alpha\leq\beta\leq\ell$, where $\ell\geq 1$. Suppose that $a^jH\subseteq\Pi_j(T)$ for every $\alpha\leq j\leq\beta$. Suppose also that every $L$ terms of $R$ contain a compatible subsequence of length at most $e$, where $1\leq e\leq L$. If either $\beta=\ell$ or both
$$ \beta-\alpha\geq e-1 \qquad\text{and}\qquad |R|\geq\ell-\beta+L-1, $$
then $a^\ell H\subseteq\Pi_\ell(S)$. In particular, this gives a product-one subsequence when $a^\ell\in H$. One may always use $L=e=b_F(a,H)$.
\end{lemma}

\begin{proof}
The case $\beta=\ell$ is already part of the hypothesis. Otherwise, extract disjoint compatible subsequences from $R$ until their total length $u$ first reaches $\ell-\beta$. Before each required extraction, the selected total is at most $\ell-\beta-1$, so at least $L$ terms remain. At the first crossing,
$$ \ell-\beta\leq u\leq\ell-\beta+e-1\leq\ell-\alpha. $$
By Lemma~\ref{lem:compatible-extraction}, the concatenation has a selected product $y\in a^uH$. Set $j=\ell-u$, so $\alpha\leq j\leq\beta$. Because $a$ normalizes $H$, one has $(a^jH)y=a^\ell H$. Multiplying every available product from $a^jH\subseteq\Pi_j(T)$ by this fixed $y$ proves the conclusion, using disjoint terms of total length $\ell$.

If $j=0$, then the convention $\Pi_0(T)=\{1\}$ applies; a hypothesis with $\alpha=0$ consequently forces $H=\{1\}$. The final bound follows from Lemma~\ref{lem:compatible-extraction}.
\end{proof}

\section{The bounded-height case}\label{sec:bounded-height}

We first consider sequences with sufficiently many terms in a coset $aH$, where $a\in N_G(H)$. We then use product-set enlargement to treat the remaining sequences of height at most $m$.

\begin{lemma}\label{lem:height-selection}
If $q\geq v\geq 1$, $|T|=q+v$, and $h(T)\leq q$, then $T$ has a subsequence $W$ of length $2v$ and height at most $v$.
\end{lemma}

\begin{proof}
If two distinct values each occur at least $v$ times, then choose $v$ copies of each. This gives the required subsequence. If exactly one value $g$ occurs at least $v$ times, then let $c=v_g(T)\leq q$. There are $q+v-c\geq v$ terms different from $g$, and each of their values occurs fewer than $v$ times. Choose $v$ copies of $g$ and any $v$ of those other terms; the resulting subsequence has length $2v$ and height at most $v$. If no value occurs at least $v$ times, then every subsequence has height at most $v$, and $|T|=q+v\geq2v$. In this case, any $2v$ terms suffice.
\end{proof}

\begin{definition}\label{def:balanced-sequences}
A sequence $E$ is \emph{balanced} when $|E|=2h$ and $h(E)\leq h$ for some integer $h\geq 1$. A balanced sequence with $|\Pi_h(E)|\leq h$ is \emph{minimal} when no proper balanced subsequence has the analogous property at its own half-length.
\end{definition}

\begin{lemma}\label{lem:balanced-coset}
Let $E$ be minimal as in Definition~\ref{def:balanced-sequences}, with $|E|=2h$. Then $h\geq 2$, and, for $P=\langle\supp(E)\rangle$, there are a normal subgroup $H\mathbin{\trianglelefteq} P$ of order $h$ and an element $a\in N_G(H)$ such that $\supp(E)\subseteq aH$ and
$$ \Pi_h(E)=a^hH. $$
One can take $H$ to be the right stabilizer of $\Pi_h(E)$. Every sequence $D$ obtained by deleting two terms and satisfying $h(D)\leq h-1$ satisfies $\Pi_{h-1}(D)=a^{h-1}H$.
\end{lemma}

\begin{proof}
When $h=1$, balance gives two different terms, so $|\Pi_1(E)|=2$. Thus $h\geq 2$. Let $B=\Pi_h(E)$. If a value has multiplicity $h$, then choose such a value as $a$; otherwise, choose any value in the support.

Fix $b\in\supp(E)\setminus\{a\}$ and delete one copy each of $a$ and $b$, obtaining $D_b$. Its height is at most $h-1$: if two values have multiplicity $h$, then they account for all terms and are $a,b$; if only one does, then it is $a$; otherwise, every original multiplicity was already at most $h-1$. Minimality yields $|Q_b|\geq h$ for $Q_b=\Pi_{h-1}(D_b)$. Each of $aQ_b$, $Q_ba$, $bQ_b$, and $Q_bb$ lies in $B$, by adding one deleted term at the appropriate end. Since $|B|\leq h$, all four sets equal $B$ and $|B|=h$.

The identities $aQ_b=Q_ba=B$ imply that $aB=Ba$; similarly, $bB=Bb$. Also, $Ba=(bQ_b)a=b(Q_ba)=bB=Bb$. Varying $b$, we obtain
\begin{equation} Bx=xB=Ba. \tag{9.1}\label{eq:balanced-translate} \end{equation}
This identity holds for every $x\in\supp(E)$.

Let $H=\{g:Bg=B\}$. This is a subgroup and lies in $P$: if $b_0\in B\subseteq P$ and $g\in H$, then both $b_0$ and $b_0g$ lie in $P$. The preceding identities show that $xBx^{-1}=B$, and hence $xHx^{-1}=H$, for each value in the support. Therefore $H\mathbin{\trianglelefteq} P$. The equality $Bx=Ba$ also implies that $xa^{-1}\in H$, so $x\in Ha=aH$.

All $h$-term products consequently lie in $a^hH$. Since $BH=B$, the nonempty subset $B$ of this single right $H$-orbit is the whole orbit. Thus $B=a^hH$ and $|H|=|B|=h$.

Finally, for any allowed deletion $D$, minimality gives $|\Pi_{h-1}(D)|\geq h$, while its terms lie in $aH$, so $\Pi_{h-1}(D)\subseteq a^{h-1}H$. Equal cardinalities prove the last assertion.
\end{proof}

We shall use the following enlargement argument with an independent height parameter $\rho$.

\begin{lemma}\label{lem:stabilizer-descent}
Let $K\leq G$, $k=|K|$, and $\rho\geq 1$. Suppose that $T$ has length $\rho+k-1$, height at most $\rho$, and support in $aK$, where $a\in N_G(K)$. Suppose that $U\mid T$, $1\leq j\leq|U|$, and $\varnothing\neq A\subseteq\Pi_j(U)$, with $|A|\geq|U|$, has a nontrivial left or right stabilizer in $K$. Then either
\begin{itemize}
\item[(i)] some triple with these properties has $A=a^jK$; or
\item[(ii)] there are $1<R<K$ and $x\in N_G(R)$ such that at least $\rho+|R|-1$ terms of $T$ lie in $xR$.
\end{itemize}
If the initial triple satisfies $2j>|U|$, then the triple in (i) can be chosen to satisfy the same inequality.
\end{lemma}

\begin{proof}
Choose a triple satisfying the hypotheses with $s=|U|$ maximal. When $2j>s$ is imposed, maximize among triples satisfying this condition. Every represented product lies in $a^jK$. If $A=a^jK$, then conclusion (i) holds. Assume that $A\neq a^jK$. Let
$$ Q=\{g:gA=A\},\qquad R=\{g:Ag=A\}, \qquad q=|Q|,\qquad\text{and}\qquad d=|R|. $$
Both stabilizers lie in $K$. Indeed, for $u,ug\in a^jK$, one has $g=u^{-1}(ug)\in K$; the left version uses $a^jKa^{-j}=K$. The nonempty complement of $A$ in $a^jK$ is also invariant under these stabilizers. Therefore $s\leq|A|\leq k-q$ and $s\leq|A|\leq k-d$. Denote the unused terms by $x_1,\ldots,x_v$. Then
\begin{equation} v=\rho+k-1-s\geq\rho+q-1 \qquad\text{and}\qquad v\geq\rho+d-1. \tag{9.2}\label{eq:descent-lengths} \end{equation}
Inverting all terms and represented products preserves the hypotheses, maximality, and both conclusions. We may therefore assume that $q>1$. Then $v\geq\rho+1$, so two unused values differ.

If $Ax_i\neq Ax_l$, then their union is left $Q$-invariant, has at least $|A|+q\geq s+2$ elements, and lies in $\Pi_{j+1}(U\bm\cdot x_i\bm\cdot x_l)$. This contradicts maximality. Hence
\begin{equation} Ax_i=Ax_l. \tag{9.3}\label{eq:equal-translates} \end{equation}
This equality holds for $1\leq i,l\leq v$. In particular, distinct unused values satisfy $1\neq x_ix_l^{-1}\in R$, so $d>1$. The same maximality argument applied to $Ax_ix_l\cup Ax_lx_i\subseteq\Pi_{j+2}(U\bm\cdot x_i\bm\cdot x_l)$ gives
$$ Ax_ix_l=Ax_lx_i. $$
This identity holds for $1\leq i,l\leq v$; for equal indices it is automatic.

Choose $x_i\neq x_l$ and set $\delta=x_ix_l^{-1}\in R\setminus\{1\}$. For every unused index $e$, the preceding identities give
$$ Ax_ex_i=Ax_ix_e=Ax_lx_e=Ax_ex_l. $$
Thus both $Ax_e$ and $x_eA$ are right invariant under $\langle\delta\rangle\leq K$. If these sets differ, then their union has at least $s+1$ elements and lies in $\Pi_{j+1}(U\bm\cdot x_e)$, contradicting maximality again. Consequently, $x_eA=Ax_e$ for every $e$.

We have $x_iAx_i^{-1}=A$, so $x_iRx_i^{-1}=R$. Moreover, $Ax_i=Ax_1$ implies that $x_ix_1^{-1}\in R$. All unused terms therefore lie in $Rx_1=x_1R$, and $x_1\in N_G(R)$. The lower bound on $v$ shows that there are at least $\rho+d-1$ such terms. We already have $d>1$, while $d<k$, since right $K$-invariance of a nonempty subset of $a^jK$ would force that subset to equal the entire coset. This proves (ii).

Finally, the three enlargements change $(s,j)$ by $(2,1)$, $(2,2)$, or $(1,1)$. They leave $2j-s$ unchanged or increase it, so all maximality arguments remain valid with the additional strict inequality $2j>s$.
\end{proof}

\begin{lemma}\label{lem:short-compatible}
Let $1<K<G$, $k=|K|$, and $x\in N_G(K)$. Every $m-k+1$ terms contain an $x$-compatible subsequence modulo $K$ of length at most $m-k+1$.
\end{lemma}

\begin{proof}
For $k=2$, the normalizer centralizes $K$, and by Lemma~\ref{lem:compatible-extraction}(i), the bound is $m-1$. For $k\geq 3$, let $b=n/k\geq 5$. Then
$$ 2(m-k+1-b)=(k-2)(b-2)-2\geq 0, $$
so the general index bound in Lemma~\ref{lem:compatible-extraction} suffices.
\end{proof}

\begin{lemma}\label{lem:coset-saturation}
Let $H\leq G$, $h=|H|>1$, $a\in N_G(H)$, and $\rho\geq h-1$. Suppose that $T_0$ lies in $aH$, has length $\rho+h-1$, and has height at most $\rho$. There are $1<K\leq H$, $x\in N_G(K)$, and $U\mid T\mid T_0$ such that
$$ \supp(T)\subseteq xK\subseteq aH \qquad\text{and}\qquad |T|=\rho+k-1, $$
$$ 1\leq j\leq s\leq 2k-2,\qquad 0\leq c\leq k-1,\qquad\text{and}\qquad \Pi_j(U)=x^jK. $$
Here $k=|K|$, $s=|U|$, and $c=s-j$. Consequently,
$$ \Pi_i(T)=x^iK. $$
The last equality holds for every $j\leq i\leq\rho+k-1-c$.
\end{lemma}

\begin{proof}
Choose $K\leq H$ of least order $k>1$ for which a normalizing coset $xK$ contains at least $\rho+k-1$ terms of $T_0$. This is possible with $K=H$ and $x=a$. Select exactly that many terms as $T$. Since $xK$ meets $aH$ and $K\leq H$, one has $xK\subseteq aH$. By Lemma~\ref{lem:height-selection}, with $q=\rho$ and $v=k-1$, choose $W\mid T$ with $|W|=2k-2$ and $h(W)\leq k-1$. All its $(k-1)$-term products lie in $x^{k-1}K$.

\emph{Case 1: $\Pi_{k-1}(W)=x^{k-1}K$.} Take $U=W$ and $j=k-1$. Then $s=2k-2$ and $c=k-1$, as required.

\emph{Case 2: $\Pi_{k-1}(W)\neq x^{k-1}K$.} Then $|\Pi_{k-1}(W)|\leq k-1$. Choose a minimal balanced subsequence $E\mid W$ having a small half-length product set, and write $|E|=2k_0$. By Lemma~\ref{lem:balanced-coset}, there are $2\leq k_0\leq k-1$, a subgroup $K_0$ of order $k_0$, and a normalizing coset $zK_0$ containing $E$. The right-stabilizer description in that lemma shows that $K_0\leq K$: if $b,bg\in\Pi_{k_0}(E)\subseteq x^{k_0}K$, then $g\in K$. Thus $K_0<K$.

Delete two terms from $E$ to obtain $D$ of height at most $k_0-1$. This is possible by deleting one copy of each value of multiplicity $k_0$ (there are at most two) and choosing any additional deletion when needed. By Lemma~\ref{lem:balanced-coset}, $A_0=\Pi_{k_0-1}(D)=z^{k_0-1}K_0$. Minimality of $K$ implies that some term $c_1$ of $T$ lies outside $zK_0$; otherwise, that smaller normalizing coset would contain at least $\rho+k_0-1$ terms of $T_0$. Choose for $c_2$ one of the deleted terms of $E$. The indexed terms $D,c_1,c_2$ are disjoint, because every term of $E$ lies in $zK_0$.

The sets $A_0c_1$ and $A_0c_2$ are disjoint: an intersection would imply that $c_1\in K_0c_2=zK_0$. Both sets are left $K_0$-invariant, and hence
$$ A_0c_1\cup A_0c_2\subseteq\Pi_{k_0}(D\bm\cdot c_1\bm\cdot c_2) \qquad\text{and}\qquad |A_0c_1\cup A_0c_2|=2k_0=|D\bm\cdot c_1\bm\cdot c_2|. $$
Apply Lemma~\ref{lem:stabilizer-descent} inside $xK$ to this triple and $T$. Its concentration alternative contradicts the minimal choice of $K$, since $T\mid T_0$. Alternative (i) therefore provides $U\mid T$ and $j\geq 1$ with $\Pi_j(U)=x^jK$ and $s=|U|\leq k$. In particular, $s\leq 2k-2$ and $c=s-j\leq k-1$.

In either case, for $j\leq i\leq|T|-c$, there are at least $i-j$ terms of $T$ outside $U$. Choose any such terms, with ordered product $y\in x^{i-j}K$. Multiplying the coset $\Pi_j(U)=x^jK$ by $y$ yields $x^iK\subseteq\Pi_i(T)$. The reverse inclusion follows because all terms lie in $xK$.
\end{proof}

Combining the preceding lemma with compatible extraction gives the required concentration criterion.

\begin{proposition}\label{prop:coset-concentration}
Let $1<H<G$, $h=|H|$, $a\in N_G(H)$, and $\rho\geq m-1$. If
$$ |S|\geq 3m-1 \qquad\text{and}\qquad \sum_{g\in aH}\min\{\rho,v_g(S)\}\geq\rho+h-1, $$
then some nontrivial subgroup $K\leq H$ satisfies $K\subseteq\Pi_n(S)$. In particular, $1\in\Pi_n(S)$.
\end{proposition}

\begin{proof}
For each $g\in aH$, a subsequence of height at most $\rho$ can contain at most $\min\{\rho,v_g(S)\}$ copies of $g$. Conversely, choosing that many copies for every $g$ gives a subsequence of length $\sum_{g\in aH}\min\{\rho,v_g(S)\}$ and height at most $\rho$. The assumed inequality therefore provides a subsequence $T_0$ of length exactly $\rho+h-1$ and height at most $\rho$. Since $h\leq n/5<m$, one has $\rho\geq h-1$. Apply Lemma~\ref{lem:coset-saturation} to obtain $K,x,T,U,j,s,c$. Set
$$ \alpha=j,\qquad \beta=\rho+k-1-c,\qquad\text{and}\qquad R=S\cdot T^{[-1]}. $$
One has $\Pi_i(T)=x^iK$ throughout $[\alpha,\beta]$. Moreover, $j\leq 2k-2<n$, since $k\leq n/5$. If $\beta\geq n$, then this interval already gives $\Pi_n(T)=x^nK=K$.

Assume that $\beta<n$. By Lemma~\ref{lem:short-compatible}, compatible subsequences are available with $L=e=m-k+1$. The two numerical requirements of Lemma~\ref{lem:compatible-interval} are
$$ \beta-\alpha=\rho+k-1-s\geq\rho-k+1\geq m-k=L-1 $$
and
$$ \begin{aligned} |R|-(n-\beta+L-1) &=|S|-n-c-L+1\\ &\geq m-1-(k-1)-L+1=0. \end{aligned} $$
Here we used $s\leq 2k-2$ and $c\leq k-1$. By Lemma~\ref{lem:compatible-interval}, with $\ell=n$, one now has $x^nK=K\subseteq\Pi_n(S)$.
\end{proof}

We now prove the bounded-height case.

\begin{theorem}\label{thm:bounded-height}
If $|S|=3m-1$ and $h(S)\leq m$, then $1\in\Pi_n(S)$.
\end{theorem}

\begin{proof}
Suppose, for a contradiction, that $1\notin\Pi_n(S)$. First, $\supp(S)$ generates $G$. Indeed, if it generated a proper subgroup $H$ of order $h\leq n/5<m$, then $|S_H|=3m-1\geq n+h-1$, and Lemma~\ref{lem:subgroup-extraction}(iv) would imply the desired product. Two different values cannot each occur $m$ times, because each constant $m$-term sequence has product one by Lemma~\ref{lem:simple-group-basics}. Nor can $1$ occur at least $m-1$ times: by Lemma~\ref{lem:central-multiplicity}, with $d=m-1\geq \mathsf{d}(G)$ and $|S|=n+d$, one would have $1\in\Pi_n(S)$. Thus at most one value occurs $m$ times, and it is nonidentity.

\emph{Case 1: there is a balanced subsequence $E\mid S$ of length $2h$ such that $|\Pi_h(E)|\leq h$.} Choose such an $E$ of minimum length. By Lemma~\ref{lem:balanced-coset}, there is $H\mathbin{\trianglelefteq} P=\langle\supp(E)\rangle$ of order $h\geq 2$. Since $2h\leq 3m-1<2n$, the subgroup $H$ is proper in $G$, and simplicity implies that $N_G(H)<G$. Since $P\leq N_G(H)$, the subgroup $P$ is also proper and nontrivial, so $N_G(P)<G$.

Choose a term $c$ of $S$ outside $N_G(P)$, using the fact that the support generates $G$. Some $d\in\supp(E)$ satisfies $cdc^{-1}\notin P$: otherwise, conjugation by $c$ would take every generator of $P$ into $P$, and finiteness would imply that $cPc^{-1}=P$. Delete one copy of $d$ and a term of a different value from $E$ so that the remaining sequence $D$ has height at most $h-1$. Such a deletion is always possible: if two values have multiplicity $h$, then they are the only values; if a unique value has multiplicity $h$ and is not $d$, then delete it second; in the other cases, any value different from $d$ suffices.

By Lemma~\ref{lem:balanced-coset}, the set $B=\Pi_{h-1}(D)=a^{h-1}H$ has $h$ elements and satisfies $HB=B$. The indexed term $c$ is outside $E$, since its value is outside $N_G(P)\supseteq P$. Therefore
$$ A=Bdc\cup Bcd\subseteq\Pi_{h+1}(D\bm\cdot d\bm\cdot c). $$
The two sets lie in the distinct right $P$-cosets $Pc$ and $Pcd$: equality of those cosets would imply that $cdc^{-1}\in P$. Consequently, $|A|=2h=|D\bm\cdot d\bm\cdot c|$, and $A$ is left $H$-invariant.

Apply Lemma~\ref{lem:stabilizer-descent} with $K=G$, $T=S$, and $\rho=m$. Here $|T|=m+n-1$, and the initial represented length $h+1$ satisfies $2(h+1)>2h$. In the concentration alternative, there are $1<R<G$ and $x\in N_G(R)$ such that at least $m+|R|-1$ terms lie in $xR$, with height at most $m$. Proposition~\ref{prop:coset-concentration}, with $\rho=m$, contradicts the assumption. In alternative (i), there are $U\mid S$ and $j$ with $\Pi_j(U)=G$, $|U|\leq n$, and $2j>|U|$. Thus $|U|-j<|U|/2\leq m$, and integrality implies that $|U|-j\leq m-1=|S|-n$. By Lemma~\ref{lem:subgroup-extraction}(iii), this leads to the same contradiction.

\emph{Case 2: every balanced $W\mid S$ of length $2v$ satisfies $|\Pi_v(W)|\geq v+1$.} Choose $U\mid S$ of greatest length subject to
\begin{equation} 2\leq |U|\leq m+2,\qquad |\pi(U)|\geq |U|,\qquad\text{and}\qquad h(S\cdot U^{[-1]})\leq m-1. \tag{9.4}\label{eq:bounded-height-witness} \end{equation}
Such a $U$ exists: start with a noncommuting pair, whose two orderings give two different products. If a value $a$ occurs $m$ times, then include it in this pair. This is possible because $a\neq 1$, $Z(G)=1$, and the generating support cannot lie in $C_G(a)$. Otherwise, every pair already leaves height at most $m-1$.

Suppose that $s=|U|\leq m$, and let $A=\pi(U)$ and $R=S\bm\cdot U^{[-1]}$. Then $h(R)\leq m-1$ and $|R|=3m-1-s\geq n-1$. If $|A|>s$, then adjoining any unused term $x$ gives $Ax\subseteq\pi(U\bm\cdot x)$ with at least $s+1$ products, contradicting maximality. Hence $|A|=s$.

The remainder must generate $G$. Otherwise, let $H=\langle\supp(R)\rangle<G$, of order $h$. Necessarily $h\geq 3$, since a group of order at most two supports at most $2m-2$ terms at height $m-1$. Also, $h\leq n/5<m$, whence $|R|\geq n-1\geq(m-1)+h-1$. Select that many terms from $R$. They lie in $H$ with height at most $m-1$, so Proposition~\ref{prop:coset-concentration}, applied to $S$ with $\rho=m-1$, gives the forbidden product.

If $A$ is not conjugation invariant, then some unused value $x$ has $Ax\neq xA$. Indeed, equality for all values generating $G$ would imply conjugation invariance under all of $G$. Then $Ax\cup xA\subseteq\pi(U\bm\cdot x)$ has at least $s+1$ elements, contradicting maximality.

Suppose finally that $A$ is conjugation invariant. We claim that $\lambda_A(z)\geq 2$ for every $z\neq 1$. The conjugacy class $X=z^G$ generates a nontrivial normal subgroup, and hence generates $G$. Moreover, $|X|=|G:C_G(z)|\geq 5$ because $C_G(z)<G$. Conjugation invariance makes $\lambda_A$ constant on $X$. Since $2\leq|A|=s\leq m$, it follows from Lemma~\ref{lem:product-set-tools}(iii) that this common value is at least
$$ \min\left\{\frac{s+1}{2},\frac{|X\cup X^{-1}|+2}{4}\right\} \geq\frac{3}{2}. $$
This value is an integer, so it is at least two. Choose noncommuting unused terms $x,y$, using generation of the non-abelian group, and set $z=xy(yx)^{-1}\neq 1$. Right translation gives
$$ |Axy\cup Ayx|=|Az\cup A|=|A|+\lambda_A(z)\geq s+2. $$
Both sets lie in $\pi(U\bm\cdot x\bm\cdot y)$, again contradicting maximality. All the proposed enlargements have length at most $m+2$ and cannot increase the remaining height. Thus $|U|\in\{m+1,m+2\}$.

Let $k=|U|$ and $v=n-k\in\{m-1,m-2\}$. The remainder has length $v+(m-1)$ and height at most $m-1$. By Lemma~\ref{lem:height-selection}, it contains a balanced sequence $W$ of length $2v$. The assumption of Case 2 yields
$$ |\pi(U)|+|\Pi_v(W)|\geq k+v+1=n+1. $$
By Lemma~\ref{lem:product-set-tools}(i), $\pi(U)\Pi_v(W)=G$. Since $U$ and $W$ are disjoint and $k+v=n$, this implies that $1\in\Pi_n(S)$, the final contradiction.
\end{proof}

\section{High multiplicity and completion criteria}\label{sec:high-multiplicity}

Let $G$ be non-abelian simple, $|G|=n=2m$, and $|S|=3m-1$. By Theorem~\ref{thm:bounded-height}, it remains to consider $h(S)>m$. If an element occurs at least $n$ times, then $n$ copies of that element have product one. We may therefore write
\begin{equation} S=a^{[M]}R,\qquad t=\ord(a),\qquad\text{and}\qquad C=\langle a\rangle. \tag{10.1}\label{eq:high-mult-setup} \end{equation}
Here $a\notin\supp(R)$ and $m+1\leq M<n$. In particular, $|R|\geq m$, and by Lemma~\ref{lem:simple-group-basics}, $t\mid m$.

\begin{lemma}\label{lem:high-mult-completion}
Under the assumptions of \eqref{eq:high-mult-setup}, either of the following conditions implies that $1\in\Pi_n(S)$:
\begin{itemize}
\item[(i)] $\mathsf{d}(G\times C_t)+1\leq m+t-1$;
\item[(ii)] some $U\mid S$ satisfies $|U|\leq n$ and $\pi(U)=G$.
\end{itemize}
\end{lemma}

\begin{proof}
For (i), reserve $M_0=t\lfloor M/t\rfloor\geq m$ copies of $a$ and lift each remaining term $x$ to $(x,z)$, where $z$ generates $C_t$. Every $m+t-1$ available terms contain a nonempty product-one subsequence of length divisible by $t$, and hence of length at most $m$. Extract disjoint such subsequences until their total length $u$ first reaches $n-M_0$. Before a necessary extraction, $u\leq n-M_0-t$, so
$$ 3m-1-M_0-u\geq m+t-1. $$
At the first crossing, $n-M_0\leq u\leq n-M_0-t+m\leq n$. Both $n$ and $u$ are divisible by $t$. Appending $n-u\leq M_0$ reserved copies gives length $n$ and product one. If every $m+t-1$ terms have a nonempty product-one subsequence whose length is divisible by $t$ and is at most $m$, then the same argument applies.

Part (ii) follows from Lemma~\ref{lem:subgroup-extraction}(iii), with $j=|U|$.
\end{proof}

When $a=1$, Lemma~\ref{lem:high-mult-completion}(i) applies because $\mathsf{d}(G)+1\leq m$. Henceforth, assume that $a\neq 1$.

\begin{lemma}\label{lem:subgroup-cap}
Suppose that the assumptions of \eqref{eq:high-mult-setup} hold and $1\notin\Pi_n(S)$. Then every proper subgroup $L<G$ containing $a$ satisfies $|R_L|\leq|L|-2$. In particular, at least $m-t+2$ terms of $R$ lie outside $C$.
\end{lemma}

\begin{proof}
Let $\ell=|L|$. If $\ell-1$ remainder terms lie in $L$, then combine them with $m$ copies of $a$. This subsequence in $L$ has length $m+\ell-1$ and height at most $m$, since $|G:L|\geq 5$. By Proposition~\ref{prop:coset-concentration}, with $H=L$ and $\rho=m$, the desired product-one subsequence exists, a contradiction. Apply the result to $L=C$ and use $|R|\geq m$.
\end{proof}

The remaining proof is organized by construction. Section~\ref{sec:subgroup-extraction} uses subgroups and compatible products; Section~\ref{sec:short-subsequences} constructs short product-one subsequences; Sections~\ref{sec:orbit-growth} and~\ref{sec:cyclic-seeds} enlarge conjugation-orbit products, with or without an initial cyclic subgroup. The structural coverage is checked in Section~\ref{sec:structural-coverage}, after these methods have been established.

\section{Subgroup and compatible extraction}\label{sec:subgroup-extraction}

The first construction replaces a sequence by products in a smaller group. All applications below use the same extraction principle; only the subgroup, its index, and its short-product bound change.

\begin{lemma}\label{lem:extraction-bounds}
Let $H\leq F$, $K\mathbin{\vartriangleleft} F$, and $e,f,t\geq 1$. Then
\begin{align} \mathsf{d}(F\times C_t)+1 &\leq b\bigl(\mathsf{d}(H\times C_t)+1\bigr), \tag{i}\\ \mathsf{\eta}_{ef}(F) &\leq\mathsf{\eta}_e(F/K)+e\bigl(\mathsf{\eta}_f(K)-1\bigr). \tag{ii} \end{align}
Here one may always take $b=|F:H|$. If the index is even and the coset action of $F$ consists of even permutations, then one may take $b=|F:H|-1$. In (ii), assume that $\mathsf{\eta}_e(F/K)$ and $\mathsf{\eta}_f(K)$ exist; the conclusion includes existence of $\mathsf{\eta}_{ef}(F)$. Consequently,
$$ \mathsf{d}(F)+1\leq\mathsf{\eta}_e(F/K)+e\,\mathsf{d}(K) $$
whenever $\mathsf{\eta}_e(F/K)$ exists, and unconditionally
$$ \mathsf{d}(F)+1\leq\bigl(\mathsf{d}(K)+1\bigr)\bigl(\mathsf{d}(F/K)+1\bigr). $$
Also, if $\mathsf{\eta}_e(P)$ exists and $\mathsf{\eta}_e(P)\leq L$, then
$$ \mathsf{d}(P\times C_t)+1\leq L+e(t-1). $$
\end{lemma}

\begin{proof}
By Lemma~\ref{lem:subgroup-extraction}(i), an $H$-product can be extracted from $|F:H|$ terms. In the improved case, a failure with $|F:H|-1$ terms would give an odd full coset cycle by Lemma~\ref{lem:coset-cycle}. Thus $b$ terms suffice. Extract $\mathsf{d}(H\times C_t)+1$ disjoint such subsequences, retaining their cyclic coordinates, and select a product-one subsequence of their products. This proves (i).

For (ii), let $L_0=\mathsf{\eta}_e(F/K)$ and $L_1=\mathsf{\eta}_f(K)$. Start with $L_0+e(L_1-1)$ terms. Before extraction $i\leq L_1$, at most $e(i-1)$ terms have been used, leaving at least $L_0+e(L_1-i)\geq L_0$. We can therefore extract $L_1$ disjoint nonempty subsequences of length at most $e$, each with a chosen ordered product in $K$. A nonempty selection of at most $f$ of these products has product one in some order. Concatenating the corresponding subsequences in that order gives length at most $ef$. Take $f=\mathsf{d}(K)+1$ for the first consequence, then $e=\mathsf{d}(F/K)+1$ for the second. For the last consequence, use the normal subgroup $C_t$ of $P\times C_t$, with quotient $P$ and $\mathsf{d}(C_t)+1=t$.
\end{proof}

A subgroup of $C_G(a)$ gives a second way to construct short $a$-compatible subsequences.

\begin{lemma}\label{lem:centralizer-criterion}
Under the assumptions of \eqref{eq:high-mult-setup}, suppose that $A\leq C_G(a)$. Define
$$ q_A= \begin{cases} |G:A|-1,&\text{when $A$ is abelian or $|A|$ is odd},\\ |G:A|,&\text{otherwise}. \end{cases} $$
If $q_A\bigl(\mathsf{d}(A)+1\bigr)\leq m$, then $1\in\Pi_n(S)$.
\end{lemma}

\begin{proof}
By Lemma~\ref{lem:compatible-extraction}, every $q_A$ terms contain a nonempty $a$-compatible subsequence modulo $A$ of length at most $q_A$. From $L=q_A\bigl(\mathsf{d}(A)+1\bigr)$ terms, extract $\mathsf{d}(A)+1$ such subsequences, with selected products $a^{s_i}h_i$, where $h_i\in A$. Select a nonempty product-one subsequence of the $h_i$. Since $a$ centralizes $A$, the corresponding concatenation has product $a^s$, where $s$ is its total length and $1\leq s\leq L$.

Apply Lemma~\ref{lem:compatible-interval} with $H=1$, $T=a^{[M]}$, $[\alpha,\beta]=[0,M]$, $\ell=n$, and $e=L=q_A\bigl(\mathsf{d}(A)+1\bigr)$. Indeed,
$$ M\geq m+1\geq L-1 \qquad\text{and}\qquad |R|=3m-1-M\geq n-M+L-1. $$
The lemma therefore implies that $a^n=1\in\Pi_n(S)$.
\end{proof}

\subsection{Fixed-point-free actions}\label{sec:fixed-point-free}

\begin{lemma}\label{lem:frobenius-extraction}
Let $F=A\rtimes C_b$, where $A$ is abelian of order $v$, and every nonidentity element of the complement acts without nonidentity fixed points on $A$. Then
\begin{equation} \mathsf{d}(F\times C_t)+1\leq t(v-1)+\ell+g-1. \tag{11.1}\label{eq:frobenius-davenport} \end{equation}
Here $g=\gcd(b,t)$ and $\ell=\lcm(b,t)$. In particular, if $t\mid b$, then $\mathsf{d}(F\times C_t)+1\leq tv+b-1$.
\end{lemma}

\begin{proof}
Let $T$ be a product-one free sequence over $F\times C_t$. Modulo $A$, the quotient is $Q=C_b\times C_t$, with Davenport constant $\mathsf{D}(Q)=\ell+g-1$. Extract minimal product-one sequences $U_i$ in this quotient until at most $\mathsf{D}(Q)-1$ terms remain. Write $s_i=|U_i|$ and select an ordering with product $p_i\in A\setminus\{1\}$. The $s_i$ proper prefixes, including the empty prefix, have distinct images in $Q$. At most $t$ have any fixed $C_b$-coordinate. Cyclic rotations conjugate $p_i$ by these prefixes; the fixed-point-free complement action therefore gives a subset $P_i\subseteq\pi(U_i)\cap A$ of size at least $\lceil s_i/t\rceil$.

No nonempty choice of one product from each of some of these sets has product $1$. Apply Lemma~\ref{lem:product-set-tools}(ii) successively to the sets $\{1\}\cup P_i$, using the unique representation of $1$ as $1\cdot 1$. This yields
$$ 1+\sum_i\lceil s_i/t\rceil\leq v. $$
Hence $|T|\leq t(v-1)+\mathsf{D}(Q)-1$, which proves the assertion.
\end{proof}

\begin{corollary}\label{cor:borel-divisor}
Let $G=\PSL_2(q)$. If either $q\geq 8$ is even and $t$ is a proper divisor of $q-1$, or $q\geq 5$ is odd, $t\mid(q-1)/2$, and $(q-1)/(2t)\geq 3$, then $\mathsf{d}(G\times C_t)+1\leq|G|/2$. In particular, $\mathsf{d}(\PSL_2(16)\times C_5)+1\leq 1598$.
\end{corollary}

\begin{proof}
A Borel subgroup has kernel of order $q$, complement of order $b=(q-1)/\gcd(2,q-1)$, and index $q+1$. If $q$ is even, then $t\leq(q-1)/3$, so by Lemmas~\ref{lem:extraction-bounds} and~\ref{lem:frobenius-extraction},
$$ \mathsf{d}(G\times C_t)+1\leq(q+1)(tq+q-2)\leq\frac{q(q^2-1)}{2}. $$
The last inequality follows from $(q-3)(q-4)\geq 0$. If $q$ is odd, then the coset action is even by simplicity of $G$, so its index factor improves to $q$. Since $t\leq(q-1)/6$,
$$ \mathsf{d}(G\times C_t)+1\leq q\bigl(tq+(q-3)/2\bigr) \leq\frac{q(q^2-1)}{4}. $$
After division by $q$, the difference is at least $(q^2-4q+15)/12>0$. For $q=16$ and $t=5$, the direct bound is $17(5\cdot16+15-1)=1598$.
\end{proof}

For $G=\PSL_2(q)$, the Borel subgroup $B$ has order $q(q-1)/\varepsilon$ and index $q+1$, where $\varepsilon=\gcd(2,q-1)$. By Lemma~\ref{lem:frobenius-extraction}, with cyclic factor one, $\mathsf{d}(B)+1\leq q+(q-1)/\varepsilon-1$. Direct-product extraction gives $\mathsf{d}(B\times C_t)+1\leq t\bigl(\mathsf{d}(B)+1\bigr)$. For odd $q$, the even coset action improves the index factor to $q$. Thus, by Lemma~\ref{lem:extraction-bounds},
\begin{equation} \mathsf{d}(G\times C_t)+1\leq \begin{cases} 2t(q^2-1),&\text{when $q$ is even},\\ 3tq(q-1)/2,&\text{when $q$ is odd}. \end{cases} \tag{11.2}\label{eq:borel-davenport} \end{equation}
These bounds are at most $n/2$ when $4t\leq q$ and $6t\leq q+1$, respectively.

\subsection{Short products in small subgroups}\label{sec:small-subgroup-products}

Let $G$ be non-abelian simple, and assume that $S=a^{[M]}R$ satisfies the assumptions of \eqref{eq:high-mult-setup}. Write $n=|G|=2m$ and $t=\ord(a)$. If $S$ is a counterexample, then $|R|\geq m$, and Lemma~\ref{lem:subgroup-cap} applies to every proper subgroup containing $a$. Define
$$ P=\SD_{16}=\langle x,y:x^8=y^2=1,\ yxy=x^3\rangle. $$

\begin{lemma}\label{lem:small-eta-bounds}
The following bounds hold:
$$ \begin{array}{c|cccccc} Q&C_2^r&C_4^2&C_8^2&C_2^2\times C_4&P\times C_8&S_3\\ e&2&4&8&4&8&3\\ \mathsf{\eta}_e(Q)&2^r&10&22&\leq 10&\leq 26&\leq 4 \end{array} $$
\end{lemma}

\begin{proof}
In $C_2^r$, an identity or a repeated term gives a zero-sum subsequence of length at most two. Conversely, all nonzero elements taken once give a sequence of length $2^r-1$ with no such subsequence.

Apply Lemma~\ref{lem:extraction-bounds}(ii) to the doubling subgroup of $C_{2h}^2$. This yields
$$ \mathsf{\eta}_{2h}(C_{2h}^2)\leq 4+2\bigl(\mathsf{\eta}_h(C_h^2)-1\bigr). $$
Starting with $\mathsf{\eta}_2(C_2^2)=4$ gives $10$ and $22$. To see equality, take a basis $u,v$ of $C_h^2$ and the sequence $u^{[h-1]}v^{[h-1]}(u+v)^{[h-1]}$. A nonempty zero-sum subsequence must have $h-c,h-c,c$ copies of these terms, respectively, for $1\leq c<h$. Its length is $2h-c>h$. Set $h=4,8$.

For $C_2^2\times C_4$, the doubling subgroup is $C_2$ and the quotient is $C_2^3$, giving $8+2(2-1)=10$. If $c$ generates $C_8$, then the normal subgroup $\langle x^2,c^2\rangle$ of $P\times\langle c\rangle$ is $C_4^2$ and the quotient is $C_2^3$. Hence the same lemma yields $\mathsf{\eta}_8(P\times C_8)\leq 8+2(10-1)=26$.

Finally, consider four terms in $S_3$. An identity, a repeated transposition, inverse $3$-cycles, or three equal $3$-cycles suffice. If none occurs, then there must be two distinct transpositions and a $3$-cycle, since there are only three transpositions and two $3$-cycles. The two orders of the transpositions give the inverse pair of $3$-cycles. One cancels the selected $3$-cycle, proving the last bound.
\end{proof}

\begin{lemma}\label{lem:small-davenport-bounds}
For $U=\UT_3(\mathbb{F}_3)$ and $P=\SD_{16}$, one has
$$ \mathsf{d}(U\times C_3)\leq 10,\qquad \mathsf{d}(Q_8\times C_8)+1\leq 12,\qquad\text{and}\qquad \mathsf{d}(P\times C_2)+1\leq 10. $$
\end{lemma}

\begin{proof}
By Lemma~\ref{lem:augmentation-bounds}, the relation $J^b=0$ for the augmentation ideal of $\mathbb{F}_p[F]$ implies that $\mathsf{d}(F)+1\leq b$. We compute such powers for the three groups using weighted monomials.

For $U$ and $Q_8$, choose generators $x,y,z$ with $z$ central and $yx=xyz$. Their remaining relations are
$$ U:\quad x^3=y^3=z^3=1 $$
and
$$ Q_8:\quad x^2=y^2=z \qquad\text{and}\qquad z^2=1, $$
respectively. For the first presentation, one may take $x=I+E_{12}$, $y=I+E_{23}$, and $z=I-E_{13}$. Let $c$ generate the indicated cyclic factor and define $A=x-1$, $B=y-1$, $C=z-1$, and $D=c-1$, with weights $1,1,2,1$, respectively. In characteristics three and two, respectively, both groups satisfy
$$ BA=AB+C+AC+BC+ABC, $$
with $C,D$ central. The power reductions and normal monomials are as follows:
\begin{center}
\small
\begin{tabular}{c c c c}
\hline
$F$&Power reductions&$A^iB^jC^kD^\ell$&Maximum weight\\
\hline
$U\times C_3$&$A^3=B^3=C^3=D^3=0$&$0\leq i,j,k,\ell\leq 2$&$10$\\
$Q_8\times C_8$&
\begin{tabular}{@{}c@{}}$A^2=B^2=C$ and\\$C^2=D^8=0$\end{tabular}&
\begin{tabular}{@{}c@{}}$0\leq i,j,k\leq 1$ and\\$0\leq\ell\leq 7$\end{tabular}&$11$\\
\hline
\end{tabular}
\end{center}
For $P\times C_2$, work in characteristic two and use $A=x-1$, $B=y-1$, and $D=c-1$, all of weight one. Here $D$ is central, and
$$ A^8=B^2=D^2=0 \qquad\text{and}\qquad BA=AB+A^2+A^3+A^2B+A^3B. $$
The normal monomials $A^iB^jD^\ell$ have maximum weight nine, where $0\leq i\leq 7$ and $0\leq j,\ell\leq 1$.

In all three cases, expansion in the normal-form group basis is triangular with diagonal entries $1$, so the displayed monomials form a basis. The collection relations never decrease weight. Since $J$ is spanned by all nonconstant basis monomials, its $b$th power is contained in the span of monomials of weight at least $b$. The three maximum weights therefore give $J^{11}=0$, $J^{12}=0$, and $J^{10}=0$, respectively.
\end{proof}

\subsection{Extraction from triangular and parabolic subgroups}\label{sec:triangular-parabolic}

\begin{proposition}\label{prop:psl3-psu3-extraction}
For $G=\PSU_3(3)$ and $t\mid24$, one has $\mathsf{d}(G\times C_t)+1\leq 2856<|G|/2$. For $G=\PSL_3(3)$ and $t\mid12$, one has $\mathsf{d}(G\times C_t)+1\leq 2600<|G|/2$; for $t=8$, the bound is $2405$.
\end{proposition}

\begin{proof}
The standard Borel and parabolic descriptions \cite[Sections~3.3.3 and~3.6.2]{Wilson2009} give the following subgroups, where $U=\UT_3(\mathbb{F}_3)$:
\begin{center}
\begin{tabular}{c c r r}
\hline
$G$&Subgroup&Order&Index\\
\hline
$\PSL_3(3)$&$B_\ell=U\rtimes C_2^2$&$108$&$52$\\
$\PSL_3(3)$&$P_0=C_3^2\rtimes\GL_2(3)$&$432$&$13$\\
$\PSU_3(3)$&$B_u=U\rtimes C_8$&$216$&$28$\\
\hline
\end{tabular}
\end{center}
In the unitary case, the radical has order $27$, center and commutator subgroup of order $3$, and exponent $3$, so it is the same group $U$. The quotients of $B_u\times C_{24}$ and $B_\ell\times C_{12}$ by $U\times C_3$ are $C_8^2$ and $C_2^2\times C_4$, respectively. By Lemmas~\ref{lem:extraction-bounds}, \ref{lem:small-eta-bounds}, and~\ref{lem:small-davenport-bounds},
$$ \mathsf{d}(B_u\times C_{24})+1\leq22+8\cdot10=102 $$
and
$$ \mathsf{d}(B_\ell\times C_{12})+1\leq10+4\cdot10=50. $$
Multiplication by the subgroup indices yields $28\cdot102=2856<3024$ and $52\cdot50=2600<2808$. Restricting the cyclic factor gives the claims for divisors $t$.

For $t=8$, use $Q_8\mathbin{\vartriangleleft}\GL_2(3)$ with quotient $S_3$ \cite[Section~5.3.8]{Wilson2009}. By the same lemmas, $\mathsf{d}(\GL_2(3)\times C_8)+1\leq4+3\cdot11=37$. Since $\mathsf{d}(C_3^2)+1=5$, normal extraction followed by the index-$13$ parabolic embedding yields
$$ \mathsf{d}(\PSL_3(3)\times C_8)+1\leq13\cdot5\cdot37=2405. $$
\end{proof}

For $\PSU_3(3)$, Jordan form gives the necessary element-order list $\{1,2,3,4,6,7,8,12\}$. For $\PSL_3(3)$, the corresponding list is $\{1,2,3,4,6,8,13\}$. Thus only orders $7$ and $13$, respectively, are not covered by Proposition~\ref{prop:psl3-psu3-extraction}.

\subsection{Extraction from an affine point stabilizer}\label{sec:affine-stabilizer}

\begin{proposition}\label{prop:m11-extraction}
If $G=M_{11}$ and $t\in\{1,2,3,4,5,6,8\}$, then $\mathsf{d}(G\times C_t)+1<|G|/2=3960$.
\end{proposition}

\begin{proof}
The unordered-pair stabilizer in the degree-eleven action is $H=C_3^2\rtimes P$, of order $144$ and index $55$ \cite[Section~5.3.8]{Wilson2009}. For $t=6$, the normal subgroup $C_3^3$ of $H\times C_6$ has quotient $P\times C_2$. For $t=8$, use $C_3^2\mathbin{\vartriangleleft} H\times C_8$. Hence
$$ \mathsf{d}(M_{11}\times C_6)+1\leq55(7\cdot10)=3850 $$
and
$$ \mathsf{d}(M_{11}\times C_8)+1\leq55(26+8\cdot4)=3190. $$
These bounds also cover all divisors of $6$ and $8$. For $t=5$, the distinct-term bound with $u=11$ and $b=\lfloor3\sqrt{7920}\rfloor=266$ gives
$$ \mathsf{d}(G\times C_5)+1\leq(11+5-2)266+5=3729<3960. $$

The element-order bound follows from the sharply $4$-transitive action \cite[Theorem~6.4A]{DixonMortimer1996}, from $M_{11}\leq A_{11}$, and from its Sylow $3$-subgroup $C_3^2$. A nonidentity element and each of its nonidentity powers fixes at most three points. An element with order divisible by $11$ is an $11$-cycle. If its order is divisible by $5$, then an order-$5$ power moves ten points; the original element has two $5$-cycles, since a $10$-cycle is odd. Otherwise, its order is $2^a3^b$ with $b\leq1$. If $b=1$ and $a\geq2$, then an order-$3$ power requires at least nine points in cycles divisible by $3$, whereas an additional cycle divisible by $4$ requires four more; a common cycle would have length at least $12$. This is impossible. A $2$-element has order at most $8$. Thus the necessary list is $\{1,2,3,4,5,6,8,11\}$, and Cauchy's theorem implies that $u=11$. Only repeated elements of order $11$ remain.
\end{proof}

\subsection{Compatible extraction from root groups}\label{sec:root-compatible}

\begin{proposition}\label{prop:unipotent-extension}
Let $G=\PSL_2(p^f)$, and let the repeated element $a$ be unipotent. If either $f\geq3$ or both $f=2$ and $p\geq5$, then $1\in\Pi_n(S)$.
\end{proposition}

\begin{proof}
A root subgroup $A\cong C_p^f$ centralizes $a$. The compatible extraction bound of Lemma~\ref{lem:centralizer-criterion} has threshold
$$ \begin{aligned} q_A\bigl(\mathsf{d}(A)+1\bigr) &=\left(\frac{n}{p^f}-1\right)\bigl(f(p-1)+1\bigr)\\ &<\frac{n}{p^f}\bigl(f(p-1)+1\bigr)\leq\frac{n}{2}. \end{aligned} $$
Here $q_A=n/p^f-1$, since $A$ is abelian. The last inequality is $2\bigl(f(p-1)+1\bigr)\leq p^f$. For $f=2$, it follows from $p^2-4p+2\geq0$ when $p\geq5$. For $f=3$, it follows from $p^3-6p+4\geq0$ when $p\geq2$. Induction in $f$ follows because
$$ p\bigl(f(p-1)+1\bigr) -\bigl((f+1)(p-1)+1\bigr)=f(p-1)^2\geq0. $$
Lemma~\ref{lem:centralizer-criterion} now proves the result.
\end{proof}

\begin{corollary}\label{cor:a6-order-four}
One has $\mathsf{d}(A_6\times C_4)+1\leq180$, so the high-multiplicity case $(G,t)=(A_6,4)$ satisfies Lemma~\ref{lem:high-mult-completion}(i).
\end{corollary}

\begin{proof}
The published value $\mathsf{d}(A_5)=8$ gives $\mathsf{d}(A_5)+1=9$, whence $\mathsf{d}(A_5\times C_4)+1\leq4\bigl(\mathsf{d}(A_5)+1\bigr)=36$ by normal-subgroup extraction. The degree-six action of $A_6$ consists of even permutations. Thus, by Lemma~\ref{lem:extraction-bounds}, applied to its point stabilizer, $\mathsf{d}(A_6\times C_4)+1\leq5\cdot36=180<183$.
\end{proof}

\subsection{Intervals and subgroup covers}\label{sec:subgroup-covers}

\begin{lemma}\label{lem:interval-completion}
Let $H\leq G$, $a\in N_G(H)$, and $b=|G:H|\leq m$. Select $u$ terms of $R$, and let $T$ consist of these terms and all $M$ copies of $a$. If $0\leq\alpha\leq n$, $M+u-\alpha\geq b-1$, and
$$ a^jH\subseteq\Pi_j(T) $$
for every $\alpha\leq j\leq M+u$, then $1\in\Pi_n(S)$.
\end{lemma}

\begin{proof}
If $M+u\geq n$, then use the given interval at $j=n$. Otherwise, set $\beta=M+u$. By Lemma~\ref{lem:compatible-extraction}, compatible subsequences are available with $L=e=b$, and
$$ |S|-|T|=3m-1-\beta\geq n-\beta+b-1. $$
The other hypothesis is precisely $\beta-\alpha\geq b-1$. By Lemma~\ref{lem:compatible-interval}, with $\ell=n$, one has $a^nH=H\subseteq\Pi_n(S)$.
\end{proof}

\begin{lemma}\label{lem:subgroup-cover-caps}
Suppose that $S$ is a counterexample. For a proper subgroup $L\geq C$, the number of terms of $R$ in $L\setminus C$ is at most $c=|L|-2$. The following sharper upper bounds $c$ are also valid:
\begin{itemize}
\item[(i)] If $L=C_\ell\rtimes C$, where $\ell$ is an odd prime, the action of $C$ is faithful, and $t\geq2$, then $c=k_0+t-2$, where $k_0=t\lceil(\ell-1)/(t(t-1))\rceil$.
\item[(ii)] If $L$ is dihedral of order $2t$ and $t\geq3$ is odd, then $c=1$.
\item[(iii)] If $G=\PSL_2(7)$, $t=3$, and $L\cong S_4$, then $c=8$.
\item[(iv)] If $L\cong A_5$ and $t=5$, then $c=11$.
\end{itemize}
\end{lemma}

\begin{proof}
The general bound follows from Lemma~\ref{lem:subgroup-cap}. For (i), suppose that $k_0+t-1$ terms lie in $L\setminus C$. The cyclic EGZ theorem, applied repeatedly in $L/C_\ell\cong C_t$, selects $k_0$ terms whose images have product one; here $t\mid k_0$. Each selected term has an orbit of size $t$, lying in a coset of $C_\ell$. To verify the orbit size, write the term as $za^j$, with $1\neq z\in C_\ell$; a nonidentity power of $a$ cannot fix $z$ under a faithful action on a prime-order group. The repeated Cauchy--Davenport inequality in $C_\ell$ shows that the product of these orbit sets is the entire kernel, since $k_0(t-1)+1\geq\ell$ and their quotient product is one. Since $t\mid k_0$, the construction in Section~\ref{sec:orbit-growth}, expressed in \eqref{eq:orbit-realization}, realizes this kernel as products of $k_0t$ terms. Unused copies of $a$ then give $a^jC_\ell$ for every $j\in[k_0t,M+k_0]$.

Since $t\mid\ell-1$, one has $k_0\leq\ell-1$. Every proper subgroup of a non-abelian simple group has index at least five, so $\ell t\leq n/5$ and
$$ k_0(t-1)<t(\ell-1)\leq\frac{5}{2}t(\ell-2)\leq m-n/\ell. $$
Consequently, $k_0t<n$ and $M+k_0-k_0t\geq n/\ell-1$; also, $n/\ell\leq m$. Lemma~\ref{lem:interval-completion}, with $H=C_\ell$, contradicts a counterexample.

For (ii), two terms $x,y\in L\setminus C$ give
$$ a^ixa^{t-1-i}y=a^{2i-(t-1)}xy. $$
This equality holds for $0\leq i<t$. Because $t$ is odd, these products run through $C$. Together with unused copies of $a$, they supply $C=a^jC$ for every $j\in[t+1,M+2]$. Here $n\geq10t$ and $M-t+1\geq m-t+2\geq n/t-1$. Apply Lemma~\ref{lem:interval-completion} with $H=C$.

For (iii), write $L=S_4$, choose a point stabilizer $J\cong S_3$ containing $C$, and partition
$$ L\setminus C=(A_4\setminus C)\,\dot{\cup}\, (J\setminus C)\,\dot{\cup}\, \bigl(L\setminus(A_4\cup J)\bigr). $$
Any five terms in the first part contain three with quotient product one in $A_4/V_4\cong C_3$. Each of their orbits has three elements in a $V_4$-coset. The product of any two such sets is a full coset, since $3+3>4$; the three selected orbits therefore multiply to $V_4$. These terms and six copies of $a$ give $V_4$ in length nine, and hence the interval $[9,M+3]$ of cosets $a^jV_4$. With $m=84$, $M\geq85$, and $b=168/4=42$, Lemma~\ref{lem:interval-completion} applies. Thus the first part contains at most four remainder terms. Part (ii) bounds the second part by one.

Identify $S_4$ with the affine semilinear group of $\mathbb{F}_4$. Write $\mathbb{F}_4^*=\langle\omega\rangle$, let $a:z\mapsto\omega z$ and $\sigma:z\mapsto z^2$, and choose $J$ to stabilize zero. Elements in the third part are $(v,\omega^j\sigma)$ with $v\neq0$. For two such elements, the products of their orbit terms have linear and translation parts
$$ \omega^{j-k+2d} \qquad\text{and}\qquad \omega^i\bigl(v+\omega^{j+d}w^2\bigr), $$
respectively, where $d=i-\ell$. Exactly one value of $d\in\mathbb{Z}/3\mathbb{Z}$ makes the second parenthesis zero. For each of the other two values, the translation part runs through the three nonzero translations, while the three linear parts are distinct. Thus the two orbit sets have product size $1+3+3=7$ in $A_4$. Four such orbits multiply to $A_4$, since $7+7>12$. These terms and eight copies of $a$ supply $A_4$ in length twelve, and then for all lengths in $[12,M+4]$. Lemma~\ref{lem:interval-completion}, with $b=168/12=14$, excludes four terms in this part. The total bound is $4+1+3=8$.

For (iv), the normalizer $N_L(C)\cong D_{10}$ contributes at most one term outside $C$, by (ii). The two other $C$-double cosets in $L$ have size $25$. Section~\ref{sec:a5-involution-product} proves that six terms in either one, together with twenty-four copies of $a$, have product set $L$. Unused copies give $L$ throughout $[30,M+6]$. Since $M-24\geq m-23\geq n/60-1$, Lemma~\ref{lem:interval-completion} applies with $H=L$ and $b=n/60$. Thus each double coset contributes at most five terms, and $c=1+2\cdot5=11$.
\end{proof}

\section{Short product-one subsequences}\label{sec:short-subsequences}

For $e\geq \meo(F)$, let $\mathsf{\eta}_e(F)$ be the least length forcing a nonempty product-one subsequence of length at most $e$. For a sequence $T$, let $\Pi(T)$ be the products of its nonempty subsequences. We use Olson's quadratic estimate \cite{Olson1975}:
\begin{equation} |\Pi(T)|\geq |T|^2/9. \tag{12.1}\label{eq:quadratic-products} \end{equation}
This estimate applies when $T$ is squarefree and product-one free. The method has two steps. A bound on concentrated terms supplies disjoint squarefree subsequences. Their product sets give a short product-one subsequence, and successive extraction imposes the required cyclic coordinate.

\begin{lemma}\label{lem:quasirandom-growth}
Let $|F|=n$, and suppose that every nontrivial irreducible complex representation of $F$ has degree at least $d$. For nonempty subsets $A,B$ with $AB\neq F$, define $\rho(X)=|X|/(n-|X|)$. Then
$$ \rho(AB)\geq d\rho(A)\rho(B). $$
In particular, if
\begin{equation} d^{r-2}\prod_{i=1}^r p_i>\prod_{i=1}^r(n-p_i), \tag{12.2}\label{eq:quasirandom-cover} \end{equation}
then $P_1\cdots P_r=F$. Here $r\geq3$, $|P_i|\geq p_i$, and $0<p_i<n$ for every $1\leq i\leq r$.
\end{lemma}

\begin{proof}
Set $a=|A|$, $b=|B|$, $Z=F\setminus AB$, and $z=|Z|$. For matrices indexed by $F$, define $M_{x,y}=\mathbf{1}_A(xy^{-1})$ and $H=M-(a/n)J$. The matrix $H$ annihilates constants, commutes with right translations, and has squared Frobenius norm $a(n-a)$. Every nonzero eigenspace of $H^*H$ is invariant under right translations and orthogonal to constants, so it has dimension at least $d$. Hence $\Vert H\Vert_{\mathrm{op}}^2\leq a(n-a)/d$.

Writing $f_X=\mathbf{1}_X-(|X|/n)\mathbf{1}$, the absence of incidences between $B$ and $Z$ gives
$$ \frac{abz}{n} =|\langle f_Z,Hf_B\rangle| \leq \sqrt{\frac{z(n-z)}{n}}\, \sqrt{\frac{a(n-a)}{d}}\, \sqrt{\frac{b(n-b)}{n}}. $$
Squaring and cancelling gives $dabz\leq(n-a)(n-b)(n-z)$, which is the first assertion.

Iterate this bound through $P_1\cdots P_{r-1}$. Unless an earlier product is already $F$, the hypothesis gives $\rho(P_1\cdots P_{r-1})\rho(P_r)>1$. Thus the last two sets have cardinalities summing to more than $n$, and their product is $F$.
\end{proof}

\begin{proposition}\label{prop:extraction-profiles}
Let $F$ be a finite group of order $n$, and let $d$ be a lower bound for the degrees of its nontrivial irreducible complex representations. Write $x_+=\max\{x,0\}$.
\begin{itemize}
\item[(i)] Suppose that every sequence containing no nonempty product-one subsequence of length at most $e_0$ admits a partition into parts of sizes $b_j$, with all occurrences of any given value in the same part, such that
$$ \sum_j(b_j-s)_+\leq E. $$
Here $s\geq1$ and $E\geq0$ are fixed integers. Choose $r,K\geq1$ with $K\geq r$, and let $k_1,\ldots,k_r$ be the balanced positive integers of sum $K$; thus their values differ by at most one. Define
\begin{equation} \begin{gathered} p_i=\left\lceil k_i^2/9\right\rceil, \qquad e=\max\{e_0,K\},\\ \text{and}\qquad L=K+(s-r)_+\left\lfloor\frac{K-1}{r}\right\rfloor+E. \end{gathered} \tag{12.3}\label{eq:profile-parameters} \end{equation}
If some $p_i\geq n$, or $r=2$ and $p_1+p_2>n$, or $r\geq3$ and the hypothesis of Lemma~\ref{lem:quasirandom-growth} is satisfied, then $\mathsf{\eta}_e(F)\leq L$.
\item[(ii)] If $\mathsf{\eta}_{e_i}(F)\leq L_i$ for every $1\leq i\leq t$, then
\begin{equation} \mathsf{d}(F\times C_t)+1\leq \max_{1\leq i\leq t}\left(L_i+\sum_{j<i}e_j\right). \tag{12.4}\label{eq:iterated-extraction} \end{equation}
\end{itemize}
\end{proposition}

\begin{proof}
For (i), let $T$ be a sequence of length $L$ without the asserted short product-one subsequence, and let $c=\sum_j\min\{b_j,r\}$. Since the parts only merge equal-value classes, one has $c\leq\sum_g\min\{v_g(T),r\}$. Summing the elementary inequality
$$ b\leq\min\{b,r\}+(s-r)_+\mathbf{1}_{b\geq r}+(b-s)_+ $$
gives
\begin{equation} |T|\leq c+(s-r)_+\lfloor c/r\rfloor+E. \tag{12.5}\label{eq:partition-capacity} \end{equation}
If $\sum_g\min\{v_g(T),r\}\leq K-1$, then $c\leq K-1$, and \eqref{eq:partition-capacity} would give $|T|\leq L-1$, contrary to $|T|=L$. Hence $\sum_g\min\{v_g(T),r\}\geq K$. This sum is the maximum number of terms that can be selected with every multiplicity at most $r$: each value supplies independently up to $r$ of its available occurrences. Choose $K$ such terms, list equal values consecutively, and distribute them cyclically among $r$ subsequences. These subsequences are squarefree and have sizes $k_i$. If one has a product-one subsequence, then its length is at most $K\leq e$. Otherwise, Olson's estimate \cite{Olson1975} gives at least $p_i$ nonempty products from the $i$th subsequence. The relevant cardinality criterion produces a product equal to $1$ from disjoint subsequences, again with length at most $K$.

For (ii), successively extract $t$ disjoint subsequences with selected $F$-product $1$, the $i$th of length at most $e_i$. The stated maximum ensures that at least $L_i$ terms remain at stage $i$. Their selected products in $F\times C_t$ have the form $(1,z_i)$. Some nonempty subsequence of $z_1,\ldots,z_t$ has product $1$ in $C_t$, giving the required product-one subsequence.
\end{proof}

\subsection{Bounds for concentrated terms}\label{sec:concentration-bounds}

The simplest data in Proposition~\ref{prop:extraction-profiles}(i) are
$$ e_0=\meo(F),\qquad s=\meo(F)-1,\qquad\text{and}\qquad E=0. $$
For these data, take a part for each element value. For balanced subsequences of size $k$, this gives the squarefree-allocation threshold $\max\{s,r\}(k-1)+r$. A sharper excess estimate is available when the large multiplicities lie in cyclic root groups.

\begin{lemma}\label{lem:root-concentration}
Let $p\geq5$ be prime. Suppose that a sequence over $\PSL_2(p)$ contains no nonempty product-one subsequence of length at most $3(p-1)$. Let $b_x$ count its nonidentity terms in the root group $U_x\cong C_p$, where $x\in\mathbb{P}^1(\mathbb{F}_p)$. Then
\begin{equation} \sum_x(b_x-s)_+ \leq\left\lfloor\frac{(3p-3s-1)^2}{12}\right\rfloor. \tag{12.6}\label{eq:root-excess-bound} \end{equation}
This holds for every $0\leq s\leq p-1$.
\end{lemma}

\begin{proof}
Let $P_x$ be the nonempty products from the part in $U_x$. Distinct prefix products imply that $b_x\leq p-1$, $1\notin P_x$, and $|P_x|\geq b_x$. For every ordered triple of distinct roots, there is exactly one triple of nonidentity root elements with product $1$. To verify this, move the first root to infinity and define
$$ u(a)=\begin{pmatrix}1&a\\0&1\end{pmatrix} \qquad\text{and}\qquad v_y(b)=\begin{pmatrix}1+yb&-y^2b\\b&1-yb\end{pmatrix}. $$
One has
$$ u\bigl(2(z-y)\bigr)\, v_y\bigl(-2/(z-y)\bigr)\, v_z\bigl(-2/(z-y)\bigr)=-I $$
whenever $y\neq z$. Conversely, compare $u(a)v_y(b)$ with $\varepsilon v_z(c)^{-1}$, where $\varepsilon=\pm1$. The negative sign forces $b=c$, $2=(y-z)b$, and $a=2(z-y)$. The positive sign forces $b=-c$ and $(z-y)b=0$, which is impossible. In particular, the first element and the next root determine the triple.

Let $H=\{x:b_x>s\}$, $h=|H|$, and $M=\sum_{x\in H}(p-1-|P_x|)$. Ordered triples of roots in $H$, modulo cyclic rotation, give $h(h-1)(h-2)/3$ product-one triples. Each must contain a vertex missing from its $P_x$, since otherwise three disjoint root parts produce a forbidden product of length at most $3(p-1)$. A missing vertex belongs to at most $h-1$ triples, by the uniqueness just proved. Thus, for $h\geq2$, one has $M\geq h(h-2)/3$, and
$$ \begin{aligned} \sum_x(b_x-s)_+ &\leq h(p-1-s)-M\\ &\leq\frac{h(3p-3s-1-h)}{3} \leq\frac{(3p-3s-1)^2}{12}. \end{aligned} $$
The cases $h=0,1$ satisfy the last bound directly. Integrality proves the assertion.
\end{proof}

For $F=\PSL_2(p)$, set $s=(p-1)/2$. Root parts contain all unipotent terms, and every remaining value has order dividing $s$ or $s+1$, so its multiplicity is at most $s$. We may therefore use
\begin{equation} e_0=3(p-1),\qquad s=(p-1)/2,\qquad\text{and}\qquad E=\left\lfloor\frac{(3p-3s-1)^2}{12}\right\rfloor. \tag{12.7}\label{eq:root-profile} \end{equation}
The coarse choice $e_0=p$ and $E=s(p+1)$ is also valid, since each of the $p+1$ root parts contains at most $2s$ terms. For $r\leq s$ and $K=rk$, the threshold is $s(k+p)+r$ by Proposition~\ref{prop:extraction-profiles}(i). Thus the same allocation argument applies over prime fields.

For $M_{11}$, use its degree-twelve action with point stabilizer $\PSL_2(11)$ \cite{Wilson2009}. Its element orders are $1,2,3,4,5,6,8,11$, as established in the proof of Proposition~\ref{prop:m11-extraction}. An element of order eleven has cycle type $11\,1$ and therefore belongs to exactly one point stabilizer. Its terms partition among the twelve stabilizers. In each stabilizer, Lemma~\ref{lem:root-concentration}, with $p=11$ and $s=7$, bounds the root excess by ten. Other element values occur at most seven times. Consequently,
$$ e_0=30,\qquad s=7,\qquad\text{and}\qquad E=12\cdot10=120 $$
are valid data for Proposition~\ref{prop:extraction-profiles}(i).

\subsection{Representation degrees from weights}\label{sec:representation-degrees}

\begin{lemma}\label{lem:representation-degrees}
The following lower bounds hold for the degrees of nontrivial irreducible complex representations. In the first two rows, $q\geq4$.
\begin{center}
\begin{tabular}{c c}
\hline
$F$&Degree at least\\
\hline
$\PSL_2(q)$&$(q-1)/\gcd(2,q-1)$\\
$\PSL_2(q)$, where $q\equiv1\pmod{4}$&$(q+1)/2$\\
$A_7$&$5$\\
$A_8$&$7$\\
$\PSU_3(3)$&$6$\\
\hline
\end{tabular}
\end{center}
\end{lemma}

\begin{proof}
All groups are non-abelian simple, so their nontrivial representations are faithful. If a subgroup normalizes an abelian group $U$, then any nontrivial $U$-weight occurring in a representation brings its entire conjugation orbit, with equal positive multiplicities.

For $\PSL_2(q)$, take the upper root group $U\cong(\mathbb{F}_q,+)$. Fix a nontrivial additive character $\psi$. The characters of $U$ are $x\mapsto\psi(cx)$, with $c\in\mathbb{F}_q$. The split torus multiplies nonzero parameters by squares, giving orbits of size $h=(q-1)/\gcd(2,q-1)$. This proves the first bound. If $q\equiv1\pmod{4}$ and the degree were $h=(q-1)/2$, then the restriction would consist of one orbit of $h$ one-dimensional weight spaces. A generator of the order-$h$ torus permutes these spaces in an $h$-cycle. Its $h$th power is one, so its transition scalars have product one and its determinant is $(-1)^{h-1}=-1$. The determinant character of a non-abelian simple group is trivial, a contradiction. The degree is therefore at least $h+1$.

For $A_7$, restrict to $A_6\cong\PSL_2(9)$ and apply the bound five just proved to a nontrivial irreducible constituent. For $A_8$, use $U\rtimes\GL_3(2)$ with $U=C_2^3$. This group embeds in $A_8$: translations are products of four transpositions, and generating transvections are products of two transpositions. The linear group is transitive on the seven nontrivial weights of $U$. This proves the two alternating-group bounds.

For $\PSU_3(3)$, a Borel subgroup contains the group
$$ U=\left\{ \begin{pmatrix} 1&a&b\\ 0&1&-a^3\\ 0&0&1 \end{pmatrix} :a,b\in\mathbb{F}_9,\ b+b^3+a^4=0 \right\}, $$
which is a Heisenberg group of order $27$. Its center has $a=0$, $b+b^3=0$, and order three. If $\lambda$ generates $\mathbb{F}_9^\times$, then $\diag(\lambda,\lambda^2,\lambda^{-3})$ normalizes $U$ and sends $b$ to $\lambda^4b=-b$. Thus both nontrivial central weights occur, with equal multiplicity. In either weight space, the Heisenberg commutator relation makes one generator cyclically permute the three eigenspaces of another. The dimension of that weight space is a positive multiple of three. The two central weights therefore require at least six dimensions.
\end{proof}

\subsection{A criterion from distinct terms}\label{sec:distinct-term-criterion}

\begin{proposition}\label{prop:distinct-term-criterion}
Let $G$ be non-abelian simple, $n=|G|=2m$, $u=\meo(G)$, and $b=\lfloor3\sqrt{n}\rfloor$. If
\begin{equation} (2u-2)b+1\leq n/2, \tag{12.8}\label{eq:meo-extraction-criterion} \end{equation}
then $\mathsf{E}(G)\leq3n/2-1$. Either $n\geq144(u-1)^2+4$ or the stronger condition $n\geq144u^2$ suffices. For a fixed repeated order $t$, it is enough that $(u+t-2)b+1\leq m$.
\end{proposition}

\begin{proof}
Olson's subset theorem \cite{Olson1975} implies that any $b+1$ distinct elements contain a nonempty product-one subsequence. Every $(u-1)b+1$ terms therefore contain one of length at most $e=\max\{u,b+1\}$: either one value occurs $u$ times, or there are $b+1$ distinct values. By Proposition~\ref{prop:extraction-profiles}(ii),
$$ \mathsf{d}(G\times C_t)+1\leq(u-1)b+1+(t-1)e. $$
The hypothesis forces $b\geq u-1$: otherwise, its left side is at least $2b(b+1)+1>n/2$. Thus $e=b+1$, and
\begin{equation} \mathsf{d}(G\times C_t)+1\leq(u+t-2)b+t. \tag{12.9}\label{eq:distinct-term-davenport} \end{equation}
Since $t\leq u$, Lemma~\ref{lem:high-mult-completion} proves the conclusion when $h(S)>m$. Theorem~\ref{thm:bounded-height} applies when $h(S)\leq m$. The fixed-order condition likewise forces $b\geq u-1$ and gives the asserted bound.

Finally, let $c=u-1$. If $n\geq144c^2+4$, then
$$ (n-2)^2-144c^2n=n(n-144c^2-4)+4>0. $$
It follows that $6c\sqrt{n}+1<n/2$, which implies the required inequality.
\end{proof}

\subsection{Applications of the profiles}\label{sec:profile-applications}

For the rows with $E=0$, the required maximum element orders are
$$ \meo(\PSL_2(25))=13,\qquad \meo(\PSL_2(27))=14,\qquad\text{and}\qquad \meo(\PSU_3(3))=12, $$
$$ \meo(A_7)=7 \qquad\text{and}\qquad \meo(A_8)=15. $$
The first two follow from the unipotent and torus orders, and the last two from cycle decompositions. The unitary element-order list accompanies Proposition~\ref{prop:psl3-psu3-extraction}. These supply $e_0$ and $s=e_0-1$.

The final unrestricted subsequences use only subgroup extraction. For every odd prime power $q\geq5$, a Borel subgroup $(\mathbb{F}_q,+)\rtimes C_{(q-1)/2}$, Lemma~\ref{lem:frobenius-extraction}, and the parity improvement in Lemma~\ref{lem:extraction-bounds} give
\begin{equation} \mathsf{d}(\PSL_2(q))+1 \leq q\bigl(q+(q-1)/2-1\bigr)=3q(q-1)/2. \tag{12.10}\label{eq:psl2-davenport} \end{equation}
Also, $\mathsf{d}(A_5)=8$ by \cite{AndrasCziszterDomokosSzollosi2025}; the even degree-six action and then the index-seven subgroup give $\mathsf{d}(A_6)+1\leq45$ and $\mathsf{d}(A_7)+1\leq315$. The point stabilizer $M_{10}=A_6.2$ in $M_{11}$ contains an $A_6$ \cite{WilsonEtAlND}. This subgroup has index $22$, so the even coset action gives $\mathsf{d}(M_{11})+1\leq21\cdot45=945$. In the table, $D$ denotes any of these unrestricted bounds; it may be used more than once.

\begin{proposition}\label{prop:profile-completion}
Every pair $(F,t)$ in Table~\ref{tab:extraction-profiles} satisfies the high-multiplicity criterion of Lemma~\ref{lem:high-mult-completion}. For an unstarred row, the displayed $B$ bounds $\mathsf{d}(F\times C_t)+1$. The starred row instead forces a product-one subsequence of length divisible by $t$ and at most $|F|/2$ from any $B$ terms over $F$. In addition, $\mathsf{d}(A_6\times C_4)+1\leq180$, by Corollary~\ref{cor:a6-order-four}. Thus $(A_6,4)$, $(A_7,5)$, and $(A_7,6)$ are all covered.
\end{proposition}

In the table, $L_q=\PSL_2(q)$ and $U=\PSU_3(3)$. The first column lists $(n,s,E,d)$; a dash means that no representation degree is needed. The admissible $e_0$ is as proved above. Profiles are numbered consecutively and record $(r,K,L)$, so $\mathsf{\eta}_K(F)\leq L$. A word such as $1^2\,2\,3$ specifies two extractions using profile $1$, followed by profiles $2$ and $3$. An occurrence of $D$ means $(e,L)=(D,D)$.

The first five rows use one pair of squarefree subsequences, followed by two unrestricted extractions. Their common word and bound are
\begin{equation} 1^{t-2}D^2 \qquad\text{and}\qquad B=\max\{L+(t-3)K,(t-2)K+2D\}. \tag{12.11}\label{eq:extraction-word} \end{equation}
Only the elementary condition $2\lceil(K/2)^2/9\rceil>n$ is needed. The next two rows use pairs and then four squarefree subsequences, with common word $1^{t-2}\,2D$. For their profile with four subsequences, $d=9$, so the mixing condition reduces to $4\lceil(K/4)^2/9\rceil>n$.

\begin{table}[htbp]
\caption{Short-subsequence extraction, grouped by its product-set criterion.}\label{tab:extraction-profiles}
\centering
\small
\setlength{\tabcolsep}{7pt}
\begin{tabular}{@{}c c c@{}}
\toprule
$F$ and $(n,s,E,d)$&Profiles $(r,K,L)$&$t$: extraction word $\mapsto B$\\
\midrule
\begin{tabular}{@{}c@{}}$L_{23}$\\$(6072,11,102,-)$\end{tabular}&
$(2,332,1919)$&
\begin{tabular}{@{}c@{}}$6:\ 1^4D^2\mapsto2915$\\$D=759$\end{tabular}\\[3pt]
\begin{tabular}{@{}c@{}}$L_{25}$\\$(7800,12,0,-)$\end{tabular}&
$(2,376,2246)$&
\begin{tabular}{@{}c@{}}$6:\ 1^4D^2\mapsto3374$\\$D=900$\end{tabular}\\[3pt]
\begin{tabular}{@{}c@{}}$L_{27}$\\$(9828,13,0,-)$\end{tabular}&
$(2,422,2732)$&
\begin{tabular}{@{}c@{}}$7:\ 1^5D^2\mapsto4420$\\$D=1053$\end{tabular}\\[3pt]
\begin{tabular}{@{}c@{}}$L_{29}$\\$(12180,14,161,-)$\end{tabular}&
$(2,470,3439)$&
\begin{tabular}{@{}c@{}}$7:\ 1^5D^2\mapsto5319$\\$D=1218$\end{tabular}\\[3pt]
\begin{tabular}{@{}c@{}}$L_{31}$\\$(14880,15,184,-)$\end{tabular}&
$(2,518,4056)$&
\begin{tabular}{@{}c@{}}$8:\ 1^6D^2\mapsto6646$\\$D=1395$\end{tabular}\\
\midrule
\begin{tabular}{@{}c@{}}$L_{17}$\\$(2448,8,56,9)$\end{tabular}&
\begin{tabular}{@{}c@{}}$(2,210,890)$\\$(4,300,652)$\end{tabular}&
\begin{tabular}{@{}c@{}}$4:\ 1^2\,2\,D\mapsto1128$\\$D=408$\end{tabular}\\[3pt]
\begin{tabular}{@{}c@{}}$L_{19}$\\$(3420,9,70,9)$\end{tabular}&
\begin{tabular}{@{}c@{}}$(2,250,1188)$\\$(4,352,857)$\end{tabular}&
\begin{tabular}{@{}c@{}}$5:\ 1^3\,2\,D\mapsto1688$\\$D=513$\end{tabular}\\
\midrule
\begin{tabular}{@{}c@{}}$M_{11}$\\$(7920,7,120,-)$\end{tabular}&
\begin{tabular}{@{}c@{}}$(1,267,1983)$\\$(2,378,1438)$\end{tabular}&
\begin{tabular}{@{}c@{}}$11:\ 1^8\,2^2\,D\mapsto3952$\\$D=945$\end{tabular}\\[3pt]
\begin{tabular}{@{}c@{}}$U$\\$(6048,11,0,6)$\end{tabular}&
\begin{tabular}{@{}c@{}}$(1,234,2564)$\\$(2,330,1806)$\\$(3,417,1521)$\\$(9,945,1153)$\end{tabular}&
$7:\ 1^2\,2^3\,3\,4\mapsto3028$\\[3pt]
\begin{tabular}{@{}c@{}}$A_7$\\$(2520,6,0,5)$\end{tabular}&
\begin{tabular}{@{}c@{}}$(1,151,901)$\\$(2,214,638)$\\$(3,276,549)$\end{tabular}&
\begin{tabular}{@{}c@{}}$4:\ 1^2\,2\,D\mapsto1052$\\$5:\ 1^3\,2\,D\mapsto1203$\\$6:\ 1^3\,2\,3\,D\mapsto1258$\\$D=315$\end{tabular}\\[3pt]
\begin{tabular}{@{}c@{}}$A_8$\\$(20160,14,0,7)$\end{tabular}&
\begin{tabular}{@{}c@{}}$(1,426,5951)$\\$(2,604,4216)$\\$(4,896,3126)$\end{tabular}&
$15:\ 1^{10}\,2^3\,3^2\mapsto10094$\\
\midrule
\begin{tabular}{@{}c@{}}$L_{13}$\\$(1092,6,33,7)$\end{tabular}&
\begin{tabular}{@{}c@{}}$(2,141,454)$\\$(3,175,382)$\end{tabular}&
\begin{tabular}{@{}c@{}}$3:\ 1\,2\,D\mapsto548^*$\\$D=234$\end{tabular}\\
\bottomrule
\end{tabular}
\end{table}

\begin{proof}[Proof of Proposition~\ref{prop:profile-completion}]
The preceding concentration and weight arguments establish the data in the first column. Every $K$ is at least its required $e_0$. For each displayed $(r,K,L)$, the stated threshold follows from Proposition~\ref{prop:extraction-profiles}(i), and the corresponding balanced subsequence sizes satisfy the product-set criterion. This verification uses only integer arithmetic. For example, $(r,K)=(2,141)$ gives sizes $71,70$ and $561+545>1092$, whereas $(3,175)$ gives sizes $59,58,58$ and
$$ 7\cdot387\cdot374^2=378924084 >363444420=705\cdot718^2. $$
For every balanced vector $(k,\ldots,k)$ in the table, let $P=\lceil k^2/9\rceil$. The required checks are
$$ \begin{array}{ll} P\geq n,&\text{when }r=1,\\ 2P>n,&\text{when }r=2,\\ d^{r-2}P^r>(n-P)^r,&\text{when }r\geq3. \end{array} $$
Thus Proposition~\ref{prop:extraction-profiles}(i) proves all listed profiles. For the first five rows, \eqref{eq:extraction-word} gives the displayed $B$; for the other unstarred rows, apply Proposition~\ref{prop:extraction-profiles}(ii) to the displayed extraction words. All satisfy $B\leq n/2+t-1$.

For the starred row $L_{13}$, start with $548$ terms and extract a product-one subsequence of length $\ell_1\leq141$ by profile $1$. If its length is divisible by three, then the required extraction is complete. Otherwise, $\ell_1\leq140$, so at least $408\geq382$ terms remain, and profile $2$ gives a second subsequence of length $\ell_2\leq175$. If neither subsequence nor their union has length divisible by three, then the two lengths have the same nonzero residue. Therefore
$$ \ell_1+\ell_2\leq\max\{139+175,140+173\}=314. $$
At least $234$ terms remain, and the unrestricted bound $D=234$ gives a third product-one subsequence. Among their three lengths, some nonempty selection sums to zero modulo three. Concatenating the corresponding product-one orderings gives a length divisible by three and at most $3\lfloor548/3\rfloor=546=n/2$. Since $548=n/2+3-1$, the strengthened extraction form in the proof of Lemma~\ref{lem:high-mult-completion}(i) applies to this row as well.
\end{proof}

\section{Conjugation-orbit growth}\label{sec:orbit-growth}

The sequence applications assume \eqref{eq:high-mult-setup} and proceed by contradiction, so that $1\notin\Pi_n(S)$. Let $C=\langle a\rangle$, let $t=\ord(a)$, and define
$$ O_a(x)=\{a^ixa^{-i}:0\leq i<t\}. $$
Moving $x$ among $t-1$ copies of $a$ gives the exact identity
\begin{equation} \pi(a^{[t-1]}x)=O_a(x)a^{-1}. \tag{13.1}\label{eq:orbit-realization} \end{equation}
We also write $O(x)=O_a(x)$, or simply $O$ when $x$ is fixed. Each orbit commutes with $a$ as a set. Moreover, $O_a(a^ux a^v)=a^{u+v}O_a(x)$. Thus $r$ terms from one double coset $CxC$, together with $r(t-1)$ copies of $a$, realize a translate of $O_a(x)^r$ as part of their product set. Translations preserve the size of a difference set and preserve its generation property up to conjugation.

A product of $w$ orbit factors uses $w$ remainder terms and $(t-1)w$ copies of $a$, hence $tw$ terms in all. If an initial orbit product contains $C$, then restrict that product to $C$. Subsequent products $A$ are bi-$C$-invariant and satisfy $AO=A(CO)$.

In the matrix applications, $G=\PSL_2(q)$ and $\varepsilon=\gcd(2,q-1)$.

\begin{lemma}\label{lem:orbit-generation}
Let $G$ be non-abelian simple, and suppose that $\langle C,x\rangle=G$. Suppose that every maximal subgroup containing $\langle C,xCx^{-1}\rangle$ is the only member of its $G$-conjugacy class containing $C$. Then
$$ \langle C,xCx^{-1}\rangle=\langle O_a(x)O_a(x)^{-1}\rangle=G. $$
The orbit-difference conclusion also holds for every positive power of $O_a(x)$. In particular, if each conjugacy class of maximal subgroups has at most one member containing $C$, then the hypothesis holds.
\end{lemma}

\begin{proof}
Write $O=O_a(x)$, $H=\langle OO^{-1}\rangle$, and $K=\langle C,xCx^{-1}\rangle$. Conjugation by $C$ normalizes $H$, and
$$ xa^ix^{-1}a^{-i}\in H \quad\text{and}\quad (a^ixa^{-i})(a^jxa^{-j})^{-1}=a^ixa^{j-i}x^{-1}a^{-j}\in K. $$
It follows that $K=HC$, $H\trianglelefteq K$, and $K/H$ is cyclic. If $K<G$, then choose a maximal subgroup $L$ containing $K$. Both $L$ and $x^{-1}Lx$ contain $C$, so the hypothesis implies that $x^{-1}Lx=L$. A maximal subgroup of a non-abelian simple group is self-normalizing: its normalizer is either the subgroup itself or the whole group, and the latter possibility would make the subgroup normal. Thus $x\in L$, contrary to $\langle C,x\rangle=G$. Hence $K=G$. Simplicity and the cyclic quotient $G/H$ imply that $H=G$.

Finally, fixing $v\in O^{r-1}$ gives $(uv)(wv)^{-1}=uw^{-1}$ for $u,w\in O$, whence $OO^{-1}\subseteq O^r(O^r)^{-1}$.
\end{proof}

\begin{lemma}\label{lem:full-entry-orbits}
Let $a\in G=\PSL_2(q)$ be semisimple of order $t>2$, and let $\mu=\langle\lambda^2\rangle$, where $\lambda,\lambda^{-1}$ are the eigenvalues of a determinant-one lift. Suppose that the orbit $O=O_a(x)$ has representatives
$$ X_u= \begin{pmatrix} \alpha&\beta u\\ \gamma/u&\delta \end{pmatrix} $$
in this eigenbasis, where $u\in\mu$, $\alpha\delta-\beta\gamma=1$, and $\alpha\beta\gamma\delta\neq0$. Then $|\mu|=|O|=t$. If $q$ is odd, then
$$ |OO^{-1}|\geq(t-1)^2 \quad\text{and}\quad |O^2|\geq(t-1)^2. $$
If $q$ is even, then $|OO^{-1}|=t^2-t+1$ and $|O^2|\geq t^2-t+1$. In particular, these bounds apply whenever $a$ is nonsplit and $\langle a,x\rangle=G$.
\end{lemma}

\begin{proof}
Distinct values of $u$ give projectively distinct matrices. Writing $r=u/v$, multiplication gives
\begin{equation} X_uX_v^{-1}= \begin{pmatrix} \alpha\delta-\beta\gamma r&\alpha\beta(u-v)\\ \gamma\delta(u^{-1}-v^{-1})&\alpha\delta-\beta\gamma r^{-1} \end{pmatrix}. \tag{13.2}\label{eq:orbit-differences} \end{equation}
For $r\neq1$, the upper-left and upper-right entries recover $r$ and $u-v$, hence $u,v$. There are therefore $t(t-1)$ distinct nonscalar matrices. In odd characteristic, equality to the negative of a matrix with ratio $s$ requires
$$ r+s=\frac{2\alpha\delta}{\beta\gamma}=r^{-1}+s^{-1} \quad\text{and}\quad rs=1. $$
At most two ratios are involved, hence at most $2t$ matrices. Pairing negatives loses at most $t$ images. Including the identity gives $1+t(t-1)-t=(t-1)^2$. In characteristic two there is no loss.

The square products have the common parametrization
\begin{equation} X_{rz}X_z= \begin{pmatrix} \alpha^2+\beta\gamma r&\beta(\alpha+\delta r)z\\ \gamma(\alpha/r+\delta)/z&\delta^2+\beta\gamma/r \end{pmatrix}. \tag{13.3}\label{eq:orbit-squares} \end{equation}
Here $r,z\in\mu$. The upper-left entry recovers $r$, and then the upper-right entry recovers $z$, except possibly when $r=-\alpha/\delta$. Thus there are at least $t^2-t+1$ distinct matrices. In odd characteristic, an identification up to sign between ratios $r,s$ requires
$$ r+s=-\frac{2\alpha^2}{\beta\gamma} \quad\text{and}\quad rs=\frac{\alpha^2}{\delta^2}. $$
Again at most two ratios and $2t$ matrices are involved, so at most $t$ further images are lost. This proves the square bounds.

For a nonsplit $a$, use a Frobenius-conjugate eigenbasis, in which $\delta=\alpha^q$ and $\gamma=\beta^q$. A zero entry would then make $x$ diagonal or antidiagonal, contrary to $\langle a,x\rangle=G$.
\end{proof}

\begin{lemma}\label{lem:zero-diagonal-pairs}
Let $C=\langle a\rangle$ be a split semisimple subgroup of $G=\PSL_2(q)$, of order $t>2$, and let $\varepsilon=\gcd(2,q-1)$. Let $O_1,O_2$ be two $C$-conjugation orbits whose representatives have exactly one zero diagonal entry, in the same position. Write $B=O_1O_2$. Then
\begin{equation} |B|=t^2,\qquad |BB^{-1}|\geq1+\frac{t^2(t-1)}{\varepsilon}, \quad\text{and}\quad |CB|=t^3. \tag{13.4}\label{eq:paired-orbit-bounds} \end{equation}
If $\mathbb{F}_p[\mu]=\mathbb{F}_q$, where $\mu=\langle\lambda^2\rangle$ is the orbit-parameter subgroup and $q$ is a power of $p$, then $\langle BB^{-1}\rangle=G$. This field condition holds when $q$ is prime or $t>\sqrt{q}-1$.
\end{lemma}

\begin{proof}
Interchange the eigenlines when necessary, and write
$$ X_u= \begin{pmatrix}0&\beta u\\ \gamma/u&d\end{pmatrix} \quad\text{and}\quad Y_v= \begin{pmatrix}0&\beta'v\\ \gamma'/v&d'\end{pmatrix}. $$
Here $u,v\in\mu$, $\beta\gamma=\beta'\gamma'=-1$, and $dd'\neq0$. Let $k=\beta\gamma'$, $D=dd'k$, and $A=\beta d'k$. Right multiplication of $X_uY_v$ by $\diag(k^{-1},k)$, followed by conjugation by $\diag(A^{-1},1)$, gives
\begin{equation} B=\left\{B(r,z)= \begin{pmatrix} r&z\\ Dr/z&D+r^{-1} \end{pmatrix}:r,z\in\mu\right\}. \tag{13.5}\label{eq:paired-orbit-form} \end{equation}
Here $D\neq0$, $r=u/v$, and $z=u$. These operations preserve all three cardinalities in the statement and conjugate the difference subgroup. The matrices have distinct projective images, since equality recovers $r,z$. In odd characteristic, an additional identification up to a minus sign would require $(r,z)\mapsto(-r,-z)$, contradicting the nonzero lower-left entry.

For $z/w=\ell\neq1$, the entries of $P=B(r,z)B(r,w)^{-1}$ satisfy
$$ P_{11}=1+Dr(1-\ell),\qquad P_{12}=rw(\ell-1), \quad\text{and}\quad P_{22}=1+Dr(1-\ell^{-1}). $$
These entries recover successively $\ell=(1-P_{11})/(P_{22}-1)$, $r$, and $w$. Thus this construction gives $t^2(t-1)$ distinct nonscalar matrices, and projectivization loses a factor of at most $\varepsilon$. Adding the identity proves the difference bound. The diagonal product $1+Dr$ takes $t$ distinct values. For fixed $r$, the upper-left and upper-right entries show that the $t^2$ projective matrices in $C\{B(r,z):z\in\mu\}$ are distinct. Therefore $|CB|=t^3$.

To prove generation, fix a ratio $u/v\neq1$ in a single orbit's differences. Their diagonal entries are fixed, while their lower-left entries run through a nonzero multiple of $\mu$. Ratios of these differences give lower unipotents with parameters in a nonzero multiple of $\mu-\mu$. Its nonzero $\mathbb{F}_p$-span is invariant under $\mu$, hence under $\mathbb{F}_p[\mu]=\mathbb{F}_q$, and therefore equals $\mathbb{F}_q$. Thus each single-orbit difference subgroup contains the full lower root group.

By fixing one factor at a time, we see that $BB^{-1}$ contains the first orbit's differences and the second orbit's differences conjugated by a member of the first orbit. The conjugating matrix moves the point fixed by the lower root group. Root groups at two distinct points generate $\PSL_2(q)$ by elementary row operations. Finally, a proper subfield of $\mathbb{F}_q$ has multiplicative group of order at most $\sqrt{q}-1$, which proves the sufficient field condition.
\end{proof}

\begin{lemma}\label{lem:cyclic-amplification}
Let $C=\langle a\rangle\leq\PSL_2(q)$ be semisimple of order $t>2$, and let $T$ be the full torus containing $C$. If $x\notin N_G(T)$, then $|CxC|=t^2$. Let $D=CO_a(x)=CxC$. For every bi-$C$-invariant set $A$,
$$ AO_a(x)=AD,\qquad DD^{-1}=CxCx^{-1}C, \quad\text{and}\quad \langle DD^{-1}\rangle=\langle C,xCx^{-1}\rangle. $$
If all four entries of $x$ in a torus eigenbasis are nonzero, then
\begin{equation} |DD^{-1}|\geq t+\left\lfloor\frac{t-1}{2}\right\rfloor t^2. \tag{13.6}\label{eq:amplified-differences} \end{equation}
\end{lemma}

\begin{proof}
Two distinct conjugates of $C$ intersect trivially. Indeed, a common element of order greater than two has a unique torus centralizer. For a common involution, its centralizer is the torus normalizer; every element outside its torus has order two, so a cyclic subgroup of order $t>2$ lies in the torus. Either possibility forces the two full tori, and hence their order-$t$ subgroups, to coincide. Thus $C\cap xCx^{-1}=1$ and $|CxC|=t^2$. The displayed set identities follow from $CO_a(x)=O_a(x)C=CxC$.

For the last bound, let $k=\alpha\delta$ in the coordinates of Lemma~\ref{lem:full-entry-orbits}. The diagonal product of $X_uX_v^{-1}$ is
\begin{equation} 1+k(k-1)(2-r-r^{-1}). \tag{13.7}\label{eq:double-coset-invariant} \end{equation}
Here $r=u/v$. For $r\neq1$, the off-diagonal entries of $X_uX_v^{-1}$ are nonzero. Both diagonal entries can vanish only when $r=-1$ and $k=1/2$ in odd characteristic. Discard $r=-1$ when $t$ is even. The remaining $\lfloor(t-1)/2\rfloor$ inverse pairs give distinct values of this diagonal product. The product is invariant under left and right multiplication by $C$, as well as projective signs. Each value therefore supplies a different double coset of size $t^2$ inside the bi-$C$-invariant set $DD^{-1}$. These double cosets are disjoint from $C\subseteq DD^{-1}$, which proves the bound.
\end{proof}

\begin{lemma}\label{lem:squared-orbit-bounds}
Let $G=\PSL_2(q)$, let $C=\langle a\rangle$ be a full torus of order $t$, and suppose that $\langle C,x\rangle=G$. For the parameters in the following table, $B=O_a(x)^2$ satisfies $|B|\geq b$, $|BB^{-1}|\geq d$, and $\langle BB^{-1}\rangle=G$:
\begin{equation} \begin{array}{rrrr} q&t&b&d\\ \hline 11&6&25&76\\ 13&7&36&99\\ 17&8&49&160 \end{array} \tag{13.8}\label{eq:squared-orbit-params} \end{equation}
\end{lemma}

\begin{proof}
The first two tori are nonsplit; the third is split. Use the orbit coordinates of Lemma~\ref{lem:full-entry-orbits}. First suppose that all four entries of $x$ are nonzero, and let $k=\alpha\delta\in\mathbb{F}_q\setminus\{0,1\}$. By Lemma~\ref{lem:full-entry-orbits}, $|B|\geq(t-1)^2=b$. Write the square products as
$$ B(r,z)= \begin{pmatrix} A_r&b_rz\\ d_r/z&D_r \end{pmatrix} \quad\text{and}\quad c_r=b_rd_r=\frac{k-1}{r}(\alpha+\delta r)^2. $$
For $\ell=z/w\neq1$, the two diagonal entries of $B(r,z)B(r,w)^{-1}$ are
\begin{equation} (P,Q)=\bigl(1+c_r(1-\ell),\,1+c_r(1-\ell^{-1})\bigr). \tag{13.9}\label{eq:diagonal-differences} \end{equation}
The off-diagonal entries are $A_rb_r(z-w)$ and $D_rd_r(z^{-1}-w^{-1})$, whose product is
\begin{equation} c_r(c_r+1)(2-\ell-\ell^{-1}). \tag{13.10}\label{eq:off-diagonal-product} \end{equation}

We count the diagonal pairs using only arithmetic in $\mathbb{F}_q$. Let $\chi$ be its quadratic character, and define
$$ H_\pm=\{u+u^{-1}:u\in T_0,\ u^t=\pm1\} \quad\text{and}\quad \Sigma=H_+\setminus\{2\}. $$
Here $T_0$ is the norm-one group of order $q+1$ in the nonsplit cases, and $T_0=\mathbb{F}_{17}^{\times}$ in the split case. Thus $\mu=\{u\in T_0:u^t=1\}$. The sets needed here are
$$ \begin{array}{rcc} q&H_+&H_-\\ \hline 11&\{1,2,9,10\}&\{0,5,6\}\\ 13&\{2,7,8,10\}&\{3,5,6,11\}\\ 17&\{0,2,6,11,15\}&\{5,8,9,12\} \end{array} $$
These lists follow by solving $u^2-hu+1=0$, or directly by taking powers of $3+2w$, $9+w$, and $3$, respectively, where $w^2=2$ in the first two fields.

With $s=\delta r/\alpha$, the nonzero values of $c_r$ form
\begin{equation} V_k=\{k(k-1)(h+2):h\in H_{\chi(k)}\}\setminus\{0\}. \tag{13.11}\label{eq:scalar-orbit-values} \end{equation}
Indeed, in the nonsplit cases $\delta=\alpha^q$ and $(\delta/\alpha)^t=\alpha^{(q^2-1)/2}=\chi(k)$. In the split case $\delta/\alpha=k/\alpha^2$, so it lies in the same square class as $k$.

For completeness, the following scalar rule counts the diagonal pairs $(P,Q)$ up to simultaneous sign. For $c\in V_k$ and $h\in\Sigma$, define
$$ T=2+c(2-h) \quad\text{and}\quad N=1+c(c+1)(2-h), $$
and let
$$ J_k=\{(T^2,N):c\in V_k,\ h\in\Sigma\}. $$
Then the number of pairs is exactly
\begin{equation} N_k=\sum_{(u,v)\in J_k} \begin{cases} 1,&\text{when }u=0\text{ or }u=4v,\\ 2,&\text{otherwise}. \end{cases} \tag{13.12}\label{eq:diagonal-pair-count} \end{equation}
To see this, observe that $T=P+Q$ and $N=PQ$. Hence $(T^2,N)$ determines $(P,Q)$ up to simultaneous sign and interchange. Interchanging $\ell$ and $\ell^{-1}$ realizes both orders. These orders give one pair precisely when $P=-Q$ or $P=Q$, that is, when $T=0$ or $T^2=4N$; otherwise they give two pairs. This also shows that no unlisted identifications occur.

The complete scalar calculations are small:
$$ \begin{array}{ccc} \multicolumn{3}{c}{q=11}\\ k&V_k&N_k\\ \hline 2,10&\{3,4,5\}&14\\ 3,9&\{2,6,7\}&15\\ 4&\{1,3,4\}&14\\ 5&\{3,5,9\}&13\\ 6&\{1,5,9\}&13\\ 7&\{6,7,8\}&15\\ 8&\{2,7,8\}&14 \end{array} \qquad\text{and}\qquad \begin{array}{ccc} \multicolumn{3}{c}{q=13}\\ k&V_k&N_k\\ \hline 2&\{1,3,10\}&16\\ 3&\{2,7,8,11\}&22\\ 4,10&\{1,3,4,9\}&22\\ 5&\{4,9,10\}&14\\ 6,8&\{2,6,7\}&18\\ 7&\{2,8,11\}&16\\ 9&\{2,5,6,11\}&24\\ 11&\{3,4,9\}&16\\ 12&\{5,7,8,11\}&24 \end{array} $$
$$ \begin{array}{ccc} \multicolumn{3}{c}{q=17}\\ k&V_k&N_k\\ \hline 2,16&\{4,8,9,16\}&23\\ 3&\{8,9,15,16\}&24\\ 4&\{3,7,11,14\}&28\\ 5&\{4,8,13,16\}&24\\ 6,12&\{6,7,11,12\}&26\\ 7,11&\{3,5,10,12\}&28\\ 8&\{3,6,10,14\}&27\\ 9&\{1,8,15,16\}&24\\ 10&\{1,2,4,16\}&24\\ 13&\{5,6,7,12\}&28\\ 14&\{1,13,15,16\}&24\\ 15&\{7,10,12,14\}&28 \end{array} $$
For example, the minimum row $q=13$, $k=5$ has
$$ J_5=\{(0,5),(0,7),(3,6),(3,10),(3,11),(10,4),(12,5),(12,9)\}. $$
The first two pairs have weight one and the other six have weight two, giving $N_5=14$. Every other entry is checked by the same rule, using at most sixteen scalar pairs.

Each counted diagonal pair supplies a $C$-conjugacy orbit of matrices in $BB^{-1}$. If $c\neq-1$, then the off-diagonal product in \eqref{eq:off-diagonal-product} is nonzero. The orbit has $t$ projective elements unless both diagonal entries vanish; in that case it has $t/2$ elements. The latter case occurs only for $c=-1/2$ and $\ell=-1$, so at most one counted pair is affected. If $c=-1$, then both off-diagonal entries can vanish only when $A_r=D_r=0$. These equations imply $k^2=(k-1)^2$, hence $k=1/2$. Outside this last case, the same full-orbit count applies.

For $q=11$, the lists contain no $c=-1$. Since $N_k\geq13$, adding the identity gives
$$ |BB^{-1}|\geq6\cdot13-3+1=76. $$
For $q=13$, again $c=-1$ never occurs, and $-1\notin\mu$, so no half-size orbit occurs. Therefore
$$ |BB^{-1}|\geq7\cdot14+1=99. $$
For $q=17$ and $k\neq9$, one obtains $|BB^{-1}|\geq8\cdot23-4+1=181$. When $k=9$, the value $c=-1$ gives diagonal differences $(\ell,\ell^{-1})$. Their four projective pairs are $(1,1)$, $(2,9)$, $(4,13)$, and $(8,15)$, supplying four elements including the identity. Here $A_rD_r=1+c=0$, and $k^2=(k-1)^2$ makes either equality $A_r=0$ or $D_r=0$ force the other, so the off-diagonal entries do vanish. These pairs cannot arise from another value of $c$, because $PQ=1$ in \eqref{eq:diagonal-differences} forces $c=-1$. The other $N_9-4=20$ pairs each supply eight elements, except for the possible zero pair, which supplies four. Consequently,
$$ |BB^{-1}|\geq4+20\cdot8-4=160. $$

Generation follows from Lemma~\ref{lem:orbit-generation}. In the two nonsplit cases, the torus normalizer is the unique maximal overgroup of $C$. At $(q,t)=(17,8)$, the only proper overgroups are the two eigenline Borels and the torus normalizer, by Dickson's classification \cite{Giudici2007}. The disjoint eigenline pairs exclude both Borels. The dihedral normalizer has a unique cyclic subgroup of order eight; containing $xCx^{-1}$ would therefore force $x\in N_G(C)$, contrary to $\langle C,x\rangle=G$.

Finally, a zero diagonal entry is possible only in the split case. By Lemma~\ref{lem:zero-diagonal-pairs}, with $O_1=O_2=O_a(x)$, we then have $|B|=64$, $|BB^{-1}|\geq225$, and generation. This proves all rows.
\end{proof}

\subsection{Selecting and enlarging orbit subsequences}\label{sec:orbit-selection}

The same argument applies before and after a product set contains $C$. In the latter case, rounding to the possible cardinalities of the product set can give stronger bounds.

\begin{proposition}\label{prop:orbit-growth-criterion}
Assume \eqref{eq:high-mult-setup}. Let $r,k,b,d\geq1$ and $E,Q,s\geq0$ be integers. Suppose that at least $E$ terms of $R$ are eligible, lying in at most $Q$ double cosets of $C$. Choose one of the following settings:
\begin{itemize}
\item[(i)] No initial seed is used; let $s=0$ and $B_x=O_a(x)^r$.
\item[(ii)] An orbit product on at most $s$ eligible terms contains $C$; let $B_x=CO_a(x)^r$.
\end{itemize}
Suppose that, for every eligible $x$,
$$ |B_x|\geq b,\qquad |B_xB_x^{-1}|\geq d, \quad\text{and}\quad \langle B_xB_x^{-1}\rangle=G. $$
In case \textup{(i)}, let $\Omega=\{1,\ldots,n\}$. In case \textup{(ii)}, let $\Omega\subseteq\{1,\ldots,n\}$ contain the cardinality of every nonempty bi-$C$-invariant subset of $G$. Define
$$ \lceil u\rceil_\Omega=\min\{v\in\Omega:v\geq u\} $$
for $1\leq u\leq n$. Let $\delta=\lceil(d+1)/4\rceil$ and $v_1=\lceil b\rceil_\Omega$, and define
$$ v_{j+1}= \begin{cases} n,&\text{when }v_j+b>n,\\[3pt] \displaystyle \left\lceil \min\left\{n,\ v_j+ \min\left\{ \left\lceil\frac{\min(v_j,n-v_j)+1}{2}\right\rceil,\delta \right\}\right\} \right\rceil_\Omega, &\text{when }v_j+b\leq n. \end{cases} $$
If $v_k=n$ and
\begin{equation} \begin{gathered} E-s-(r-1)Q\geq rk,\\ (t-1)(s+rk)\leq m+1 \quad\text{and}\quad t(s+rk)\leq n, \end{gathered} \tag{13.13}\label{eq:orbit-resource-bounds} \end{equation}
then $1\in\Pi_n(S)$. When $r=1$, no bound on $Q$ is needed.
\end{proposition}

\begin{proof}
After the seed, group the remaining eligible terms into subsequences of $r$ terms within their double cosets. At most $(r-1)Q$ terms remain unused, so the first inequality supplies $k$ disjoint subsequences. If $x_i=a^{u_i}xa^{w_i}$, then $O_a(x_i)=a^{u_i+w_i}O_a(x)$. Orbit sets commute with $C$, so each selected subsequence gives a toral translate of $O_a(x)^r$, with the same difference bounds up to conjugacy.

In case (ii), restrict the initial product to $C$. Subsequent sets are bi-$C$-invariant and satisfy $AO_a(x)^r=A(CO_a(x)^r)$. Their cardinalities lie in $\Omega$. The first factor therefore gives at least $v_1$ products. By Lemma~\ref{lem:difference-growth}, the cardinalities satisfy the increment bound before rounding, while $|A|+|B_x|>n$ implies $AB_x=G$ directly. Rounding a lower bound up to the next member of $\Omega$ preserves the bound, since the actual cardinality belongs to $\Omega$. Both the update before rounding and the map $u\mapsto\lceil u\rceil_\Omega$ are nondecreasing. Thus larger actual sets preserve the recurrence.

Each orbit uses one remainder term and $t-1$ copies of $a$, by \eqref{eq:orbit-realization}. The last two inequalities make the whole construction a subsequence of $S$ of length at most $n$, with product set $G$. By Lemma~\ref{lem:high-mult-completion}(ii), the conclusion follows.
\end{proof}

In the ordinary case, rounding uses all integers. After a cyclic seed, one may always use the multiples $t,2t,\ldots,n$. By Lemma~\ref{lem:prime-double-cosets}, a smaller set of possible cardinalities, and hence a stronger rounding step, is sometimes available when $C$ has prime order.

\subsection{One growth estimate for ordinary and paired factors}\label{sec:weighted-orbit-growth}

\begin{lemma}\label{lem:weighted-growth}
Let $n=|G|$, let $\delta,b\geq1$ be integers, and suppose that factors of costs $r\in\{1,2\}$ satisfy
$$ |AB|-|A|\geq \min\left\{ \left\lceil\frac{\min(|A|,n-|A|)+1}{2}\right\rceil,r\delta \right\} $$
whenever $A\neq G$ is a preceding product set. Suppose that every factor of cost one has at least $b$ elements and every factor of cost two has at least $4\delta$ elements. If $4\delta-1<n/2$, then define
$$ A_b=\min\{j\geq0:(b+1)3^j\geq2\delta\,2^j\} \quad\text{and}\quad B_\delta=\lfloor\log_2(2\delta)\rfloor. $$
The total cost needed for the product to reach $G$ satisfies
\begin{equation} W<\frac{n+2}{\delta}+A_b+B_\delta+1. \tag{13.14}\label{eq:mixed-cost-growth} \end{equation}
If every factor has cost one, then the stronger bound is
\begin{equation} W<\frac{n+2}{\delta}+A_b+B_\delta-2. \tag{13.15}\label{eq:unit-cost-growth} \end{equation}
\end{lemma}

\begin{proof}
First consider factors of cost one, and let $T=2\delta-1$. The first factor has size at least $b$. Below $T$, the size plus one grows by a factor of at least $3/2$, so at most $A_b$ additional factors reach $T$. While $T\leq v<n-T$, every factor increases the size by at least $\delta$. Including the crossing factor, this middle stage takes fewer than $(n-2T)/\delta+1$ factors. Thereafter the complement has size at most $T$; its size plus one at least halves at each step, so at most $B_\delta$ further factors finish. The sum proves the stronger bound for factors of cost one.

For mixed costs, let $T=4\delta-1$. With only factors of cost one, the initial stage reaches $T$ in at most $A_b+3$ factors: reach $2\delta-1$ as above and then add two increments of $\delta$. A factor of cost two reaches $T$ immediately, because its own size is greater than $T$. Allowing such a factor as the last initial factor gives initial cost at most $A_b+4$.

In $[T,n-T)$, a factor of cost $r$ increases the size by at least $r\delta$, so the middle stage, including its crossing factor, costs less than $(n-2T)/\delta+2$. For a complement of size at most $T$, two factors of cost one reduce its size to at most $2\delta-1$, after which $B_\delta$ factors finish. A factor of cost two finishes immediately, since its size exceeds the complement. Allowing such a factor in the last position bounds the tail cost by $B_\delta+3$. Summing proves the mixed-cost bound. An overshoot only omits part of a stage and cannot increase these bounds.
\end{proof}

\subsection{Growth from two root groups}\label{sec:root-orbit-growth}

\begin{proposition}\label{prop:prime-root-growth}
If $G=\PSL_2(p)$ with $p\geq11$ prime and $a$ unipotent, then $1\in\Pi_n(S)$.
\end{proposition}

\begin{proof}
Choose coordinates with $a$ represented by $\begin{psmallmatrix}1&1\\0&1\end{psmallmatrix}$, and let $B$ be the upper triangular subgroup. We first show that, for $x\notin B$, the orbit $O=O_a(x)$ has $p$ members, at least $(p-1)^2$ differences, and a generating difference set.

Left and right multiplication by upper unipotents reduce $x$ to
$$ w_c=\begin{pmatrix}0&-c^{-1}\\c&0\end{pmatrix}, $$
where $c\neq0$. These operations translate the orbit and conjugate its difference set. Its nonidentity differences are represented by
$$ D(i,k)= \begin{pmatrix} 1-c^2ik&-k+c^2ik(i+k)\\ -c^2k&1+c^2k(i+k) \end{pmatrix}, $$
where $i\in\mathbb{F}_p$ and $k\in\mathbb{F}_p^*$. The lower-left and upper-left entries determine $k$ and $i$. Equality up to a minus sign forces $c^2k^2=-2$ and pairs $(i,k)$ with $(i+k,-k)$. At most $2p$ matrices take part in these identifications, so including the identity gives at least $1+p(p-1)-p=(p-1)^2$ images.

Let $K=\langle OO^{-1}\rangle$. It is normalized by $C$, and $KC$ contains $C$ and $xCx^{-1}$, the full root groups at two distinct points. These groups generate $G$: after moving the points to $0,\infty$, upper and lower elementary matrices generate $\SL_2(\mathbb{F}_p)$ by row operations. Thus $KC=G$, and simplicity implies $K=G$. The orbit has size $p$, since its size divides the prime $p$ and a singleton orbit would have trivial difference subgroup.

Let $\delta=(p-1)^2/4+1$, and let $A_p,B_\delta$ be the constants of Lemma~\ref{lem:weighted-growth} with $b=p$. Its estimate for factors of cost one applies, since $4\delta-1<n/2$, and
$$ \frac{n+2}{\delta} =2p+4-\frac{4p+12}{p^2-2p+5} <2p+4. $$
Thus at most $2p+A_p+B_\delta+1$ orbit factors suffice. For $p=11,13$, the pairs $(A_p,B_\delta)$ are $(4,5)$ and $(5,6)$, giving totals of at most $32$ and $38$, respectively. For $p\geq17$,
$$ A_p+B_\delta+1 <\log_{3/2}(p/2)+2\log_2p+1 \leq p. $$
The last inequality holds at $17$, and the derivative of its right side minus its left side is
$$ 1-\frac{1/\log(3/2)+2/\log2}{p}>0 $$
thereafter. Hence any ordered product of $3p$ such orbit factors is $G$.

There are at least $m-|B|+1=p(p-1)^2/4+1\geq3p$ terms outside $B$. Also $3p(p-1)\leq m$ and $3p^2\leq n$ for $p\geq11$. Identity \eqref{eq:orbit-realization} realizes all the factors using available copies of $a$, and Lemma~\ref{lem:high-mult-completion} completes the proof.
\end{proof}

\subsection{Semisimple orbit factors}\label{sec:semisimple-factors}

\begin{lemma}\label{lem:semisimple-factors}
Suppose that $a\in\PSL_2(q)$ is semisimple of order $t\geq9$, and that either $q$ is prime or $t>\sqrt{q}+1$. Let
$$ \delta=\left\lfloor\frac{(t-1)^2}{4}\right\rfloor+1 \quad\text{and}\quad K_q=\frac{2(q^2-1)}{\varepsilon}-6. $$
At most $K_q$ remainder terms lie in the union of the proper overgroups of $C$. The other terms can be grouped, leaving at most two unused, into factors of costs $r=1$ or $2$. A factor of cost $r$ uses $r$ remainder terms and $r(t-1)$ copies of $a$. It has at least $t$ products, and multiplication by this factor increases a product set of size $v$ by at least
\begin{equation} \min\left\{ \left\lceil\frac{\min(v,n-v)+1}{2}\right\rceil,r\delta \right\}. \tag{13.16}\label{eq:semisimple-increment} \end{equation}
\end{lemma}

\begin{proof}
By Dickson's classification, every proper overgroup of $C$ is contained in its torus normalizer or, in the split case, in one of the two eigenline Borel subgroups. Subfield subgroups have element orders at most $\sqrt{q}+1$, and exceptional subgroups have maximum element order five. The subgroup cap bounds the union in the split case by
$$ 2\left(\frac{q(q-1)}{\varepsilon}-2\right) +\left(\frac{2(q-1)}{\varepsilon}-2\right) =\frac{2(q^2-1)}{\varepsilon}-6=K_q. $$
In the nonsplit case, the single normalizer has order $2(q+1)/\varepsilon$, giving a still smaller bound.

By Lemma~\ref{lem:full-entry-orbits}, the orbit of a term outside this union with four nonzero entries has size $t$ and at least $(t-1)^2$ differences. Its torus and the conjugate of its torus by $x$ have disjoint eigenline pairs. They lie together in neither an eigenline Borel nor the torus normalizer; the other overgroups were excluded above. By Lemma~\ref{lem:orbit-generation}, the differences therefore generate $G$. By Lemma~\ref{lem:difference-growth}, the factor has cost one and constant $\lceil((t-1)^2+1)/4\rceil=\delta$.

Pair the terms with exactly one zero diagonal entry within each of the two types. Lemma~\ref{lem:zero-diagonal-pairs} applies, since $q$ is prime or $t>\sqrt{q}+1$. Each pair has size $t^2$, has a generating difference set, and has growth constant at least
$$ \left\lceil\frac{2+t^2(t-1)/2}{4}\right\rceil\geq2\delta. $$
Here $t\geq9$. The last inequality follows from $t^2(t-1)/8\geq(t-1)^2/2+2\geq2\delta$. Thus the pairs have cost two and satisfy the asserted growth bound. At most two terms remain unpaired. Also $t^2\geq4\delta$, so these factors satisfy the size hypothesis of Lemma~\ref{lem:weighted-growth} with $b=t$.
\end{proof}

\subsection{One resource bound for semisimple factors}\label{sec:semisimple-resources}

\begin{proposition}\label{prop:semisimple-completion}
Let $G=\PSL_2(q)$ with $q\geq11$, and suppose that $a$ is semisimple of order $t\geq9$. If $q$ is prime or $t>\sqrt{q}+1$, then $1\in\Pi_n(S)$.
\end{proposition}

\begin{proof}
Use $\delta,K_q$ from Lemma~\ref{lem:semisimple-factors}, and take $b=t$ in Lemma~\ref{lem:weighted-growth}; write $A_t=A_b$ and $B_t=B_\delta$. The hypotheses force $q\geq16$, and the order formula gives
$$ 4\delta-1\leq(t-1)^2+3\leq q^2+3 <\frac{q(q^2-1)}{4}\leq m. $$
When $t=9$, we have $\delta=17$, $A_t=4$, and $B_t=5$. At $q=17$, the torus is nonsplit, so all factors have cost one. By Lemma~\ref{lem:weighted-growth}, applied to factors of cost one, $W<2450/17+7<152$, hence $8W\leq1208<m=1224$. In every other case $n\geq3420$, and for $t=9$,
$$ W<\frac{n+2}{17}+10 \quad\text{and}\quad 8W<\frac{8n}{17}+81<\frac{n}{2}. $$

For $t\geq10$, we have $A_t,B_t\leq t-2$. Indeed, $2\delta/(t+1)<t/2<(3/2)^{t-2}$ and $2\delta<2^{t-1}$; both exponential comparisons hold at $10$ and strengthen thereafter. Consequently,
$$ W<\frac{4n}{(t-1)^2}+2t. $$
We verify $(t-1)W<m$ in two numerical cases. If $q\geq2t-1$, then $n\geq2t(t-1)(2t-1)$, and the assertion follows from $(2t-1)(t-9)\geq2(t-1)$ for $t\geq10$. Otherwise $q$ is even and $t=q-1$ or $q+1$, because a proper divisor of an odd integer $q\pm1$ is at most $(q+1)/3$. Here $t\geq15$, $n\geq t(t-1)(t-2)$, and the required inequality is $(t-2)(t-9)\geq4(t-1)$, which holds for $t\geq15$.

Thus in all cases $W<m/(t-1)\leq m/8$. There are enough eligible terms, even after pairing leaves at most two unused, since
$$ m-K_q-\frac{m}{8}-2 =\frac{(7q-32)(q^2-1)}{16\varepsilon}+4>0. $$
The factors use fewer than $m$ copies of $a$, and their total length is $tW<tm/(t-1)<n$. Identity \eqref{eq:orbit-realization} realizes the factors inside $S$, and Lemma~\ref{lem:high-mult-completion} completes the full product set to length $n$.
\end{proof}

\subsection{A single-orbit calculation}\label{sec:psl2-8-t9}

For $G=\PSL_2(8)$ and $t=9$, we have $n=504$ and $m=252$. The unique maximal overgroup of $C$ is its normalizer $N$, of order $18$. At least $252-16=236$ remainder terms lie outside $N$. For each such term, Lemmas~\ref{lem:orbit-generation} and \ref{lem:full-entry-orbits} imply that
$$ |O_a(x)|=9,\qquad |O_a(x)O_a(x)^{-1}|=73, \quad\text{and}\quad \langle O_a(x)O_a(x)^{-1}\rangle=G. $$
These bounds establish row $(8,9)$ of Table~\ref{tab:orbit-growth}; Proposition~\ref{prop:orbit-growth-criterion} completes this case.

\subsection{Higher orbit powers}\label{sec:higher-orbit-powers}

For the remaining small orders, higher orbit powers give the required growth constants. The order-three case below uses explicit matrix witnesses; the other three cases use the common scalar calculation in Lemma~\ref{lem:squared-orbit-bounds}. Selection, growth, and length checks are supplied by Proposition~\ref{prop:orbit-growth-criterion} and Table~\ref{tab:orbit-growth}.

\begin{lemma}\label{lem:psl2-8-cubic-orbit}
Let $G=\PSL_2(8)$, let $C=\langle a\rangle$ have order three, and suppose that $\langle a,x\rangle=G$. Let $O=\{a^ixa^{-i}:0\leq i<3\}$ and $B=O^3$. Then
$$ |B|\geq15,\qquad |BB^{-1}|\geq49, \quad\text{and}\quad \langle BB^{-1}\rangle=G. $$
\end{lemma}

\begin{proof}
We use a Frobenius-conjugate eigenbasis for $a$. Write
$$ F=\mathbb{F}_2[w]/(w^3+w+1) \quad\text{and}\quad E=F[\omega]/(\omega^2+\omega+1). $$
Thus $|F|=8$, $|E|=64$, and the involution of $E/F$ is $z\mapsto\bar{z}=z^8$. In this basis, the group consists of the matrices
$$ [u,v]:= \begin{pmatrix} u&v\\ \bar{v}&\bar{u} \end{pmatrix}, $$
where $u\bar{u}+v\bar{v}=1$, and we may take $C=\langle c\rangle$ with $c=\diag(\omega,\omega^2)$. Indeed, the two columns of $P=\begin{psmallmatrix}1&1\\\omega&\omega^2\end{psmallmatrix}$ are interchanged by the involution. Conjugating $\SL_2(F)$ by $P$ therefore gives exactly this matrix model.

The norm-one group $T=\{z\in E^*:z\bar{z}=1\}$ has order nine. Let $\tau=w^2+w\omega$. The defining equations give $\tau\bar{\tau}=1$ and $\tau^3=\omega$, so $T=\langle\tau\rangle$. If either entry $u$ or $v$ of $x=[u,v]$ were zero, then $x$ would normalize the diagonal torus, contrary to $\langle a,x\rangle=G$. Consequently,
$$ k=u\bar{u}\in F\setminus\{0,1\} \quad\text{and}\quad v\bar{v}=1+k. $$
Conjugation by $\diag(z,z^{-1})$, with $z\in T$, multiplies $v$ by $z^2$. Since squaring permutes $T$, we may arrange that $v=(1+k)^4\in F$. We can then write $u=k^4\tau^h$.

Replacing $x$ by $c^ixc^i$ leaves $v$ fixed and changes $h$ by $6i$. This replacement does not change $O^3$: every conjugation orbit commutes with $C$ as a set, and
$$ O(c^ixc^i)^3=(c^{2i}O(x))^3=O(x)^3. $$
We may thus reduce $h$ modulo three. Conjugation by $\begin{psmallmatrix}0&1\\1&0\end{psmallmatrix}$ replaces $h$ by $-h$, so $h$ may be taken to be zero or one. Finally, entrywise squaring replaces $(k,h)$ by $(k^2,2h)$. When $h=1$, conjugation by the same matrix and a further replacement $c^ixc^i$ restore $h=1$. The two squaring orbits in $F\setminus\{0,1\}$ have representatives $w$ and $1+w$. All cardinalities under consideration are therefore reduced to
\begin{equation} x=[k^4\tau^h,(1+k)^4], \tag{13.17}\label{eq:psl2-8-orbit-forms} \end{equation}
where $(k,h)\in\{(w,0),(w,1),(1+w,0),(1+w,1)\}$.

We give the short matrix calculations for these four forms. For $r=s_0+2s_1+4s_2\in\{0,\ldots,7\}$, write $\underline{r}=s_0+s_1w+s_2w^2\in F$, and extend this notation by $\underline{r+8s}=\underline{r}+\underline{s}\omega$. Integers in the following tables denote these field elements. Multiplication uses the identities
$$ (p+q\omega)(r+s\omega) =(pr+qs)+(ps+qr+qs)\omega \quad\text{and}\quad w^3=w+1, $$
together with
$$ [u,v][r,s]=[ur+v\bar{s},us+v\bar{r}] \quad\text{and}\quad [u,v]^{-1}=[\bar{u},v]. $$
Conjugation is $\overline{p+q\omega}=(p+q)+q\omega$. In particular, $\tau=\underline{20}$, and the pairs $(u,v)$ in the four forms are
$$ (\underline{6},\underline{7}),\qquad (\underline{61},\underline{7}),\qquad (\underline{7},\underline{6}), \quad\text{and}\quad (\underline{41},\underline{6}), $$
respectively. Let $X_i=[u,v\omega^i]$ for $i\in\{0,1,2\}$, and number nine triple words by
$$ Y_{3r+s}=X_0X_rX_s. $$
Here $0\leq r,s<3$.

For a matrix $Z=[U,V]$, use the invariant $\mathcal{I}(Z)=(U,V^3)$. Conjugation by $C$ preserves this pair. If $V\neq0$, then its $C$-conjugacy orbit has exactly three elements. Thus distinct displayed pairs with nonzero second coordinate give disjoint three-element orbits. The following five triple words in each column have distinct invariants:
$$ \begin{array}{cc|cccc} \text{word for }h=0&\text{word for }h=1 &(w,0)&(w,1)&(1+w,0)&(1+w,1)\\ \hline Y_0&Y_0&(6,2)&(46,6)&(7,7)&(27,3)\\ Y_3&Y_1&(7,63)&(62,27)&(2,32)&(19,40)\\ Y_5&Y_2&(15,3)&(60,36)&(42,2)&(18,32)\\ Y_6&Y_3&(7,56)&(58,2)&(2,36)&(17,45)\\ Y_7&Y_4&(14,3)&(62,24)&(47,2)&(19,45) \end{array} $$
Since $B$ is invariant under conjugation by $C$, this proves that $|B|\geq15$. For the difference set, define
$$ Z_{pqh}=Y_p(c^hY_qc^{-h})^{-1}\in BB^{-1}. $$
For example, in the first form $X_0^2=1$, so $Y_0=X_0$ and
$$ Z_{001}=X_0X_2=[w+(1+w)\omega,w^2\omega]. $$
Its invariant is $(\underline{26},\underline{5})$, the first entry below. The difference set is also closed under inversion, and $[U,V]^{-1}=[\bar{U},V]$. Thus a row with $U\notin F$ supplies two conjugacy orbits, with invariants $(U,V^3)$ and $(\bar{U},V^3)$. The same multiplication rules give the following eight pairs:
$$ \begin{array}{c|cccc} (p,q,h)&(w,0)&(w,1)&(1+w,0)&(1+w,1)\\ \hline (0,0,1)&(26,5)&(62,27)&(19,5)&(37,7)\\ (0,3,1)&(47,2)&(43,6)&(39,4)&(28,4)\\ (0,3,2)&(49,24)&(29,48)&(49,27)&(37,63)\\ (0,5,0)&(31,4)&(16,4)&(23,2)&(27,3)\\ (0,6,1)&(55,27)&(27,27)&(55,24)&(36,54)\\ (3,3,1)&(19,45)&(26,45)&(44,56)&(55,24)\\ (3,5,0)&(29,54)&(25,40)&(18,36)&(17,45)\\ (3,7,0)&(54,16)&(58,16)&(48,16)&(60,32) \end{array} $$
All second coordinates are nonzero, every first coordinate lies outside $F$, and in each column the eight pairs remain distinct when the first coordinate is replaced by its conjugate. They therefore give sixteen disjoint orbits. Together with the identity, these yield $|BB^{-1}|\geq16\cdot3+1=49$.

Finally, $K=\langle BB^{-1}\rangle$ is normalized by $C$. By Dickson's classification \cite{Giudici2007}, every proper overgroup of $C$ is contained in its nonsplit torus normalizer, of order eighteen. Indeed, the other maximal subgroup orders are $56$ and $14$, and the subfield groups of order six containing $C$ normalize $C$. Thus $KC=G$, since $|K|\geq49$. As $K\trianglelefteq KC$ and $G$ is simple, $K=G$.
\end{proof}

\begin{proposition}\label{prop:psl2-8-t3}
In the high-multiplicity setting, if $G=\PSL_2(8)$ and $\ord(a)=3$, then $1\in\Pi_{504}(S)$.
\end{proposition}

\begin{proof}
Here $m=252$, $M\geq253$, and $|R|\geq252$. The unique maximal overgroup of $C=\langle a\rangle$ is its normalizer $N$, of order eighteen. By Lemma~\ref{lem:subgroup-cap}, $|R_N|\leq16$, so at least $236$ terms of $R$ lie outside $N$. For $x\notin N$, the groups $C$ and $xCx^{-1}$ intersect trivially; therefore $|CxC|=9$. There are exactly $(504-18)/9=54$ double cosets outside $N$. The local bounds in Lemma~\ref{lem:psl2-8-cubic-orbit} give $(r,b,d)=(3,15,49)$. Thus the third row of Table~\ref{tab:orbit-growth} applies, and Proposition~\ref{prop:orbit-growth-criterion} completes the proof.
\end{proof}

\begin{proposition}\label{prop:small-squared-orbits}
Let $G=\PSL_2(q)$. In the high-multiplicity setting, the conclusion $1\in\Pi_n(S)$ holds for $(q,t)=(11,6),(13,7),(17,8)$.
\end{proposition}

\begin{proof}
Call a remainder term eligible when it generates $G$ together with $C$. For the two nonsplit rows, the unique maximal overgroup is the normalizer, of orders twelve and fourteen, respectively. By Lemma~\ref{lem:subgroup-cap}, there are at least $E=320$ and $E=534$ eligible terms. For the split row, at least $m-t+2=1218$ terms lie outside $C$. The two Borels of order $136$ and the normalizer of order sixteen contribute at most $2(136-2)+(16-2)=282$ terms, leaving $E=936$. By Lemma~\ref{lem:cyclic-amplification}, every eligible double coset has size $t^2$. Thus the number of eligible double cosets is at most $Q=\lfloor n/t^2\rfloor$.

Together with Lemma~\ref{lem:squared-orbit-bounds}, these bounds establish the rows $(11,6)$, $(13,7)$, and $(17,8)$ of Table~\ref{tab:orbit-growth}. By Proposition~\ref{prop:orbit-growth-criterion}, all three cases follow.
\end{proof}

\section{Cyclic seeds}\label{sec:cyclic-seeds}

The applications below assume \eqref{eq:high-mult-setup} and proceed under the assumption that $1\notin\Pi_n(S)$. Let $C=\langle a\rangle$, let $t=|C|$, and define $O_a(x)=\{a^ixa^{-i}:0\leq i<t\}$. Recall that an orbit factor uses one term $x$ and $t-1$ copies of $a$, by \eqref{eq:orbit-realization}. Orbit products commute with $C$ as sets, and replacing a term by another member of its $C$-double coset translates its orbit by an element of $C$.

\begin{lemma}\label{lem:power-relation-seeds}
Let $r,l,d$ be positive integers. If $(xa^i)^r=1$, then
\begin{equation} a^{-ir}\in O_a(x)^r. \tag{14.1}\label{eq:power-relation-orbit} \end{equation}
Consequently, if $(xa^i)^l=(xa^j)^l=1$ and $t\nmid l(j-i)$, then $O_a(x)^l\cap C$ contains at least two elements. If $t$ is odd and $x^d$ is an involution inverting $a$, then $C\subseteq O_a(x)^{2d}$.
\end{lemma}

\begin{proof}
Expanding the power relation gives
$$ 1=(xa^i)^r =\left(\prod_{q=0}^{r-1}a^{qi}xa^{-qi}\right)a^{ir}, $$
which proves the first assertion and its consequence.

For the last assertion, let $z=x^d$. Then $z^2=1$ and
$$ z(a^jza^{-j})=a^{-2j}\in O_a(x)^{2d}. $$
As $t$ is odd, these products run through $C$.
\end{proof}

\begin{lemma}\label{lem:seed-selection}
Suppose that $N$ available terms lie in $Q$ specified $C$-double cosets. For the $i$th double coset, let $r_i$ be a positive integer and suppose that the orbit product of any $r_i$ of its terms contains a subset $T\subseteq C$ such that
\begin{equation} |AT|\geq\min\{t,|A|+w_i\}. \tag{14.2}\label{eq:seed-weight-growth} \end{equation}
This condition is required for every nonempty $A\subseteq C$, where $1\leq w_i\leq t-1$. Define
$$ E_0=\sum_{i=1}^{Q}(r_i-1) \quad\text{and}\quad \lambda=\max_i\frac{r_i}{w_i}. $$
If $N>E_0+\lambda(t-2)$, then disjoint subsequences can be selected whose orbit product contains $C$. The collection can be chosen inclusion-minimal subject to having total weight at least $t-1$.

If $t$ is prime, then the growth condition holds with $w=|T|-1$. If $T$ contains $\{c,cd,\ldots,cd^w\}$ for some generator $d$ of $C$, then the condition also holds with weight $w$. In particular, consecutive pairs in $C_4$ have weight one and three-element subsets have weight two.
\end{lemma}

\begin{proof}
Partition the available terms in the $i$th double coset into subsequences of $r_i$ terms, leaving at most $r_i-1$ terms unused. If the resulting subsequences had total weight at most $t-2$, then their total length would be at most $\lambda(t-2)$. Including the unused terms would give $N\leq E_0+\lambda(t-2)$, a contradiction. Choose an inclusion-minimal collection of subsequences whose total weight is at least $t-1$. Starting with $A=\{1\}$ and applying the growth condition successively proves that the product of the selected cyclic subsets is $C$.

The prime-order assertion is the Cauchy--Davenport inequality. For the progression assertion, a nonempty proper subset of $C$ cannot be invariant under multiplication by a generator $d$. Thus adjoining one more translate by $d$ increases its size by at least one until it reaches $t$. Iterating $w$ times proves the growth condition. Every three-element subset of $C_4$ is a translate of $\{1,a,a^2\}$, which proves the last assertion.
\end{proof}

\begin{lemma}\label{lem:prime-double-cosets}
Let $C$ be a Sylow subgroup of prime order $p$ in $G$, and let $h=|N_G(C):C|<p$. Then $q=(|G|-ph)/p^2$ is an integer, and every nonempty bi-$C$-invariant subset has cardinality in
\begin{equation} \Omega_p=\{p^2u+pj:0\leq u\leq q,\ 0\leq j\leq h\}\setminus\{0\}. \tag{14.3}\label{eq:prime-cardinalities} \end{equation}
If $\langle C,xCx^{-1}\rangle=G\neq C$, then $B=CxC$ satisfies
$$ |B|=p^2,\qquad |BB^{-1}|\geq p^2+p, \quad\text{and}\quad \langle BB^{-1}\rangle=G. $$
\end{lemma}

\begin{proof}
The double cosets inside $N_G(C)$ are its $h$ cosets of size $p$; every other double coset has size $p^2$, since two distinct subgroups of order $p$ intersect trivially. This proves the assertion about $q$ and the stated set of possible cardinalities.

For the last assertion, $x\notin N_G(C)$, so $|B|=p^2$ and $BB^{-1}=CxCx^{-1}C$ generates $G$. If $1\neq z\in xCx^{-1}$, then $z\notin N_G(C)$: the normal Sylow $p$-subgroup $C$ is the unique Sylow $p$-subgroup of $N_G(C)$, whereas $C\cap xCx^{-1}=1$. The disjoint subsets $C$ and $CzC$ of $BB^{-1}$ have sizes $p$ and $p^2$, respectively, which proves the bound.
\end{proof}

\subsection{A full product from involutions}\label{sec:a5-involution-product}

For $G=A_5$ and $t=5$, we have $m=30$ and $M\geq31$. Take $a=(12345)$. The unique maximal overgroup of $C$ is $N=N_{A_5}(C)\cong D_{10}$, so at least $30-8=22$ terms of $R$ lie outside $N$. Every double coset there has size $25$. Representatives of the two double cosets are
$$ x=(12)(34) \quad\text{and}\quad x'=(13)(25). $$
To check this without enumeration, label the points cyclically by $\mathbb{Z}/5\mathbb{Z}$. The multiset $\Delta(y)=\{y(i+1)-y(i):i\in\mathbb{Z}/5\mathbb{Z}\}$ is unchanged by left or right multiplication by $C$, whereas
$$ \Delta(x)=\{2,2,3,4,4\} \quad\text{and}\quad \Delta(x')=\{1,1,2,3,3\}. $$
The two disjoint double cosets account for all $50$ elements outside $N$. Moreover, $b=(2354)$ satisfies $bab^{-1}=a^2$ and $bxb^{-1}=x'$.

Let $X_j=a^jxa^{-j}$ and $O=\{X_0,\ldots,X_4\}$. The three products
$$ X_0^3=(12)(34),\qquad X_0X_1X_0=(14)(35), \quad\text{and}\quad X_0X_2X_0=(25)(34) $$
lie in distinct $C$-conjugacy orbits. Indeed, move the unique fixed point of each involution to $5$ by a power of $a$; the resulting matchings of the other four points are distinct. Each orbit has size five, so $O^3$ contains all fifteen involutions of $A_5$.

Every element of $A_5$ is a product of two involutions. Besides the identity, this follows for the three possible cycle types from
$$ \begin{aligned} (12)(34)&=(13)(24)\,(14)(23),\\ (123)&=(12)(45)\,(23)(45),\\ (12345)&=(12)(35)\,(25)(34). \end{aligned} $$
Conjugation by $S_5$ covers every element of each type, including both $A_5$-classes of five-cycles. Thus $O^6=A_5$. Conjugation by $b$ gives the same conclusion for the second double coset.

Of the $22$ available remainder terms, at least eleven lie in one double coset; select six of them. Identity \eqref{eq:orbit-realization} gives a full product set using these six terms and $24\leq M$ copies of $a$. The total length is $30<60$, and Lemma~\ref{lem:high-mult-completion} completes this case.

\subsection{Seeds from prime cyclic sumsets}\label{sec:psl2-8-t7-seed}

Let $G=\PSL_2(8)$ and $t=7$. Then $C$ is the split torus. By Dickson's subgroup classification \cite{Giudici2007}, its maximal overgroups are the two eigenline Borels $B_1,B_2$, of order $56$, and its normalizer $N$, of order $14$. Their union contributes at most $54+54+12=120$ terms of $R$, leaving at least $132$ terms $x$ with $\langle C,x\rangle=G$. All pairwise intersections are $C$, so the union has $56+56+14-3\cdot7+7=112$ elements. Outside this union, the double cosets have size $49$; there are $(504-112)/49=8$ of them. In particular, one such double coset supplies twelve remainder terms.

\smallskip
\noindent\emph{A product containing $C$.}
Work in $\SL_2(\mathbb{F}_8)=\PSL_2(8)$ with $a=\diag(\lambda,\lambda^{-1})$. Since squaring permutes $\mathbb{F}_8^*$, toral left and right translations and diagonal conjugation give the following orbit normal forms. These operations preserve the property that a power contains $C$. Every generating matrix has both off-diagonal entries nonzero.

\smallskip
\noindent\emph{Case 1.} Suppose that all four entries are nonzero. Let $k=\alpha\delta\notin\{0,1\}$, choose $s^2=k$, and let $c=k+1$. The normalized orbit consists of
$$ X_u= \begin{pmatrix} s&u\\ c/u&s \end{pmatrix}, $$
where $u\in\mathbb{F}_8^*$. Its elements are involutions. For $r,v\neq0$, define
$$ B_r(v)=X_{rv}X_v= \begin{pmatrix} k+cr&sv(r+1)\\ sc(1+r^{-1})/v&k+c/r \end{pmatrix}. $$
Choose $r\neq1$ with $A=k+cr\neq0$ and $D=k+c/r\neq0$; at most two of the six choices are excluded. With $h=A/D$ and $w=vr/h$, direct multiplication gives
$$ B_r(v)=\diag(h,h^{-1})B_{r^{-1}}(w). $$
Here $h\neq1$, since $r\neq r^{-1}$. Thus $z=\diag(h,h^{-1})\in O^4$. Reversing the involution word gives $z^{-1}\in O^4$, and $1=X_1^4\in O^4$. Since $z$ generates $C$,
$$ C=\{z^j:-3\leq j\leq3\}\subseteq(O^4)^3=O^{12}. $$

\smallskip
\noindent\emph{Case 2.} Suppose that exactly one diagonal entry is zero. Interchange the eigenlines when necessary. The normal form and its two-factor products are
$$ X_u= \begin{pmatrix}0&u\\u^{-1}&1\end{pmatrix} \quad\text{and}\quad B_r(u)=X_uX_{u/r}= \begin{pmatrix}r&u\\r/u&1+r^{-1}\end{pmatrix}. $$
For $r\neq0,1$, let $h(r)=r^2/(r+1)$. Then
$$ B_r(u)B_r\bigl(u/h(r)\bigr) =\diag\bigl(h(r),h(r)^{-1}\bigr). $$
The six parameters $r$ yield exactly three different values $h(r)$: the equation $h(r)=z$ is $r^2+zr+z=0$, and its other root is $r/(r+1)$, different from $0,1,r$. Hence $O^4$ contains a three-element subset $D$ of $C_7$. By the Cauchy--Davenport inequality, $D^3=C$, so once again $C\subseteq O^{12}$. Twelve actual terms in the selected double coset give a toral translate of $O^{12}$, and thus a product set containing $C$.

\smallskip
\noindent\emph{Completion by growth.}
Restrict this product set to $C$. Subsequent products $A$ are bi-$C$-invariant and satisfy $AO=A(CO)$. For an orbit with four nonzero entries, let $B=CO=CxC$, so that $|B|=49$. By Lemma~\ref{lem:cyclic-amplification},
\begin{equation} |BB^{-1}|\geq7+\left\lfloor\frac{7-1}{2}\right\rfloor7^2=154. \tag{14.4}\label{eq:psl2-8-amplification} \end{equation}
The difference set contains $C$ and $xCx^{-1}$. These groups generate $G$: the eigenline pairs are disjoint for a matrix with four nonzero entries, excluding both Borels containing $C$, and $N$ has a unique subgroup of order seven. The parameter in Lemma~\ref{lem:difference-growth} is $\lceil155/4\rceil=39$. All sizes are multiples of seven, so thirteen further factors with four nonzero entries suffice, with successive lower bounds
$$ 7,\ 49,\ 77,\ 119,\ 161,\ 203,\ 245,\ 287,\ 329,\ 371,\ 413,\ 455,\ 483,\ 504. $$
The last step uses $483+49>504$.

Pair terms having a zero diagonal entry in the same position. By Lemma~\ref{lem:zero-diagonal-pairs}, each pair gives an amplified factor of size $7^3=343$. This applies to any two such terms; their double cosets need not be equal.

After the initial twelve terms, at least $120$ generating terms remain. If thirteen have four nonzero entries, then use the preceding recurrence. Otherwise at least $108$ terms have a zero diagonal entry, and one of the two types supplies four terms. These terms give two amplified paired factors of size $343$, whose product is $G$ because $343+343>504$. Thus at most $25$ remainder terms suffice in either case. They use $150\leq253\leq M$ copies of $a$ and $175<504$ total terms. By Lemma~\ref{lem:high-mult-completion}, the conclusion follows.

\subsection{A seed construction inside a covering subgroup}\label{sec:psl2-7-t7-seed}

Let $G=\PSL_2(7)$ and $t=7$. Here $n=168$, $m=84$, and $M\geq85$. The unique maximal overgroup of $C$ is $N=N_G(C)$, of order $21$. There are three double cosets outside $N$, each of size $49$. The subgroup bound leaves at least $84-19=65$ terms outside $N$; in particular, one double coset supplies fourteen terms. We prove that the fourteenth power of its orbit is $G$.

Work modulo signs with determinant-one matrices over $\mathbb{F}_7$. Write
$$ u(j)=\begin{pmatrix}1&j\\0&1\end{pmatrix} \quad\text{and}\quad w=\begin{pmatrix}0&-1\\1&0\end{pmatrix}. $$
Left and right multiplication by $C=\{u(j)\}$ reduce every matrix outside $N$ to the form $w_c=\begin{psmallmatrix}0&-1/c\\c&0\end{psmallmatrix}$. Conjugation by $\diag(c,1)$ preserves $C$ and sends $w_c$ to $w$. Thus it suffices to consider
$$ O=\{X_j:j\in\mathbb{F}_7\} \quad\text{and}\quad X_j=u(j)wu(-j)= \begin{pmatrix}j&-j^2-1\\1&-j\end{pmatrix}. $$
These matrices represent projective involutions. For $h\neq0$,
\begin{equation} X_jX_{j+h}= \begin{pmatrix} jh-1&h(1-j^2-jh)\\ h&-h(j+h)-1 \end{pmatrix}. \tag{14.5}\label{eq:psl2-7-orbit-pair} \end{equation}
For a fixed $h$, the seven products form a $C$-conjugacy orbit, with trace and lower-left entry $(-h^2-2,h)$. These pairs, understood up to simultaneous sign, distinguish the six orbits: comparing $h$ and $-h$ would require $h^2+2=0$, which is impossible in $\mathbb{F}_7$. Therefore $|O^2|=1+6\cdot7=43$.

Let $K=\langle O^2\rangle$. It is normalized by $C$ and contains $X_0X_1$, whose order is four: its trace is $4$, and its square has trace $4^2-2=0$. Hence $KC$ is not contained in the order-$21$ normalizer. The unique maximal overgroup property implies that $KC=G$; then $K\trianglelefteq G$, and simplicity implies $K=G$.

Since every element of $O$ is an involution, $OO^{-1}=O^2$. Lemma~\ref{lem:difference-growth} first applies to $B=O$, with growth parameter $\lceil44/4\rceil=11$. Starting from $|O^2|=43$, two steps give
$$ |O^3|\geq54 \quad\text{and}\quad |O^4|\geq65. $$
Apply Lemma~\ref{lem:difference-growth} to $B=O^2$. Now $BB^{-1}=O^4$ generates $G$, since $1\in O^2$ implies that $O^2\subseteq O^4$. Its growth parameter is at least $\lceil66/4\rceil=17$. Successive powers have lower bounds
$$ \begin{gathered} |B|\geq43,\qquad |B^2|\geq60,\qquad |B^3|\geq77,\\ |B^4|\geq94,\qquad |B^5|\geq111, \quad\text{and}\quad |B^6|\geq128. \end{gathered} $$
The variable part of the growth bound is at least the stated constant at each preceding value, and the update is nondecreasing. Since $128+43>168$, we have $B^7=O^{14}=G$. Fourteen terms in the chosen double coset use $84\leq M$ copies of $a$ and $98<168$ total terms. By Lemma~\ref{lem:high-mult-completion}, the conclusion follows.

\subsection{Seeds from a trace parametrization}\label{sec:singer-trace-seed}

Let $G=\PSL_3(3)$ and $t=13$. We first prove the finite-field lemma used in this case.

\begin{lemma}\label{lem:singer-trace}
Let $F=\mathbb{F}_{27}$, choose $\zeta$ satisfying $\zeta^3-\zeta-1=0$, and let $c$ be multiplication by $\zeta$ on the three-dimensional $\mathbb{F}_3$-space $F$. Let $C=\langle c\rangle$. If $x\in\SL_3(3)\setminus C$, then $xC$ contains at least three elements with reducible characteristic polynomial. Consequently, there exist $y\in xC$, $d\in\{6,8\}$, and $j\not\equiv0\pmod{13}$ such that
$$ \{1,c^j\}\subseteq O_c(y)^d. $$
\end{lemma}

\begin{proof}
The polynomial has no root in $\mathbb{F}_3$, hence is irreducible. Its root has norm $1$, is not $1$, and consequently has order $13$. Write the linearized polynomial of $x$ as $x(z)=Az+Bz^3+Dz^9$, and let $A'$ be the coefficient of $z$ in $x^{-1}$. The Dickson matrix, representing the same map after scalar extension, is
$$ \begin{pmatrix} A&B&D\\ D^3&A^3&B^3\\ B^9&D^9&A^9 \end{pmatrix}. $$
Its inverse minor and determinant give
\begin{equation} \begin{gathered} A'=A^{12}-B^3D^9,\\ 1=N(A)+N(B)+N(D)-\Tr(AB^3D^9). \end{gathered} \tag{14.6}\label{eq:dickson-identities} \end{equation}
Here $N$ and $\Tr$ are the norm and trace from $F$ to $\mathbb{F}_3$. If $s_i=\tr(xc^i)$ and $t_i=\tr((xc^i)^{-1})$, then
$$ s_i=\Tr(A\zeta^i) \quad\text{and}\quad t_i=\Tr(A'\zeta^{-i}). $$
The characteristic polynomial is $T^3-s_iT^2+t_iT-1$. Checking its possible roots $1,-1$ shows that it is irreducible if and only if
\begin{equation} (s_i,t_i)\in\{(0,2),(1,2),(2,0),(2,1)\}. \tag{14.7}\label{eq:cubic-irreducibility} \end{equation}
The required traces are
$$ \begin{array}{c|l} \Tr(\zeta^i)&i\pmod{13}\\ \hline 0&0,1,3,9\\ 1&7,8,11\\ 2&2,4,5,6,10,12 \end{array} $$
They follow from $T_0=T_1=0$, $T_2=2$, and $T_i=T_{i-2}+T_{i-3}$. Let $D_2=\{2,4,5,6,10,12\}$. Then $D_2+D_2=(\mathbb{Z}/13\mathbb{Z})\setminus\{0\}$: witnesses for residues $1,\ldots,12$ are, respectively,
$$ \begin{gathered} 2+12,\ 5+10,\ 4+12,\ 2+2,\ 6+12,\ 2+4,\\ 2+5,\ 2+6,\ 4+5,\ 4+6,\ 5+6,\ 2+10. \end{gathered} $$
Zero is absent because $D_2\cap(-D_2)=\varnothing$.

For a coefficient of norm $1$, the set of indices at which its trace is $2$ has six elements; for norm $-1$ it has three elements; for coefficient zero it is empty. The irreducibility condition can thus hold at most nine times unless $N(A)=N(A')=1$. In that exceptional case, write $A=\zeta^\alpha$ and $A'=\zeta^\beta$. The two sets of indices at which the traces are $2$ are $D_2-\alpha$ and $\beta-D_2$. They intersect unless $AA'=1$. If they intersect, then membership in exactly one of these sets occurs at most $6+6-2=10$ times, bounding the irreducible polynomials by ten.

Finally, $AA'=1$ and $N(A)=1$ in the inverse-minor identity imply that $B^3D^9=0$. The determinant identity then forces both $B$ and $D$ to vanish. This makes $x$ multiplication by $A\in\langle\zeta\rangle$, contrary to the hypothesis. Thus at least three polynomials are reducible.

A matrix with reducible characteristic polynomial here has order dividing $8$ or $6$. An irreducible quadratic factor gives a semisimple component over $\mathbb{F}_9$ and order dividing $8$; if all eigenvalues lie in $\mathbb{F}_3$, then the semisimple and unipotent parts have orders dividing $2$ and $3$, respectively. Of the three elements, two therefore have orders dividing a common $d\in\{6,8\}$. Write them as $y$ and $yc^j$, with $j\not\equiv0\pmod{13}$. Since $13\nmid dj$, Lemma~\ref{lem:power-relation-seeds} implies that $\{1,c^{-dj}\}\subseteq O_c(y)^d$, which proves the assertion.
\end{proof}

Now $G=\PSL_3(3)=\SL_3(3)$, $n=5616$, $m=2808$, and $M\geq2809$. The unique maximal overgroup of $C$ is its normalizer, of order $39$. There are $33$ double cosets of size $169$ outside the normalizer, and three double cosets of size $13$ inside it \cite{Wilson2009}. There are at least $2808-37=2771$ remainder terms outside the normalizer; the weaker bound $2760$ will suffice.

By Lemma~\ref{lem:singer-trace} and double-coset translation, each of the thirty-three generating double cosets supplies a subsequence of length at most eight whose orbit product contains two elements of $C$. Apply Lemma~\ref{lem:seed-selection} with weight one. The allowance for unused terms is at most $33(8-1)=231$, and $\lambda\leq8$. Since
$$ 2760>231+8(13-2)=319, $$
a cyclic seed uses at most $12\cdot8=96$ remainder terms.

The unique maximal overgroup property verifies the hypothesis of Lemma~\ref{lem:orbit-generation} for every generating term. By Lemma~\ref{lem:prime-double-cosets}, we may use $h=3$, $\Omega_{13}$, and the bounds $b=169$ and $d=182$. Together with $E=2760$ and $s=96$, these parameters establish the $\PSL_3(3)$ row of Table~\ref{tab:orbit-growth}. By Proposition~\ref{prop:orbit-growth-criterion}, the Singer case follows.

\subsection{Seeds from short permutation identities}\label{sec:a6-permutation-seed}

\begin{proposition}\label{prop:a6-order-five}
The high-multiplicity case $G=A_6$, $t=5$ satisfies $1\in\Pi_{360}(S)$.
\end{proposition}

\begin{proof}
Write $S=a^{[M]}R$, where $a=(12345)$, $C=\langle a\rangle$, $M\geq181$, and $|R|\geq180$, and suppose that $1\notin\Pi_{360}(S)$.

The maximal subgroups of $A_6$ whose orders are divisible by five are the two conjugacy classes of $A_5$ subgroups \cite{WilsonEtAlND}. Each class has six members and each member has six Sylow five-subgroups. Since $N_{A_6}(C)\cong D_{10}$, there are $36$ Sylow five-subgroups in $A_6$. Counting incidences shows that precisely one member of each class contains $C$. Denote these members by $L_1,L_2$, taking $L_1$ to be the stabilizer of $6$. Their normalizers of $C$ both equal $N=N_{A_6}(C)$.

By Lemma~\ref{lem:subgroup-cover-caps}(iv), $|R_{L_i\setminus C}|\leq11$. By Lemma~\ref{lem:subgroup-cap}, at least $180-5+2=177$ terms lie outside $C$, so at least $177-22=155$ terms lie outside $L_1\cup L_2$. Call these terms generating. For each such $x$, we have $\langle C,x\rangle=A_6$, and $CxC$ has size $25$. There are at most $\lfloor360/25\rfloor=14$ generating double cosets.

\smallskip
\noindent\emph{Claim 1.}
For every $x\notin L_1$, there is an integer $2\leq l\leq4$ such that $O_a(x)^l$ contains at least two elements of $C$.

Each double coset outside $L_1$ has a unique representative
$$ x=(56)\sigma, $$
where $\sigma\in S_{\{1,2,3,4\}}\setminus A_{\{1,2,3,4\}}$. Indeed, first multiply on the right by the unique power of $a$ which makes $x(5)=6$, and then on the left by the unique power which makes $x(6)=5$. These two conditions are unaffected by the other multiplication. The resulting permutation interchanges $5$ and $6$, and its restriction to the remaining four points must be odd.

Let $b=(1243)$, fixing $5$ and $6$. It normalizes $C$, since $bab^{-1}=a^2$. Conjugation by $b$ has two orbits on the six transpositions of $\{1,2,3,4\}$, represented by $(12)$ and $(14)$. On the six four-cycles it has one orbit of size four represented by $(1234)$, and the two singleton orbits $b,b^{-1}$. For the first three representatives, the following parameters give elementary permutation identities:
$$ \begin{array}{c|ccc} \sigma&i&j&l\\ \hline (12)&2&3&4\\ (14)&1&4&3\\ (1234)&0&1&4 \end{array} $$
For each row, $(xa^i)^l=(xa^j)^l=1$. These identities follow by multiplying the displayed cycles; for example, when $x=(12)(56)$, both $xa^2$ and $xa^3$ have cycle type $(4,2)$. Since $5$ divides neither $l$ nor $j-i$, Lemma~\ref{lem:power-relation-seeds} implies that the indicated orbit power contains two elements of $C$. Conjugation by $b$ preserves this property.

In the remaining cases $\sigma=b$ or $b^{-1}$, we have $x^2=(14)(23)$, which inverts $a$. By the second assertion of Lemma~\ref{lem:power-relation-seeds}, with $d=2$, we therefore have $C\subseteq O_a(x)^4$.

Finally, left and right multiplication by $C$ merely translates an orbit power by an element of $C$. Thus the claim holds for every member of each double coset, proving Claim 1.

Each of the at most fourteen generating double cosets therefore supplies a subsequence of length at most four and weight one in Lemma~\ref{lem:seed-selection}. The parameters satisfy
$$ E_0\leq14(4-1)=42 \quad\text{and}\quad \lambda\leq4. $$
Since $155>42+4(5-2)=54$, a cyclic seed uses at most $4\cdot4=16$ terms.

The incidence argument verifies the hypothesis of Lemma~\ref{lem:orbit-generation} for every generating term. By Lemma~\ref{lem:prime-double-cosets}, we may use $h=2$, $\Omega_5$, and the bounds $b=25$ and $d=30$. Together with $E=155$ and $s=16$, these parameters establish the $A_6$ row of Table~\ref{tab:orbit-growth}. By Proposition~\ref{prop:orbit-growth-criterion}, the conclusion follows.
\end{proof}

\subsection{Seeds from coprime element orders}\label{sec:a7-order-seed}

\begin{proposition}\label{prop:a7-order-seven}
The high-multiplicity case $G=A_7$, $t=7$ satisfies $1\in\Pi_{2520}(S)$.
\end{proposition}

\begin{proof}
Let $C=\langle a\rangle$, where $a$ is a seven-cycle, and write $S=a^{[M]}R$ as in \eqref{eq:high-mult-setup}. Thus $|S|=3779$, $M\geq1261$, and $|R|\geq1260$. Suppose, for a contradiction, that $1\notin\Pi_{2520}(S)$.

We use the standard maximal-subgroup description of $A_7$: the maximal subgroups with order divisible by seven form two conjugacy classes, each consisting of subgroups isomorphic to $\PSL_3(2)\cong\PSL_2(7)$; see the $A_7$ entry of the ATLAS \cite{WilsonEtAlND}. The other maximal-subgroup types are $A_6$, $S_5$, and $(A_4\times C_3):C_2$, whose orders are not divisible by seven. The two classes each have $2520/168=15$ members. Each member has eight Sylow seven-subgroups, while $A_7$ has $2520/21=120$ such subgroups. Counting pairs shows that $C$ is contained in exactly one subgroup from each class, say $L_1,L_2$. Here $N=N_{A_7}(C)\cong C_7\rtimes C_3$ has order $21$. Indeed, the normalizer of a regular seven-cycle in $S_7$ is the affine group of order $42$, and its intersection with $A_7$ has index two. In each $L_i$, the normalizer of $C$ has order $21$, so it equals $N$.

\smallskip
\noindent\emph{Claim 1.}
Each $L_i\setminus C$ contains at most $43$ terms of $R$.

Every element of $N\setminus C$ has a $C$-conjugation orbit equal to its full coset of $C$: in the affine model this follows because the complement multiplier is different from one. If five terms of $R$ lie in $N\setminus C$, then the cyclic EGZ theorem in $N/C\cong C_3$ selects three whose quotient product is one. Their three orbit sets multiply to $C$. Together with eighteen copies of $a$, these terms therefore give $C$ in length $21$. Unused copies of $a$ then give the product set $C$ at every length $21\leq j\leq M+3$, since $C=a^jC$. Lemma~\ref{lem:interval-completion}, with $H=C$ and $b=2520/7=360$, applies because
$$ M+3-21=M-18\geq1243\geq359=b-1. $$
This contradicts the choice of $S$. Thus at most four terms lie in $N\setminus C$.

There are three $C$-double cosets in $L_i\setminus N$, each of size $49$. The matrix argument in Section~\ref{sec:psl2-7-t7-seed} proves that fourteen terms in any one of these double cosets, together with $84$ copies of $a$, have product set $L_i$. Their length is $98$. Unused copies give the product set $L_i$ at every length $98\leq j\leq M+14$, since $L_i=a^jL_i$. Applying Lemma~\ref{lem:interval-completion} with $H=L_i$ and $b=15$ is valid because $M+14-98=M-84\geq1177\geq14$. Therefore each of these three double cosets contains at most thirteen terms of $R$. The total is at most $4+3\cdot13=43$, proving Claim 1.

By Lemma~\ref{lem:subgroup-cap}, applied to $C$, at least $1260-7+2=1255$ terms of $R$ lie outside $C$. By Claim 1, at least $1255-86=1169$ of these terms lie outside $L_1\cup L_2$. Call these terms generating. For each such term $x$, we have $\langle C,x\rangle=A_7$. Each generating $C$-double coset has size $49$, because two distinct subgroups of order seven have trivial intersection. There are consequently at most $\lfloor2520/49\rfloor=51$ generating double cosets.

\smallskip
\noindent\emph{Claim 2.}
For every $x\notin C$, there is an integer $2\leq l\leq12$ such that $O_a(x)^l$ contains at least two elements of $C$.

For the natural action on seven points,
$$ \sum_{j=0}^{6}|\Fix(xa^j)|=7. $$
Indeed, for each point $v$, exactly one power $a^j$ sends $v$ to $x^{-1}(v)$, since $C$ acts regularly. No member of $xC$ is the identity. A nonidentity even permutation with a fixed point has order $2$, $3$, $4$, or $5$; the corresponding maximum numbers of fixed points are $3$, $4$, $1$, and $2$. There are at least two members with a fixed point, since each member has at most four fixed points.

There must be two members of $xC$ whose orders have least common multiple at most $12$. Otherwise no two members with a fixed point could have the same order, order two could not occur, and orders three and four could not both occur. Thus there would be at most two members with a fixed point, with orders either $3,5$ or $4,5$. Their total number of fixed points would be at most six, contrary to the displayed identity.

Let $l$ be the least common multiple of these two orders. Then $l\leq12$ and $7\nmid l$, so Lemma~\ref{lem:power-relation-seeds} implies that $O_a(x)^l$ contains two elements of $C$. This proves Claim 2.

The at most fifty-one generating double cosets thus supply subsequences of length at most twelve and weight one. In Lemma~\ref{lem:seed-selection},
$$ E_0\leq51(12-1)=561 \quad\text{and}\quad \lambda\leq12. $$
Since $1169>561+12(7-2)=621$, a cyclic seed uses at most $6\cdot12=72$ terms.

The incidence argument verifies the hypothesis of Lemma~\ref{lem:orbit-generation} for every generating term. By Lemma~\ref{lem:prime-double-cosets}, we may use $h=3$, $\Omega_7$, and the bounds $b=49$ and $d=56$. Together with $E=1169$ and $s=72$, these parameters establish the $A_7$ row of Table~\ref{tab:orbit-growth}. By Proposition~\ref{prop:orbit-growth-criterion}, the conclusion follows.
\end{proof}

\subsection{Weighted seeds from element orders}\label{sec:weighted-seeds}

\subsubsection{The order-three case in $\PSL_2(7)$}\label{sec:psl2-7-t3-seed}

\begin{proposition}\label{prop:psl2-7-order-three}
In the setting \eqref{eq:high-mult-setup}, the case $G=\PSL_2(7)$, $t=3$ satisfies $1\in\Pi_{168}(S)$.
\end{proposition}

\begin{proof}
Work with determinant-one matrices over $\mathbb{F}_7$, modulo signs, and let
$$ a=\begin{pmatrix}2&0\\0&4\end{pmatrix} \quad\text{and}\quad C=\langle a\rangle. $$
The maximal overgroups of $C$ are two groups $C_7\rtimes C_3$ and two groups $S_4$. Indeed, the two triangular groups fix the two eigenlines. Dickson's classification \cite[Theorem~2.2]{Giudici2007} gives two conjugacy classes of $S_4$ subgroups, each with seven members. Each member contains four order-three subgroups, and $G$ contains $168/6=28$ such subgroups. Thus exactly one member of each class contains $C$. The covering bounds already proved give at most $4+4+8+8=24$ remainder terms outside $C$ in their union. Since at least $m-t+2=83$ remainder terms lie outside $C$, at least $83-24=59$ remainder terms $x$ satisfy $\langle C,x\rangle=G$. Call these terms eligible.

\smallskip
\noindent\emph{Claim 1.}
There are ten eligible double cosets. Up to automorphisms preserving $C$, their representatives are
$$ x_k=\begin{pmatrix}1&1\\k-1&k\end{pmatrix} \quad\text{and}\quad z=\begin{pmatrix}0&1\\6&1\end{pmatrix}. $$
Here $k=2,4,6$. The three $x_k$ types each account for two double cosets; the $z$ type accounts for four.

\smallskip
\noindent\emph{Proof of Claim 1.}
The two groups $C_7\rtimes C_3$ are the upper and lower triangular subgroups. Thus an eligible matrix has both off-diagonal entries nonzero. If its two diagonal entries are nonzero, then left and right multiplication by $C$, and simultaneous negation, give the unique form
$$ x_{k,\epsilon}= \begin{pmatrix} 1&\epsilon\\ (k-1)/\epsilon&k \end{pmatrix}, $$
where $k\in\{2,3,4,5,6\}$ and $\epsilon\in\{1,3\}$. Indeed, the diagonal product $k$ and the square class of the product of the upper-left and upper-right entries are invariant. Diagonal rescaling normalizes the upper-left entry to one and the upper-right entry to the indicated representative of its square class. If exactly one diagonal entry vanishes, then the same rescaling gives
$$ \begin{pmatrix}0&\epsilon\\-1/\epsilon&1\end{pmatrix} \quad\text{or}\quad \begin{pmatrix}1&\epsilon\\-1/\epsilon&0\end{pmatrix}, $$
where $\epsilon\in\{1,3\}$. If both diagonal entries vanish, then the matrix normalizes $C$ and is not eligible. Conjugation by $\diag(3,1)$ exchanges the two square classes and fixes $a$. Conjugation by $\begin{psmallmatrix}0&-1\\1&0\end{psmallmatrix}$ interchanges the two forms with a zero diagonal entry and inverts $a$.

For $k=3$, the projective relations $a^3=x_{3,1}^4=(ax_{3,1})^2=1$ place the generated subgroup in a quotient of the spherical triangle group $S_4$. For $k=5$, the relations $(a^2)^3=x_{5,1}^3=(a^2x_{5,1})^2=1$ similarly give a quotient of $A_4$. These four double cosets are not eligible.

The remaining types generate $G$. The element $x_2$ has order four, and the commutator $ax_2a^{-1}x_2^{-1}$ is represented by $\begin{psmallmatrix}5&3\\2&0\end{psmallmatrix}$, a noncentral matrix of trace $-2$, hence of projective order seven. For $x_4$, the orders of $x_4$ and $ax_4$ are seven and four. For $x_6$, the orders of $x_6$ and $ax_6$ are two and seven. For $z$, the orders of $az$ and $a^2z$ are four and seven. No proper subgroup of $G$ contains these respective pairs of orders, since its maximal subgroups are $S_4$ and $C_7\rtimes C_3$. This proves Claim 1.

Every double coset outside $N_G(C)$ has size nine, since $|C|=3$ and a nontrivial intersection of two conjugates of $C$ forces them to coincide. The normalizer consists of the diagonal and antidiagonal matrices and has order six.

\smallskip
\noindent\emph{Claim 2.}
At most fourteen eligible terms give an orbit product containing $C$.

\smallskip
\noindent\emph{Proof of Claim 2.}
Cayley--Hamilton applied to the displayed matrices gives
$$ \begin{array}{c|ccc} x&\ord(x)&\ord(xa)&\ord(xa^2)\\ \hline x_2&4&4&3\\ x_6&2&7&7\\ z&3&4&7 \end{array} $$
By Lemma~\ref{lem:power-relation-seeds}, both $O(x_2)^4$ and $O(x_6)^7$ contain two members of $C$. For $z$, the same lemma implies that $1\in O(z)^3$, $a^2\in O(z)^4$, and $a\in O(z)^7$; thus $O(z)^7=O(z)^4O(z)^3$ contains both $a$ and $a^2$. For $x=x_4$, let $X=axa^{-1}$. Direct multiplication gives
$$ (xX)^2=1 \quad\text{and}\quad x^2X^2=a $$
in $\PSL_2(7)$, so $O(x_4)^4$ also contains two members of $C$. These conclusions pass to all terms in the corresponding double cosets. Every selected subsequence may therefore be assigned weight one, with lengths four for the $x_2,x_4$ types and seven for the $x_6,z$ types. Hence
$$ E_0=2\cdot3+2\cdot3+2\cdot6+4\cdot6=48 \quad\text{and}\quad \lambda=7. $$
By Lemma~\ref{lem:seed-selection}, the inequality $59>48+7=55$ ensures that two subsequences can be selected using at most fourteen terms. This proves Claim 2.

\smallskip
\noindent\emph{Claim 3.}
For every eligible $x$, the set $B=CO(x)^2=(CxC)^2$ satisfies
$$ |B|\geq21 \quad\text{and}\quad |BB^{-1}|\geq48. $$

\smallskip
\noindent\emph{Proof of Claim 3.}
It suffices to use the four representatives in Claim 1. A matrix is denoted by the string $\alpha\beta\gamma\delta$ of its four entries. The products $xa^jx$, for $j=0,1,2$, have representatives
$$ \begin{array}{c|ccc} x&j=0&j=1&j=2\\ \hline x_2&5442&1443&1662\\ x_4&3262&0320&3512\\ x_6&1001&6230&0231\\ z&1610&3432&2522 \end{array} $$
For $x_2$, the first two matrices have the same diagonal product three but opposite square classes of $\alpha\beta$; the third has diagonal product two. For $x_4$, the first and last have diagonal product six and opposite square classes, while the middle matrix is antidiagonal. For $x_6$, the first is the identity and the other two have their zero diagonal entry in different positions. For $z$, the three diagonal products are $0,6,4$. These tests distinguish the three double cosets in every row. The sizes of these double cosets are $9,9,9$ or $9,3,9$, proving that $|B|\geq21$.

For the difference-set estimate, define
$$ W(i,j,k)=xa^ixa^jx^{-1}a^kx^{-1}. $$
Each $CW(i,j,k)C$ lies in $BB^{-1}$. The following table gives five such products for each representative; an entry consists of the triple $ijk$, followed by the four entries of its product matrix:
$$ \begin{array}{c|ccccc} x&\multicolumn{5}{c}{ijk:W(i,j,k)}\\ \hline x_2&001:0541&010:5551&011:1532&012:6313&021:1061\\ x_4&001:3243&002:3533&010:1230&011:3651&020:0541\\ x_6&001:6230&002:0231&111:2462&112:4302&121:2623\\ z&001:4022&011:2433&012:5631&022:0162&121:6253 \end{array} $$
These identities involve only multiplication of four fixed $2\times2$ matrices and three diagonal powers. None of the resulting matrices is diagonal or antidiagonal, so each of their double cosets has size nine. In the first row, their diagonal products are $0,5,2,4,1$, all different. In the second row, the first two have diagonal product two and opposite square classes of $\alpha\beta$; the third and fifth have diagonal product zero but different vanishing diagonal positions; the fourth has diagonal product three. In the third row, the first two are separated by their vanishing diagonal position, and the other three have diagonal products $4,1,6$. In the last row, the five diagonal products are $1,6,5,0,4$. Thus the five double cosets in each row are disjoint.

In addition, $C\subseteq BB^{-1}$, because $B$ is nonempty and bi-$C$-invariant. It follows that $|BB^{-1}|\geq5\cdot9+3=48$. The difference set generates $G$, as every proper subgroup has order at most twenty-four. This proves Claim 3.

Claims 1--3 give $E=59$, $Q=10$, a cyclic seed of cost at most $s=14$, and paired factors with $b=21$, $d=48$, and generating differences. These parameters establish the first row of Table~\ref{tab:orbit-growth}; Proposition~\ref{prop:orbit-growth-criterion} completes the proof.
\end{proof}

\subsubsection{The order-four case in $\PSL_2(7)$}\label{sec:psl2-7-t4-seed}

\begin{proposition}\label{prop:psl2-7-order-four}
In the setting \eqref{eq:high-mult-setup}, the case $G=\PSL_2(7)$, $t=4$ satisfies $1\in\Pi_{168}(S)$.
\end{proposition}

\begin{proof}
Identify $G$ with $\GL_3(2)$, and write
$$ a= \begin{pmatrix} 1&1&0\\0&1&1\\0&0&1 \end{pmatrix} \quad\text{and}\quad C=\langle a\rangle= \left\{ \begin{pmatrix} 1&u&v\\0&1&u\\0&0&1 \end{pmatrix}:u,v\in\mathbb{F}_2 \right\}. $$
Here $n=168$, $m=84$, and $M\geq85$. The two maximal overgroups of $C$ are the stabilizers of the point $p=\langle e_1\rangle$ and plane $L=\langle e_1,e_2\rangle$, both isomorphic to $S_4$. Thus at least $84-22-22=40$ remainder terms lie outside their union. Call these terms eligible. We shall use at most $26$ of them.

\smallskip
\noindent\emph{Claim 1.}
There are eight eligible $C$-double cosets, with representatives $x_1,\ldots,x_8$ in the following table. If $B=Cx_iC$, then $|B|=16$ and $|BB^{-1}|=52$.

Represent a binary matrix by its three rows. The eight double cosets fall into four pairs under automorphisms preserving $C$:
$$ \begin{array}{c|ccc} (i,j)&x_i&x_j&\text{conjugate indices}\\ \hline (1,4)&(001;010;100)&(010;011;100)&8,1,5\\ (2,6)&(001;011;100)&(011;010;100)&5,1,8\\ (3,7)&(010;001;100)&(001;100;010)&6,9,2\\ (5,8)&(011;001;100)&(001;110;010)&2,9,6 \end{array} $$
Here $x_9=(100;001;010)$; the last column gives the double-coset indices of $x_ia^kx_i^{-1}$ for $k=1,2,3$. For the pairing, let $h=(110;010;001)$ and $\sigma(g)=x_1g^{-T}x_1$. Both conjugation by $h$ and $\sigma$ invert $a$, and
$$ \begin{gathered} hx_1ha^3=x_4 \quad\text{and}\quad hx_2ha^3=x_6,\\ \sigma(x_3)=x_7 \quad\text{and}\quad \sigma(x_5)=x_8. \end{gathered} $$
Thus orbit-product bounds pass between each pair.

\smallskip
\noindent\emph{Proof of Claim 1.}
The normalizer $N_G(C)$ is the upper unitriangular subgroup, of order eight. It also equals $C_G(a^2)$: indeed, $a^2=I+E_{13}$, and commuting with $E_{13}$ forces the first column and last row to be $(1,0,0)^T$ and $(0,0,1)$, respectively. If $C$ intersects $xCx^{-1}$ nontrivially, then the unique involutions coincide, so $x\in C_G(a^2)=N_G(C)$. Consequently every double coset outside this normalizer has size sixteen.

Let $x$ be eligible. Then either $x_{31}=1$, or $x_{31}=0$ and $x_{21}=x_{32}=1$. In the first case, left and right multiplication by the displayed matrices in $C$ gives a unique form
$$ \begin{pmatrix} 0&b&c\\0&d&e\\1&0&0 \end{pmatrix}, $$
where $be+cd=1$. To obtain this form, use the left parameter $u$ to clear $x_{21}$, the right parameter $u$ to clear $x_{32}$, and the left and right parameters $v$ to clear $x_{11}$ and $x_{33}$. There are exactly six choices for $(b,c,d,e)$, giving $x_1,\ldots,x_6$. In the second case, the same row and column operations first clear $x_{11}$ and $x_{33}$, and then $x_{12}$ and $x_{23}$. The unique form is
$$ \begin{pmatrix} 0&0&1\\1&d&0\\0&1&0 \end{pmatrix}, $$
where $d\in\mathbb{F}_2$, giving $x_7,x_8$. This proves completeness and disjointness. The extra representative $x_9$ lies in the point stabilizer outside $N_G(C)$, so its double coset is distinct from the eight eligible double cosets.

For the four representatives, multiply $x_ia^kx_i^{-1}$ and apply the same row and column reductions, obtaining the three indices displayed in the last column. These indices are distinct, and none corresponds to a double coset in $N_G(C)$. Since
$$ BB^{-1}=Cx_iCx_i^{-1}C =C\cup\bigcup_{k=1}^{3}Cx_ia^kx_i^{-1}C, $$
the union is disjoint and has size $4+3\cdot16=52$. This proves Claim 1.

\smallskip
\noindent\emph{Claim 2.}
At most fifteen eligible terms give an orbit product containing $C$.

\smallskip
\noindent\emph{Proof of Claim 2.}
For the first three representatives, the element orders are
$$ \begin{array}{c|cccc} i&\ord(x_i)&\ord(x_ia)&\ord(x_ia^2)&\ord(x_ia^3)\\ \hline 1&2&7&3&7\\ 2&4&4&3&3\\ 3&3&3&7&7 \end{array} $$
These orders follow by squaring or cubing, or from an irreducible cubic characteristic polynomial for order seven. By Lemma~\ref{lem:power-relation-seeds}, both $O(x_2)^3$ and $O(x_3)^3$ contain consecutive powers of $a$. For $x_1$, the same lemma and $7=2+2+3$ imply that $\{a,a^2,a^3\}\subseteq O(x_1)^7$. For $x=x_5$, let $X_j=a^jxa^{-j}$. The two remaining identities are
$$ X_0^3X_2X_1=a \quad\text{and}\quad X_0X_1X_0X_2X_1=a^2. $$
The pairing in Claim 1 gives all eight double cosets from these four calculations.

The four three-term subsequences and two five-term subsequences have weight one, and the two seven-term subsequences have weight two. Thus
$$ E_0=4\cdot2+2\cdot4+2\cdot6=28 \quad\text{and}\quad \lambda=5. $$
Since $40>28+2\cdot5=38$, Lemma~\ref{lem:seed-selection} applies. An inclusion-minimal collection of weight at least three consists of three subsequences of weight one, one subsequence of each weight, or two subsequences of weight two. Their lengths are at most fifteen, twelve, and fourteen, respectively, which proves Claim 2.

Claims 1--2 give $E=40$, a cyclic seed of cost at most $s=15$, and factors $CO(x)$ with size $16$ and difference size $52$. The differences generate $G$, since every proper subgroup has order at most twenty-four. Thus the second row of Table~\ref{tab:orbit-growth} applies, and Proposition~\ref{prop:orbit-growth-criterion} completes the proof.
\end{proof}

\subsection{Torus seeds with a common weighted-growth calculation}\label{sec:split-torus-seeds}

We prove the high-multiplicity assertion for $G=\PSL_2(q)$ with $(q,t)=(11,5),(13,6)$. As usual, $n=|G|=2m$, the sequence $S=a^{[M]}R$ satisfies \eqref{eq:high-mult-setup}, and $C=\langle a\rangle$ is a full split torus. All products of matrices in this subsection are understood modulo $\{I,-I\}$.

\smallskip
\noindent\emph{The matrix normalization.}
Work over $\mathbb{F}_q$, let $z=2$, and let $\mu=\langle z^2\rangle$. The order of $z$ is $2t$, and $C$ is represented by $\diag(z,z^{-1})$. Choosing a different generator of $C$ does not change its conjugation orbits.

A generating representative with all four entries nonzero has orbit normal form
\begin{equation} X_j= \begin{pmatrix} \alpha&z^{2j}\\ (k-1)z^{-2j}&\delta \end{pmatrix}. \tag{14.8}\label{eq:split-orbit-form} \end{equation}
Here $0\leq j<t$, $\alpha=1$, and $\delta=k\notin\{0,1\}$. Indeed, torus left multiplication makes the upper-left entry one, and diagonal conjugation makes the upper-right entry one; the determinant fixes the lower-left entry as $k-1$. Diagonal conjugation fixes $C$ pointwise, and torus left multiplication translates every orbit power by an element of $C$. Both operations preserve the assertion that an orbit power contains $C$.

If exactly one diagonal entry is zero, then interchange the eigenlines when necessary and use the same normal form with $k=0$, $\alpha=0$, and $\delta=1$. The same two operations give this normalization. If both diagonal entries are zero, or an off-diagonal entry is zero, then the element lies in a proper torus overgroup and is excluded from the generating terms. Thus the displayed forms cover every term used below.

\smallskip
\noindent\emph{Torus products from short power relations.}
We apply Lemma~\ref{lem:power-relation-seeds} before making any word calculations. In the tables below, a pair $r:I$ means that $(X_0a^i)^r=1$ for every $i\in I$. It therefore supplies $a^{-ri}\in O^r$, where $O=\{X_0,\ldots,X_{t-1}\}$. Concatenating these products gives the exponents in the column $P_s$, so that $\{a^j:j\in P_s\}\subseteq O^s$. At length six, we use two three-term products, three two-term products, or one six-term product; at length seven, we use one seven-term product or lengths $2+2+3$. Thus this column requires only addition modulo $t$.

The short power relations follow directly from Cayley--Hamilton. More explicitly, a determinant-one matrix $Y$ of trace $T$ satisfies
$$ Y^r=F_r(T)Y-F_{r-1}(T)I. $$
Here the polynomials are defined by $F_0=0$, $F_1=1$, and $F_{r+1}=TF_r-F_{r-1}$. For $Y=X_0a^i$, the trace is $\alpha2^i+\delta2^{-i}$. Every displayed relation has $F_r(T)=0$, so $Y^r=\pm I$. No element-order classification is required.

A word $i_1\cdots i_s$ in the last column denotes $X_{i_1}\cdots X_{i_s}$. An entry $j:w$ asserts that the product of the word $w$ is $a^j$; these are the few extra products not supplied by the power relations. They are checked by the common scalar rule
$$ \begin{aligned} A'&=\alpha A+(k-1)2^{-2i}B &\quad\text{and}\quad B'&=2^{2i}A+\delta B,\\ C'&=\alpha C+(k-1)2^{-2i}D &\quad\text{and}\quad D'&=2^{2i}C+\delta D, \end{aligned} $$
starting with $(A,B,C,D)=(1,0,0,1)$. For an entry corresponding to $j$, the result is $\pm(2^j,0,0,2^{-j})$, with arithmetic modulo $q$.

For $q=11$ and $t=5$, every row has $s=6$:
\begingroup
\small
\setlength{\arraycolsep}{4pt}
$$ \begin{array}{r|l|l|l} k&\text{power relations }r:I&P_6&\text{additional products }j:w\\ \hline 0&3:\{0\},\ 6:\{1\}&0,4&1:010232\\ 2&3:\{2,4\}&1,2,3&\text{---}\\ 3&3:\{4\}&1&0:001101,\ 2:010131\\ 4&6:\{0,2\}&0,3&1:013412\\ 5&3:\{1,3\}&2,3,4&\text{---}\\ 6&2:\{2\},\ 6:\{1,3\}&2,3,4&\text{---}\\ 7&2:\{1\}&4&1:003311,\ 2:021324\\ 8&2:\{4\},\ 6:\{1,2\}&1,3,4&\text{---}\\ 9&3:\{0,1\}&0,2,4&\text{---}\\ 10&3:\{2,3\}&0,2,3&\text{---} \end{array} $$
\endgroup
For $q=13$ and $t=6$, the calculations are as follows:
\begingroup
\small
\setlength{\arraycolsep}{3pt}
$$ \begin{array}{rr|l|l|l} k&s&\text{power relations }r:I&P_s&\text{additional products }j:w\\ \hline 0&7&7:\{1,2,3\}&3,4,5&0:0000101\\ 2&7&7:\{0,1,3,4\}&0,2,3,5&\text{---}\\ 3&7&7:\{1,2,3\}&3,4,5&1:0030334\\ 4&7&7:\{0,2\}&0,4&1:0330330,\ 2:0544005\\ 5&7&7:\{0,3\}&0,3&1:0000014,\ 2:0001403\\ 6&7&7:\{0,1,4,5\}&0,1,2,5&\text{---}\\ 7&6&3:\{1,4\}&0,3&1:002252,\ 2:000222\\ 8&7&7:\{1,2,4,5\}&1,2,4,5&\text{---}\\ 9&7&2:\{1\},\ 3:\{3\},\ 7:\{0,2,4\}&0,2,4,5&\text{---}\\ 10&7&2:\{2\},\ 3:\{5\},\ 7:\{1,3\}&1,3,5&4:0003303\\ 11&6&3:\{0,1\}&0,3&1:002022,\ 2:003430\\ 12&7&7:\{1,2,3,4,5\}&1,2,3,4,5&\text{---} \end{array} $$
\endgroup
In each row, the listed products give a set $D\subseteq C$ with $|D|>|C|/2$. Hence $D^2=C$, because $|D|+|cD^{-1}|>|C|$ for every $c\in C$. Squaring the corresponding orbit power proves the seed bounds in both split cases:
$$ \begin{array}{l|rr} (q,t)&(11,5)&(13,6)\\ \hline \text{number of terms from one double coset sufficient}&12&14 \end{array} $$
Each bound means \emph{at most} the displayed number: the two six-term rows at $q=13$ need only twelve terms. For actual terms in the same double coset, their orbit sets are toral translates of the normalized orbit. The product of the chosen factors thus contains a translate of $C$, hence $C$ itself.

\smallskip
\noindent\emph{Availability of a seed.}
By Lemma~\ref{lem:cyclic-amplification}, every generating double coset has size $t^2$. The number of generating double cosets is therefore at most $\lfloor n/t^2\rfloor$.

For $(13,6)$, the proper overgroups are the normalizer of order $12$ and the two eigenline Borels of order $78$. Each Borel contributes at most ten terms outside $C$: apply Lemma~\ref{lem:subgroup-cover-caps}(i) with $(\ell,t)=(13,6)$, so that $k_0=6$ and $k_0+t-2=10$. The normalizer contributes at most ten terms. Since at least $m-t+2=542$ terms lie outside $C$, at least $512$ are generating.

For $(11,5)$, the proper overgroups are the two eigenline Borels, together with the two $A_5$ overgroups. The normalizer of $C$ is contained in either $A_5$. By Lemma~\ref{lem:subgroup-cover-caps}(i), with $(\ell,t)=(11,5)$ and $k_0=5$, each Borel contributes at most eight terms outside $C$. By Lemma~\ref{lem:subgroup-cover-caps}(iv), each $A_5$ contributes at most eleven terms. Thus at most $2\cdot8+2\cdot11=38$ of the $m-t+2=327$ terms outside $C$ are nongenerating, leaving at least $289$ generating terms.

The assertions about the proper overgroups follow from Dickson's classification \cite{Giudici2007}. For clarity, the two $A_5$ overgroups at $q=11$ consist of one from each of its two conjugacy classes of $A_5$ subgroups. Each class has eleven subgroups, each containing six order-five subgroups, while $G$ has $660/10=66$ order-five subgroups. Double counting gives one overgroup from each class for every $C$.

The pigeonhole inequalities needed for the two seed bounds are
$$ \begin{array}{c|rrrr} (q,t)&\text{generating terms at least}&Q\text{ at most}&s_0&Q(s_0-1)\\ \hline (11,5)&289&26&12&286\\ (13,6)&512&30&14&390 \end{array} $$
In every row, more than $Q(s_0-1)$ generating terms are available, so one generating double coset supplies the required seed.

\smallskip
\noindent\emph{Amplified difference sets.}
After constructing the seed, restrict its product set to $C$. Subsequent product sets are bi-$C$-invariant. By Lemma~\ref{lem:cyclic-amplification}, an orbit with four nonzero entries may therefore be replaced by $D=CxC$ at no additional cost, with the difference-set bound stated in that lemma.

These differences generate $G$. The disjoint eigenline pairs exclude Borels and the torus normalizer. At $(q,t)=(13,6)$, there is no exceptional overgroup. At $(11,5)$, the only remaining possibilities are the $A_5$ overgroups; each is the unique overgroup of $C$ in its conjugacy class, as counted above. By Lemma~\ref{lem:orbit-generation}, these possibilities are excluded as well.

Pair the split terms with a zero diagonal entry within each type, leaving at most two unused. By Lemma~\ref{lem:zero-diagonal-pairs}, over the prime field, each pair satisfies $|CB|=t^3$, and its differences generate $G$. Consequently, the amplified paired factor satisfies $|(CB)(CB)^{-1}|\geq t^3$, and its differences generate $G$.

\smallskip
\noindent\emph{Completion in the split cases.}
For $(11,5)$, a factor with four nonzero entries has size $25$ and growth constant $\lceil(55+1)/4\rceil=14$; a paired factor has size $125$ and growth constant at least $32\geq2\cdot14$. For $(13,6)$, the corresponding sizes are $36,216$ and growth constants are $20,55\geq2\cdot20$. Lemma~\ref{lem:weighted-growth} therefore applies with
$$ (n,b,\delta,A_b,B_\delta) =(660,25,14,1,4) \quad\text{and}\quad (1092,36,20,1,5), $$
respectively. Indeed, the paired sizes exceed $4\delta-1$, and $4\delta-1<m$ reads $55<330$ and $79<546$. The additional costs satisfy
$$ W<\frac{662}{14}+6<54 \quad\text{and}\quad W<\frac{1094}{20}+7<62, $$
so they are at most $53$ and $61$, respectively. Starting with the seed $C$ does not change the first-factor lower bound, since $|CB|\geq|B|$. There are enough unused generating terms, even after pairing, because $289-12\geq53+2$ and $512-14\geq61+2$. Finally, the total numbers $s$ of remainder terms used satisfy
$$ \begin{array}{c|rrrr} (q,t)&s\text{ at most}&(t-1)s&m+1&ts\leq n\\ \hline (11,5)&65&260&331&325\leq660\\ (13,6)&75&375&547&450\leq1092 \end{array} $$
Identity \eqref{eq:orbit-realization} realizes every factor with its remainder terms and the indicated copies of $a$. By Lemma~\ref{lem:high-mult-completion}, the resulting full product set extends to length $n$.

\subsection{One completion table for the orbit calculations}\label{sec:orbit-completion}

The following calculations apply Proposition~\ref{prop:orbit-growth-criterion} to the bounds established above. In the table, $(q,t)$ denotes $(\PSL_2(q),t)$, a dash means that $Q$ is not needed, and $w=s+rk$ is the maximum number of remainder terms used.

\begin{table}[htbp]
\centering
\caption{Bounds for the common subsequence-growth criterion.}\label{tab:orbit-growth}
\small
\begin{tabular}{c*{9}{r}}
\hline
$(G,t)$&$s$&$r$&$E$&$Q$&$b$&$d$&$k$&$w$&$(t-1)w$\\
\hline
$(7,3)$&14&2&59&10&21&48&11&36&72\\
$(7,4)$&15&1&40&---&16&52&11&26&78\\
$(8,3)$&0&3&236&54&15&49&39&117&234\\
$(8,9)$&0&1&236&---&9&73&30&30&240\\
$(11,6)$&0&2&320&18&25&76&33&66&330\\
$(13,7)$&0&2&534&22&36&99&44&88&528\\
$(17,8)$&0&2&936&38&49&160&60&120&840\\
$(A_6,5)$&16&1&155&---&25&30&28&44&176\\
$(A_7,7)$&72&1&1169&---&49&56&102&174&1044\\
$(\PSL_3(3),13)$&96&1&2760&---&169&182&34&130&1560\\
\hline
\end{tabular}
\end{table}

For the first two rows, take $\Omega=\{t,2t,\ldots,n\}$. The successive lower bounds, including the final application of $v+b>n$, are
$$ \begin{array}{c|l} (7,3)&21,\ 33,\ 48,\ 63,\ 78,\ 93,\ 108,\ 123,\ 138,\ 153,\ 168,\\ (7,4)&16,\ 28,\ 44,\ 60,\ 76,\ 92,\ 108,\ 124,\ 140,\ 156,\ 168. \end{array} $$
For the next five rows, take $\Omega=\{1,\ldots,n\}$. The higher orbit powers give the recurrences in the following table:
$$ \begin{array}{c|c|l|l} (q,t)&(v_1,v_2)&\text{linear part}&\text{last values}\\ \hline (8,3)&(15,23)&v_j=35+13(j-3)&v_{38}=490\text{ and }v_{39}=504\\ (11,6)&(25,38)&v_j=58+20(j-3)&v_{32}=638\text{ and }v_{33}=660\\ (13,7)&(36,55)&v_j=80+25(j-3)&v_{43}=1074\text{ and }v_{44}=1092\\ (17,8)&(49,74)&v_j=112+41(j-3)&v_{59}=2408\text{ and }v_{60}=2448 \end{array} $$
The respective ranges for the linear parts are $3\leq j\leq38$, $3\leq j\leq32$, $3\leq j\leq42$, and $3\leq j\leq59$.

For $(8,9)$, the first four values are $9,14,22,34$, followed by
$$ v_j=52+19(j-5) $$
for $5\leq j\leq27$, and then
$$ v_{28}=488,\qquad v_{29}=497, \quad\text{and}\quad v_{30}=504. $$
For the last three rows, use $\Omega_p$ from Lemma~\ref{lem:prime-double-cosets}. The parameters $(p,h,q)$ are $(5,2,14)$, $(7,3,51)$, and $(13,3,33)$, respectively. For $A_6$ and $A_7$, writing $b=p^2$, the recurrence is
$$ v_{2j-1}=bj \quad\text{and}\quad v_{2j}=bj+ph, $$
where $1\leq j\leq q$ for the first formula and $1\leq j<q$ for the second, followed by $v_{2q}=n$. For $\PSL_3(3)$, the recurrence is $v_j=169j$ for $1\leq j\leq33$, followed by $v_{34}=5616$.

Indeed, the constant increments before rounding are $8,15,46$. At every displayed nonfinal update, the variable increment in Proposition~\ref{prop:orbit-growth-criterion} is at least this constant. The final steps use, respectively,
$$ 350+25>360,\qquad 2499+49>2520, \quad\text{and}\quad 5577+169>5616. $$
Every other final step above also uses $v+b>n$.

Finally, the values of $E-s-(r-1)Q$, in table order, are
$$ 35,\ 25,\ 128,\ 236,\ 302,\ 512,\ 898,\ 139,\ 1097,\ 2664, $$
and each is at least $rk$. The displayed copy counts are at most $m+1$. The total lengths $tw$ are, respectively,
$$ 108,\ 104,\ 351,\ 270,\ 396,\ 616,\ 960,\ 220,\ 1218,\ 1690, $$
and each is at most $n$. Thus all three inequalities in Proposition~\ref{prop:orbit-growth-criterion} hold, proving the completion asserted in each local application of the table.

\section{Structural coverage and the simple-group theorem}\label{sec:structural-coverage}

This section uses the classification of finite non-abelian simple groups into alternating groups, groups of Lie type, and sporadic groups, together with the standard isomorphisms in low rank~\cite{Solomon2018,Wilson2009}. We use the group-order formulas in those references, the bounds for maximum element orders in~\cite{GuestMorrisPraegerSpiga2015,KantorSeress2009,MarotiTamburiniBellani2013}, and the sporadic data in~\cite{WilsonEtAlND}. The element-order bounds below apply to all elements, including those with a nontrivial unipotent part.

\begin{proposition}\label{prop:simple-coverage}
Let $G$ be non-abelian simple, and let $n=|G|=2m$. Define $u=\meo(G)$ and $b=\lfloor3\sqrt{n}\rfloor$. Every group outside the list
\begin{equation} A_5,\ A_6,\ A_7,\ A_8,\ \PSL_2(q),\ \PSL_3(3),\ \PSU_3(3),\ M_{11} \tag{15.1}\label{eq:local-coverage-list} \end{equation}
satisfies
$$ (2u-2)b+1\leq m. $$
For a group on this list and an element of order $t>1$, either one of the criteria in Table~\ref{tab:coverage-criteria} applies, one of the subgroup bounds in~\eqref{eq:subgroup-coverage} applies, or $(G,t)$ occurs in Table~\ref{tab:remaining-pairs}. Thus these criteria and the listed local arguments cover every possible repeated element in the setting of~\eqref{eq:high-mult-setup}.
\end{proposition}
\begin{proof}
\emph{1. The uniform comparison.} By Proposition~\ref{prop:distinct-term-criterion}, the condition $n\geq144(u-1)^2+4$ implies $(2u-2)b+1\leq m$. Since $u\geq2$, it is enough to verify
\begin{equation} n\geq144u^2. \tag{15.2}\label{eq:uniform-meo-bound} \end{equation}
We check this condition by three comparisons: factorial growth, increasing products of cyclotomic factors, and a monomial quotient.

For the factorial comparison, a permutation on $d$ letters has order at most $3^{d/3}$. Indeed, its cycle lengths $\ell_i$ satisfy $\prod_i\ell_i\leq\prod_i3^{\ell_i/3}$. At $d=9$, one has $|A_9|=181440>144\cdot3^6$, and the ratio $d!/(2\cdot3^{2d/3})$ increases for $d\geq9$. Thus~\eqref{eq:uniform-meo-bound} covers all $A_d$ with $d\geq9$.

For the product comparisons, define
$$ U_d^L=\frac{q^d-1}{q-1} \qquad\text{and}\qquad U_r^S=\frac{q^{r+1}}{q-1}. $$

\begin{table}[htbp]
\caption{Criteria used after the uniform comparison of maximum element orders. The notation $(q,t)$ means $(\PSL_2(q),t)$.}\label{tab:coverage-criteria}
\small
\begin{tabularx}{\textwidth}{@{}L{0.23\textwidth}>{\raggedright\arraybackslash}XL{0.22\textwidth}@{}}
\toprule
Method & Sufficient conditions or explicit pairs & Result\\
\midrule
Fixed-order extraction & $(u+t-2)b+1\leq m$ & Proposition~\ref{prop:distinct-term-criterion}\\
Borel extraction & $G=\PSL_2(q)$, with $4t\leq q$ for even $q$, or $6t\leq q+1$ for odd $q$ & Lemmas~\ref{lem:extraction-bounds} and~\ref{lem:frobenius-extraction}\\
Compatible centralizer & $q=p^f$ and $t=p$, with $f\geq3$, or $f=2$ and $p\geq5$ & Proposition~\ref{prop:unipotent-extension}\\
Growth from root orbits & $q=p\geq11$ is prime and $t=p$ & Proposition~\ref{prop:prime-root-growth}\\
Weighted semisimple growth & $q\geq11$, $t\geq9$, and either $q$ is prime or $t>\sqrt{q}+1$ & Proposition~\ref{prop:semisimple-completion}\\
Profile extraction & All pairs in Table~\ref{tab:extraction-profiles}, and $(A_6,4)$ & Proposition~\ref{prop:profile-completion} and Corollary~\ref{cor:a6-order-four}\\
Split Frobenius extraction & $(16,5)$ & Corollary~\ref{cor:borel-divisor}\\
\bottomrule
\end{tabularx}
\end{table}

The structural bounds $u\leq U$ yield the following lower bounds for $n/U^2$:
\begin{equation} \begin{array}{c|c|l} \text{Range}&U&\text{Lower bound for }n/U^2\\ \hline \PSL_d(q),\ d\geq4 &U_d^L &(q-1)q^{d^2-2d-1}/2\\[3pt] \PSp_{2r}(q),\ r\geq2 &U_r^S& \displaystyle\frac{(q-1)^2q^{r^2-2r-2}}{\gcd(2,q-1)} \prod_{i=1}^{r}(q^{2i}-1)\\[10pt] \POmega_{2r}^{\pm}(q),\ r\geq4 &U_r^S& \displaystyle\frac{1}{4}q^{r^2-3r-2}(q-1)^2(q^r-1) \prod_{i=1}^{r-1}(q^{2i}-1)\\[10pt] \PSU_d(q),\ d\geq3 &3q^{d-1}/2& \displaystyle\frac{8q^{(d-1)^2}}{27(q+1)} \end{array} \tag{15.3}\label{eq:classical-order-ratios} \end{equation}
The first and last inequalities are strict. We now derive these bounds and check all endpoints.

For the first row, the order formula is
$$ n=\frac{q^{d^2-1}}{z}\prod_{i=2}^{d}(1-q^{-i}), $$
where $z=\gcd(d,q-1)\leq q-1$. Moreover,
$$ \prod_{i=2}^{d}(1-q^{-i})> 1-\frac{1}{q(q-1)}\geq\frac{1}{2}. $$
The resulting lower bound is at least $8192$ for $d\geq5$, and at least $2187$ for $d=4$ and $q\geq3$. For $d=3$, use $z\leq3$ instead. This gives
$$ \frac{n}{u^2}> \frac{(q-1)^2q^2}{3}\left(1-\frac{1}{q(q-1)}\right)\geq574 $$
when $q\geq7$. At $q=5$, the equality $z=1$ improves the bound to $380$.

The second row follows from
$$ n=\frac{q^{r^2}}{\gcd(2,q-1)} \prod_{i=1}^{r}(q^{2i}-1). $$
At $r=3$, the lower bound in~\eqref{eq:classical-order-ratios} is at least $2835$. Increasing the rank from $r-1$ to $r$ multiplies that bound by $q^{2r-3}(q^{2r}-1)>1$. At $r=2$ and $q\geq4$, it is at least $(9/2)(15/16)255>144$. The odd-dimensional orthogonal groups in odd characteristic satisfy the same order comparison, with central divisor two; in characteristic two, they are isomorphic to the symplectic groups.

In the third row, every factor is increasing for $r\geq4$ and $q\geq2$. At the endpoint $(r,q)=(4,2)$, the bound is $42525$.

For the last row, the cited estimate for the maximum order of an arbitrary element is as follows. If $d$ is odd and $q$ is nonprime, then $u\leq q^{d-1}-1$. If $d$ is odd and $q$ is prime, then $u\leq q(q^{d-2}+1)$. If $d$ is even and $q>2$, then $u\leq q^{d-1}+1$. If $d$ is even and $q=2$, then $u\leq4(2^{d-3}+1)$. Each of these bounds is at most $3q^{d-1}/2$.

In the order formula
$$ n=\frac{q^{d^2-1}}{\gcd(d,q+1)} \prod_{i=2}^{d}\bigl(1-(-1)^iq^{-i}\bigr), $$
discarding factors greater than one leaves a product greater than
$$ 1-\sum_{j\geq1}q^{-2j}\geq\frac{2}{3}. $$
This proves the fourth row of~\eqref{eq:classical-order-ratios}. Its bound exceeds $144$ for $d\geq5$, and is at least $1458$ for $d=4$ and $q\geq3$. When $d=3$ and $q\geq7$, the sharper bound $u\leq q(q+1)$ yields
$$ \frac{n}{u^2}\geq\frac{q(q-1)(q^2-q+1)}{3}\geq602. $$

The small endpoints not covered by these comparisons are listed below.
\begin{center}
\small
\begin{tabular}{@{}lrrl@{}}
\toprule
$G$&$n$&$U\geq u$&Comparison\\
\midrule
$\PSL_3(4)$&20160&7&$n>144U^2$\\
$\PSp_4(3)$&25920&13&$n>144U^2$\\
$\PSU_4(2)$&25920&12&$n>144U^2$\\
$\PSU_3(4)$&62400&15&$n>144U^2$\\
$\PSU_3(5)$&126000&30&$(2U-2)1064+1=61713<63000$\\
\bottomrule
\end{tabular}
\end{center}
For $\PSp_4(3)$, integrality improves $U_2^S=27/2$ to $13$. For $\PSL_3(4)$, the irreducible and $2+1$ tori in $\SL_3(4)$ have orders $21$ and $15$, yielding projective orders at most $7$ and $5$. The completely split torus has exponent three. A repeated eigenvalue lies in $\mathbb{F}_4$, since a nontrivial Frobenius orbit repeated twice would require dimension at least four. The determinant condition then makes the semisimple part scalar, so an element with a nontrivial unipotent part has projective order at most four. The last endpoint uses $\lfloor3\sqrt{126000}\rfloor=1064$.

For the monomial comparison, the bounds for the exceptional families are
$$ u\leq c(q+1)^r \qquad\text{and}\qquad n>\frac{q^a}{2z}. $$
The parameters are as follows.
\begin{center}
\begin{tabular}{@{}lrrrrr@{}}
\toprule
Family&$q_{\min}$&$a$&$r$&$c$&$z$\\
\midrule
$G_2$&3&14&2&8&1\\
${}^3D_4$&2&28&4&8&1\\
$F_4$&2&52&4&16&1\\
$E_6$&2&78&6&16&3\\
${}^2E_6$&2&78&6&16&3\\
$E_7$&2&133&7&32&2\\
$E_8$&2&248&8&32&1\\
\bottomrule
\end{tabular}
\end{center}
The order estimates follow by discarding factors above one and using $\prod(1-q^{-i})>1/2$ for the distinct indices $i\geq2$. In even characteristic, the semisimple-order bounds and the unipotent-order bounds in \cite[Table 5]{GuestMorrisPraegerSpiga2015} give $u\leq c(q+1)^r$. In odd characteristic, the explicit maximum-order formulas in \cite[Table A.7]{KantorSeress2009} give the stronger bound $u\leq (q+1)^r$. Since $a>2r$, the common quotient
$$ R(q)=\frac{q^a}{2zc^2(q+1)^{2r}} $$
is increasing. Its first two endpoint values are $4782969/32768>144$ and $2097152/6561>144$. The other five exceed $10^9$, $10^{14}$, $10^{14}$, $10^{29}$, and $10^{63}$, respectively.

The remaining twisted families satisfy the following increasing bounds:
\begin{center}
\small
\begin{tabularx}{\textwidth}{@{}L{0.29\textwidth}l>{\raggedright\arraybackslash}X@{}}
\toprule
Range&$U\geq u$&Lower bound for $n/U^2$\\
\midrule
${}^2B_2(q)$, $q=2^{2j+1}\geq8$&$13q/8$&$64(q^2+1)(q-1)/169\geq29120/169>144$\\
${}^2G_2(q)$, $q=3^{2j+1}\geq27$&$3q/2$&$4q(q^3+1)(q-1)/9\geq6141408$\\
${}^2F_4(q)$, $q=2^{2j+1}\geq8$&$32q^2$&$q^{22}/2048\geq2^{55}$\\
\bottomrule
\end{tabularx}
\end{center}
For these three families, respectively, we use
$$ \begin{aligned} u&=q+\sqrt{2q}+1\leq13q/8,\\ u&=q+\sqrt{3q}+1<3q/2,\\ \text{and}\qquad u&\leq16\bigl(q^2+\sqrt{2q^3}+q+\sqrt{2q}+1\bigr)<32q^2. \end{aligned} $$
Their standard orders are $q^2(q^2+1)(q-1)$, $q^3(q^3+1)(q-1)$, and $n>q^{26}/2$, respectively. The Tits group has $(n,u)=(17971200,16)$, so $n/u^2=70200$.

Finally, the four smallest relevant sporadic groups have the following data:
\begin{center}
\begin{tabular}{@{}lrr@{}}
\toprule
$G$&$n$&$u$\\
\midrule
$M_{12}$&95040&11\\
$J_1$&175560&19\\
$M_{22}$&443520&11\\
$J_2$&604800&15\\
\bottomrule
\end{tabular}
\end{center}
Each satisfies $n>144u^2$. For the other twenty-one sporadic groups apart from $M_{11}$, the ATLAS bounds $n\geq10200960$ and $u\leq119$ imply
$$ 144u^2\leq2039184<10200960. $$

These comparisons prove the first assertion, using the standard identifications
$$ \PSL_3(2)\cong\PSL_2(7) \qquad\text{and}\qquad \PSL_4(2)\cong A_8, $$
$$ \PSp_4(2)'\cong A_6 \qquad\text{and}\qquad G_2(2)'\cong\PSU_3(3), $$
$$ {}^2G_2(3)'\cong\PSL_2(8) \qquad\text{and}\qquad \POmega_3(q)\cong\PSL_2(q), $$
$$ \POmega_5(q)\cong\PSp_4(q), $$
and
$$ \POmega_6^+(q)\cong\PSL_4(q) \qquad\text{and}\qquad \POmega_6^-(q)\cong\PSU_4(q). $$
The nonsimple group $\PSU_3(2)$ is excluded throughout.

\medskip
\emph{2. Uniform coverage in rank one.} For $\PSL_2(q)$, where $q=p^f\geq4$, let $\varepsilon=\gcd(2,q-1)$. Jordan form gives
$$ n=\frac{q(q^2-1)}{\varepsilon} \qquad\text{and}\qquad t\in\{p\}\cup\{d:d\mid(q-1)/\varepsilon\} \cup\{d:d\mid(q+1)/\varepsilon\}. $$
The Borel thresholds follow from Lemmas~\ref{lem:extraction-bounds} and~\ref{lem:frobenius-extraction}, as explained in Section~\ref{sec:fixed-point-free}. For unipotent elements over a nonprime field, the root subgroup $C_p^f$ gives the compatible-centralizer condition $2(f(p-1)+1)\leq p^f$. This holds for $f=2$ and $p\geq5$, and for $f=3$ and $p\geq2$, since $p^3-6p+4\geq0$. The induction uses
$$ p\bigl(f(p-1)+1\bigr) -\bigl((f+1)(p-1)+1\bigr) =f(p-1)^2\geq0. $$
For prime fields with $p\geq11$, Proposition~\ref{prop:prime-root-growth} applies instead.

For semisimple elements, the Borel and weighted-growth criteria are complementary in each of the following ranges: $q=p\geq47$ is prime; $q\geq32$ is even; or $q\geq49$ is odd and nonprime. For primes, the Borel criterion covers $t\leq8$, and weighted growth covers $t\geq9$. In the other two ranges, respectively, one has
$$ \frac{q}{4}\geq\max\{8,\sqrt{q}+1\} \qquad\text{and}\qquad \frac{q+1}{6}\geq\max\{8,\sqrt{q}+1\}. $$
These inequalities hold at $32$ and $49$, respectively, and strengthen thereafter. Any order not covered by Borel extraction therefore satisfies both $t\geq9$ and $t>\sqrt{q}+1$, as required for weighted growth.

For the primes $11\leq p<47$, exclude from the divisors of $(p-1)/2$ and $(p+1)/2$ all orders satisfying $t\geq9$ or $6t\leq p+1$. The remaining orders are exactly those in the following table:
\begin{center}
\begin{tabular}{@{}rl@{\qquad}rl@{}}
\toprule
$p$&Remaining $t$&$p$&Remaining $t$\\
\midrule
11&$3,5,6$&13&$3,6,7$\\
17&$4,8$&19&$5$\\
23&$6$&29&$7$\\
31&$8$&37&none\\
41&none&43&none\\
\bottomrule
\end{tabular}
\end{center}
Profile extraction, the subgroup bounds below, and Table~\ref{tab:remaining-pairs} account for every entry. For $p=5$, use $\PSL_2(5)\cong A_5$. For $p=7$, the nonidentity element orders are $2,3,4,7$.

The remaining nonprime prime powers are exactly
\begin{equation} 4,\ 8,\ 9,\ 16,\ 25,\ 27. \tag{15.4}\label{eq:psl2-small-fields} \end{equation}
Use $\PSL_2(4)\cong A_5$ and $\PSL_2(9)\cong A_6$. Compatible centralizers cover the other unipotent orders. Listing the divisors of the two torus orders, then applying Borel extraction and weighted growth, leaves only
$$ (8,3),\ (8,7),\ (8,9),\ (16,5),\ (25,6),\ (27,7). $$
By Corollary~\ref{cor:borel-divisor}, the pair $(16,5)$ is covered by the bound
$$ \mathsf{d}(\PSL_2(16)\times C_5)+1\leq1598<2044. $$
The pairs $(25,6)$ and $(27,7)$ are covered by profile extraction, and the three pairs with $q=8$ appear in Table~\ref{tab:remaining-pairs}.

\medskip
\emph{3. Subgroup extraction and the final local cases.} The exact values $\mathsf{d}(A_4)=4$ and $\mathsf{d}(A_5)=8$ are published in~\cite{AndrasCziszterDomokosSzollosi2025}. By the subgroup and normal-subgroup extraction lemmas, the following bounds hold:
\begin{equation} \begin{array}{ccl} G&\text{Orders covered}&\text{Upper bound for }\mathsf{d}(G\times C_t)+1\\ \hline A_5&t=2,3&9t\leq30+t-1\\ A_6&t=2,3&54t\leq180+t-1\\ \PSL_2(11)&t=2,3&99t\leq330+t-1\\ \PSL_2(7)&t=2&7(4+2\mathsf{d}(A_4))=84<85\\ \PSU_3(3)&t\mid24&2856<3024\\ \PSL_3(3)&t\mid12&2600<2808\\ \PSL_3(3)&t=8&2405<2808\\ M_{11}&t\mid6&3850<3960\\ M_{11}&t\mid8&3190<3960\\ M_{11}&t=5&3729<3960 \end{array} \tag{15.5}\label{eq:subgroup-coverage} \end{equation}
For the first three rows, use $A_5$, a subgroup $A_5$ of index six, and a subgroup $A_5$ of index eleven, respectively. For the fourth row, the relations $A_4\trianglelefteq S_4$ and $\mathsf{\eta}_2(C_2^2)=4$ yield
$$ \mathsf{d}(S_4\times C_2)+1\leq4+2\mathsf{d}(A_4). $$
Then use the subgroup $S_4$ of index seven in $\PSL_2(7)$. The last six rows follow from Propositions~\ref{prop:psl3-psu3-extraction} and~\ref{prop:m11-extraction}. Their element-order lists leave orders $13$, $7$, and $11$ in $\PSL_3(3)$, $\PSU_3(3)$, and $M_{11}$, respectively.

Finally, fixed-order extraction covers $t\leq3$ in $A_7$ and $t\leq7$ in $A_8$, since
$$ (7+3-2)150+1=1201<1260 $$
and
$$ (15+7-2)425+1=8501<10080. $$
The possible nonidentity orders in $A_7$ are $2,3,4,5,6,7$. Those in $A_8$ are $2,3,4,5,6,7,15$. The profile cases are collected in Table~\ref{tab:extraction-profiles}, and all other remaining cases appear in Table~\ref{tab:remaining-pairs}.

For the alternating groups, a symmetric-group conjugacy type determines the relevant cyclic-subgroup type. Conjugation by $S_d$ acts on $A_d$, so split $A_d$-classes cause no additional case. Semisimple cyclic subgroups of fixed order in rank one are conjugate under $\PGL_2(q)$. Changing a generator changes neither its orbits nor its double cosets. Sylow conjugacy covers the prime-order Singer, unitary, and Mathieu cases. Section~\ref{sec:psl2-7-t7-seed} gives an explicit normalization for the remaining unipotent order seven. Thus all local arguments apply to every repeated element of the specified order.

If $t=1$, then $\mathsf{d}(G)+1\leq m$ by Lemma~\ref{lem:davenport-bounds}(iii), since a non-abelian simple group has no cyclic subgroup of index two. Lemma~\ref{lem:high-mult-completion} then applies. The complementary bounded-height case is Theorem~\ref{thm:bounded-height}. This completes the coverage.
\end{proof}

\begin{table}[htbp]
\caption{The remaining pairs, arranged by proof method. In the rank-one rows, $(q,t)$ denotes $(\PSL_2(q),t)$.}\label{tab:remaining-pairs}
\small
\begin{tabularx}{\textwidth}{@{}L{0.25\textwidth}>{\raggedright\arraybackslash}XL{0.30\textwidth}@{}}
\toprule
Method & Pairs & Argument\\
\midrule
Explicit orbit products & $A_5,\ t=5;\ (7,7);\ (8,7)$ & Sections~\ref{sec:a5-involution-product}, \ref{sec:psl2-7-t7-seed}, and~\ref{sec:psl2-8-t7-seed}\\
& $\PSL_3(3),\ t=13$ & Section~\ref{sec:singer-trace-seed}\\
Orbit growth after a cyclic seed & $A_6,\ t=5;\ A_7,\ t=7$ & Propositions~\ref{prop:a6-order-five} and~\ref{prop:a7-order-seven}\\
& $(7,3),\ (7,4)$ & Propositions~\ref{prop:psl2-7-order-three} and~\ref{prop:psl2-7-order-four}\\
& $(11,5),\ (13,6)$ & Section~\ref{sec:split-torus-seeds}\\
Higher orbit powers & $(8,3),\ (8,9)$ & Proposition~\ref{prop:psl2-8-t3} and Section~\ref{sec:psl2-8-t9}\\
& $(11,6),\ (13,7),\ (17,8)$ & Proposition~\ref{prop:small-squared-orbits}\\
\bottomrule
\end{tabularx}
\end{table}

\subsection{The simple-group bound}\label{sec:simple-bound}

\begin{theorem}\label{thm:simple-bound}
Every finite non-abelian simple group $G$ satisfies
$$ \mathsf{E}(G)\leq3|G|/2-1. $$
\end{theorem}
\begin{proof}
Write $|G|=n=2m$, and let $S$ have length $3m-1$. If $h(S)\leq m$, then apply Theorem~\ref{thm:bounded-height}. If some element occurs at least $n$ times, then its $n$ copies have product one. Otherwise, the assumptions of~\eqref{eq:high-mult-setup} hold. The case $a=1$ follows from $\mathsf{d}(G)+1\leq m$ and Lemma~\ref{lem:high-mult-completion}(i).

For $a\neq1$, Proposition~\ref{prop:simple-coverage} assigns $(G,a)$ to one of the methods already proved. Subgroup or profile extraction yields $\mathsf{d}(G\times C_t)+1\leq m+t-1$, or the bounded subsequences of length divisible by $t$ required in Lemma~\ref{lem:high-mult-completion}(i). Compatible extraction proves the conclusion directly by Lemma~\ref{lem:centralizer-criterion} or Lemma~\ref{lem:interval-completion}. Orbit growth, with or without a cyclic seed, yields a subsequence $U$ of length at most $n$ such that $\pi(U)=G$, so Lemma~\ref{lem:high-mult-completion}(ii) applies. Thus $1\in\Pi_n(S)$ in every case, proving the theorem.
\end{proof}

\section{Proof of the main results}\label{sec:main-results}

\begin{proof}[Proof of Theorem~\ref{thm:main-bound}]
The solvable case follows from Corollary~\ref{cor:solvable-bound}. Suppose that $G$ is non-solvable. By Theorem~\ref{thm:simple-bound} and Proposition~\ref{prop:simple-reduction},
$$ \mathsf{E}(G)\leq\frac{3|G|}{2}-1. $$
By Lemma~\ref{lem:group-structure}(vii), one has $p=2$, so this is the required bound.
\end{proof}

\begin{proof}[Proof of Corollary~\ref{cor:prime-index-equality}]
Let $n=|G|$. If $G$ has a cyclic subgroup of index $p$, then, by Lemma~\ref{lem:gao-identities}(ii),
$$ \mathsf{E}(G)=n+\frac{n}{p}+p-2. $$
Otherwise, by Theorem~\ref{thm:main-bound},
$$ \mathsf{E}(G)\leq n+\frac{n}{p}-1<n+\frac{n}{p}+p-2, $$
because $p-1>0$. This proves the bound and its equality condition. 

If $p=2$, then the first bound becomes $\mathsf{E}(G)\leq 3n/2$, with equality if and only if $G$ has a cyclic subgroup of index $2$. If $p\geq3$, then the non-cyclicity gives $n\geq p^2$, whence
$$ \frac{n}{2}-\frac{n}{p}-p+2\geq \frac{p^2}{2}-2p+2=\frac{(p-2)^2}{2}>0. $$
Hence $\mathsf{E}(G)\leq n +n/p +p-2<3n/2$ in this case. This proves both conjectures and the stated equality conditions.
\end{proof}

\noindent {\bf Acknowledgements.} This work was supported by the National Natural Science Foundation of China (No. 11701256), and the Henan Provincial Selective Research Funding Program for Returned Scholars Studying Abroad (No. HNLX202611).

\noindent{\bf AI usage.} AI tools were used to assist with literature searches, proofs in the sections on finite simple groups, and language editing.

\end{document}